\PassOptionsToPackage{unicode}{hyperref}
\PassOptionsToPackage{hyphens}{url}
\PassOptionsToPackage{dvipsnames,svgnames,x11names}{xcolor}
\documentclass[
  12pt]{article}

\usepackage{amsmath,amssymb,amsthm, amsbsy}
\usepackage{iftex}
\ifPDFTeX
  \usepackage[T1]{fontenc}
  \usepackage[utf8]{inputenc}
  \usepackage{textcomp} 
\else 
  \usepackage{unicode-math}
  \defaultfontfeatures{Scale=MatchLowercase}
  \defaultfontfeatures[\rmfamily]{Ligatures=TeX,Scale=1}
\fi
\usepackage{lmodern}
\ifPDFTeX\else  
\fi
\IfFileExists{upquote.sty}{\usepackage{upquote}}{}
\IfFileExists{microtype.sty}{
  \usepackage[]{microtype}
  \UseMicrotypeSet[protrusion]{basicmath} 
}{}
\makeatletter
\@ifundefined{KOMAClassName}{
  \IfFileExists{parskip.sty}{%
    \usepackage{parskip}
  }{
    \setlength{\parindent}{0pt}
    \setlength{\parskip}{6pt plus 2pt minus 1pt}}
}{
  \KOMAoptions{parskip=half}}
\makeatother
\usepackage{xcolor}
\makeatletter
\ifx\paragraph\undefined\else
  \let\oldparagraph\paragraph
  \renewcommand{\paragraph}{
    \@ifstar
      \xxxParagraphStar
      \xxxParagraphNoStar
  }
  \newcommand{\xxxParagraphStar}[1]{\oldparagraph*{#1}\mbox{}}
  \newcommand{\xxxParagraphNoStar}[1]{\oldparagraph{#1}\mbox{}}
\fi
\ifx\subparagraph\undefined\else
  \let\oldsubparagraph\subparagraph
  \renewcommand{\subparagraph}{
    \@ifstar
      \xxxSubParagraphStar
      \xxxSubParagraphNoStar
  }
  \newcommand{\xxxSubParagraphStar}[1]{\oldsubparagraph*{#1}\mbox{}}
  \newcommand{\xxxSubParagraphNoStar}[1]{\oldsubparagraph{#1}\mbox{}}
\fi
\makeatother

\providecommand{\tightlist}{%
  \setlength{\itemsep}{0pt}\setlength{\parskip}{0pt}}\usepackage{longtable,booktabs,array}
\usepackage{calc} 
\usepackage{etoolbox}
\makeatletter
\patchcmd\longtable{\par}{\if@noskipsec\mbox{}\fi\par}{}{}
\makeatother
\IfFileExists{footnotehyper.sty}{\usepackage{footnotehyper}}{\usepackage{footnote}}
\makesavenoteenv{longtable}
\usepackage{graphicx}
\makeatletter
\def\maxwidth{\ifdim\Gin@nat@width>\linewidth\linewidth\else\Gin@nat@width\fi}
\def\maxheight{\ifdim\Gin@nat@height>\textheight\textheight\else\Gin@nat@height\fi}
\makeatother
\setkeys{Gin}{width=\maxwidth,height=\maxheight,keepaspectratio}
\makeatletter
\def\fps@figure{htbp}
\makeatother

\makeatletter
\@ifpackageloaded{caption}{}{\usepackage{caption}}
\AtBeginDocument{%
\ifdefined\contentsname
  \renewcommand*\contentsname{Table of contents}
\else
  \newcommand\contentsname{Table of contents}
\fi
\ifdefined\listfigurename
  \renewcommand*\listfigurename{List of Figures}
\else
  \newcommand\listfigurename{List of Figures}
\fi
\ifdefined\listtablename
  \renewcommand*\listtablename{List of Tables}
\else
  \newcommand\listtablename{List of Tables}
\fi
\ifdefined\figurename
  \renewcommand*\figurename{Figure}
\else
  \newcommand\figurename{Figure}
\fi
\ifdefined\tablename
  \renewcommand*\tablename{Table}
\else
  \newcommand\tablename{Table}
\fi
}
\@ifpackageloaded{float}{}{\usepackage{float}}
\floatstyle{ruled}
\@ifundefined{c@chapter}{\newfloat{codelisting}{h}{lop}}{\newfloat{codelisting}{h}{lop}[chapter]}
\floatname{codelisting}{Listing}

\makeatother
\makeatletter
\@ifpackageloaded{caption}{}{\usepackage{caption}}
\@ifpackageloaded{subcaption}{}{\usepackage{subcaption}}
\makeatother

\ifLuaTeX
  \usepackage{selnolig}  
\fi
\usepackage[]{natbib}
\usepackage{bookmark}

\IfFileExists{xurl.sty}{\usepackage{xurl}}{} 
\hypersetup{
  pdftitle={Title},
  pdfauthor={Author 1; Author 2},
  pdfkeywords={3 to 6 keywords, that do not appear in the title},
  colorlinks=true,
  linkcolor={blue},
  filecolor={Maroon},
  citecolor={Blue},
  urlcolor={Blue},
  pdfcreator={LaTeX via pandoc}}

\newcommand{\anon}{1}

\newcommand{\pch}[1]{{{\color{purple}  #1}}}

\newcommand{\E}{\mathbb{E}}
\newcommand{\N}{\mathbb{N}}
\newcommand{\R}{\mathbb{R}}
\newcommand{\I}{\mathbb{I}}
    \newcommand{\M}{\mathcal{M}}
\newcommand{\T}{\mathcal{T}}
\newcommand{\OO}{\mathcal{O}}

\newcommand{\bx}{\mathbf{x}}
\newcommand{\by}{\mathbf{y}}
\newcommand{\bz}{\boldsymbol{z}}
\newcommand{\bu}{\mathbf{u}}
\newcommand{\bw}{\mathbf{w}}

\newcommand{\bbeta}{\boldsymbol{\beta}}

\newcommand{\btheta}{\boldsymbol{\theta}}
\newcommand{\bepsilon}{\mathbf{\epsilon}}

\newcommand{\bX}{\mathbf{X}}

\newcommand{\bB}{\mathfrak{B}}
\newcommand{\bkappa}{\boldsymbol{\kappa}}

\newtheorem{thm}{\textsc{Theorem}}[section]
\newtheorem{prop}[thm]{\textsc{Proposition}}
\newtheorem{cor}[thm]{\textsc{Corollary}}
\newtheorem{lem}[thm]{\textsc{Lemma}}

\newcolumntype{C}[1]{>{\centering\let\newline\\\arraybackslash\hspace{0pt}}m{#1}}
\newcolumntype{M}[1]{>{\centering\arraybackslash}m{#1}}

\usepackage{enumitem}

\begin{document}

\def\spacingset#1{\renewcommand{\baselinestretch}%
{#1}\small\normalsize} \spacingset{1}


\if1\anon
{
  \title{\bf Improving variable selection properties with data integration and transfer learning}
  \author{Paul Rognon-Vael\thanks{
    The authors gratefully acknowledge the support of the ERC Grant BigBayesUQ project number: 101041064, grant Consolidaci\'on investigadora CNS2022-135963 by the AEI, PID2022-138268NB-I00 by MCIN/AEI/10.13039/501100011033 /FEDER, ICREA Academia fellowship from AGAUR (Generalitat de Catalunya), and NSERC grant RGPIN-2023-03481.}\hspace{.2cm}\\
    BIDSA, Bocconi University\\
    and \\
    David Rossell \\
    Department of Economics and Business, Universitat Pompeu Fabra \\
    and \\
    Piotr Zwiernik \\
    Department of Economics and Business, Universitat Pompeu Fabra}
  \maketitle
} \fi

\if0\anon
{
  \bigskip
  \bigskip
  \bigskip
  \begin{center}
    {\LARGE\bf Title}
\end{center}
  \medskip
} \fi

\bigskip
\begin{abstract}
We study variable selection (also called support recovery) in high-dimensional sparse linear regression when one has external information on which variables are likely to be associated with the response.
Consistent recovery is only possible under somewhat restrictive conditions on sample size, dimension, signal strength, and sparsity.
We investigate how these conditions can be relaxed by incorporating said external information. A key application that we consider is structural transfer learning, where variables selected in one or more source datasets are used to guide variable selection in a target dataset.
We introduce a family of likelihood penalties that depend on the external information, motivated by connections to Bayesian variable selection. We show that these methods achieve variable selection consistency in regimes where any method ignoring external information fails, and that they achieve consistency at faster rates. We first quantify the potential gains under ideal, oracle-chosen, penalties. We then propose computationally efficient empirical Bayes procedures that learn suitable penalties from the data. We prove that these procedures have improved variable selection properties compared to methods that do not use external information. We illustrate our approach using simulations and a genomics application, where results from mouse experiments are used to inform variable selection for gene expression data in humans.
\end{abstract}

\noindent%
{\it Keywords:} high-dimensional statistics, Bayesian variable selection, $\ell_0$ penalty, empirical Bayes, transfer learning.
\vfill

\newpage
\spacingset{1.8} 

\maketitle

\section{Introduction}

Many contemporary statistical problems involve high-dimensional models, where the goal is to identify a small subset of variables associated with a response.
In such settings, inference is often severely limited by the dataset sample size.
At the same time, basing inference on a single dataset increasingly feels like a wasted opportunity: it is common to have additional external information—such as prior studies, auxiliary datasets, or variable-specific annotations—that can suggest differential relevance across variables. This situation arises, for example, in multimodal data \citep{liang2023multimodal} where different groups of variables may exhibit different levels of sparsity, such as clinical versus genomic markers in biomedical applications.
It also arises when metadata are available on the variables themselves, for instance through functional annotations of genes.

A setting of particular interest is transfer learning, which broadly refers to improving performance on a target dataset by leveraging information from related source datasets.
In the context of variable selection, transfer learning uses information about variables identified in source datasets as external input to guide variable selection in a target regression problem.
We refer to this problem as \emph{structural transfer learning}, to distinguish it from approaches that focus on transferring parameter estimates or predictions.

In high-dimensional settings, leveraging external information is particularly attractive.
From a theoretical perspective, the limited information provided by a single dataset implies that strong assumptions are required for accurate inference. In structural learning problems such as variable selection, those assumptions regard sparsity and signal strength and, although mathematically necessary, they can be too strong in practice 
(see~\cite{giannone2021economic} for a discussion in the context of social sciences). From a practical perspective, numerous works observed improved inference when employing external information to guide structural learning \citep{stingo:2011, Boulesteix2017, chen_tinghuei:2021}. 

Despite this empirical success, a theoretical framework explaining precisely why and how leveraging external information may deliver improved structural learning is not currently available. 
Our main contribution is investigating how, in linear regression, integrating external information (and structural transfer learning as a particular case) would allow an oracle to push the mathematical conditions under which consistent variable selection is possible, and improving the corresponding rates. Consistency is understood as recovering the truly active variables with probability going to 1 as the sample size $n$, and possibly the dimension $p$, go to infinity. 
Our other main contribution is proposing an empirical Bayes data-based procedure that doesn't require oracle information, and showing that it also enjoys improved properties.

To make ideas concrete, consider the Gaussian linear regression
\begin{equation}\label{eq:linearmodel}
	\by \;=\; \bX \bbeta^* + \bepsilon,\qquad \bepsilon \sim N(0,\sigma^2 I_n),
\end{equation}
where $\bX\in \R^{n\times p}$, $\by\in \R^n$, $\sigma >0$, and $\bbeta^* \in \mathbb{R}^p$ are the data-generating parameters. Our results can be easily  extended to non-linear regression, where $\bX$ is given by a suitable basis such as splines. 
Penalized likelihood methods for variable selection are often based on optimizing the log-likelihood plus a penalty term driven by the $\ell_q$ "norm" of an estimated $\hat{\bbeta}$ for $q\in [0,1]$. 
Here we focus on $\ell_0$ penalties because they possess theoretically superior properties for support recovery compared to other penalties~\citep{infotheowainwright, gao2025optimalityl0}. That is, our goal is to study how external information can improve upon these best known properties. 

Besides, computational advances have made $\ell_0$ problems much more tractable: optimization methods can solve them exactly for $p$ in the thousands~\citep{bertsimas2020sparse} and Markov Chain Monte Carlo (MCMC) methods have linear cost in $p$ under sparsity conditions~\citep{Yang2016,zhou2022}. There is also a direct connection between $\ell_0$ penalties and Bayesian variable selection ~\citep{BIC,EBIC,rossell2022concentration}. Under mild regularity conditions, a choice of $\ell_0$ penalty corresponds to a choice of prior probability for variable inclusion. This connection converts $\ell_0$ penalization into a natural and intuitive way to encode external information on the relevance of variables for association with the outcome. 

We consider a variable selection framework where $\ell_0$ penalties may depend on external information $\bz=(z_1,\ldots,z_p) \in \R^p$. In transfer learning applications, each $z_j$ may be a $q$-dimensional vector measuring the importance of variable $j$ in $q$ source datasets. For theoretical tractability and to be able to provide sharp conditions and rates, here we focus on the case where $\bz$ partitions the $p$ variables into $b$ blocks, that is $z_j \in \{1,\ldots,b\}$.
We denote by $B_j \subset \{1,\ldots,p\}$ the set of variables in block$j=1,\ldots,b$.
When $\bz$ is informative for variable selection, key characteristics such as the proportion of truly active variables or the signal strength may vary across blocks, hence one allows the $\ell_0$ penalty to differ across blocks. 
We derive the variable selection properties of such informed $\ell_0$ penalties, first in general and then for structural transfer learning, and compare them to standard $\ell_0$ penalties.

Before proceeding, we review selected literature.
External information has been used to guide variable selection in many applications. For instance,~\cite{stingo:2011} and~\cite{cassese:2014} proposed Bayesian variable selection for gene expression where prior inclusion probabilities depend on biological knowledge and meta-variables.~\cite{chen_tinghuei:2021} predicted disease outcomes with $\ell_1$ penalties that depend on gene functional annotations. Several works discussed methodology, but not theory, for external-information dependent penalization or prior inclusion probabilities. In the Bayesian literature,~\cite{vandeWiel2019} suggested an empirical Bayes approach to regress prior inclusion probabilities on external data.~\cite{Velten2021} proposed a variational Bayes approximation for regression with external information-dependent prior inclusion probabilities. Recently,~\cite{busatto2023} considered a horseshoe prior that depends on external information. In the frequentist literature, in the context of multi-omics data,~\cite{Boulesteix2017} let the $\ell_1$ penalty depend on the modality of each variable.~\cite{Zeng2020} and~\cite{vanNee2023} proposed empirical Bayes to set externally-informed $\ell_1$ and elastic-net penalties respectively. See~\cite{vandewiel2024} for a review.

Previous theory on transfer learning in  regression has focused on parameter estimation, but provided no results on identifying the set of active variables as we do here.~\cite{Jiang2016} and~\cite{Li2021} showed improvements in convergence rates in prediction and estimation with $\ell_1$ penalties that depend on the discrepancy with source estimates. In a Bayesian setting,~\cite{lai2024} showed improvements in posterior contraction rates of parameters with horseshoe priors centered at a weighted average of source-specific estimates.

The paper is organized as follows.
In Section~\ref{sec:intro_blockpen}, we define externally-informed $\ell_0$ penalties and discuss their connection to Bayesian variable selection.
Section~\ref{sec:properties_legression} studies their theoretical properties, including oracle results characterizing the maximum gains achievable by informed $\ell_0$ penalties. In Section~\ref{sec:informed_l0}, we propose a data-driven empirical Bayes procedure for setting informed $\ell_0$ penalties.
Section~\ref{sec:TL} applies this procedure to transfer learning and establishes improved selection properties, including robustness to negative transfer in settings where the external data is uninformative regarding what variables are truly active, and to covariate shifts.
Section~\ref{sec:examplesinforml0} shows simulations and a gene expression study transferring findings from mice to humans.

\textbf{Notation:}
The parameters of interest in the target regression are denoted by $\bbeta\in \R^p$  
and $\bbeta^*$ denotes their true values. The set of variable indices is $V=\{1,\ldots,p\}$. For any $A\subseteq V$ and any vector $\bx\in \R^p$, $\bx_A$ denotes the subvector of $\bx$ with entries corresponding to indices in $A$. For any matrix $\bX \in \R^{n\times p}$, $\bX_A$ denotes the submatrix of $\bX$ obtained by selecting the \emph{columns} with indices in $A$. $S=\left\{i\in V:\,\,  \beta^*_i \neq 0 \right\}$ denotes the true support of $\bbeta^*$, its size is $s=|S|$, and hence $p-s$ is the number of truly inactive parameters. The set of candidate models is denoted by $\M$ and is given by subsets $M\subseteq V$. The set $V$ is partitioned into disjoint blocks $B_j \subseteq V$ for $j=1,\ldots,b$, where $b$ is fixed. $S_j=\left\{i\in B_j:\,\,  \beta^*_i \neq 0\right\}$ denotes the set of active parameters in block $B_j$, $s_j=|S_j|$ its size, $p_j-s_j$ the number of inactive parameters in the block. We denote by $\beta_{\min,j}^*$ and $\beta_{\min}^*$ lower bounds on the smallest active signal in $S_j$ and $S$ respectively, such that
the data-generating $\bbeta^*$ belongs to the set $\mathfrak{B}=\{\bbeta \in \R^p: \text{for all }j \in \{1,\ldots,b\} \text{ and }i \in B_j,\, |\beta_i|\geq  \beta^*_{\min,j} \text{ if } i\in S_j \text{ and }  \beta_i= 0 \text{ otherwise}\}$.

Given sequences $f(n) > 0$ and $g(n) >0$, $f(n)=O(g(n))$ means that there exists a constant $c<\infty$ such that $f(n) \leq c g(n)$ for all $n \geq n_0$ and some fixed $n_0$, $f(n)=o(g(n))$ means that $\lim_{n\to\infty}f(n)/g(n)=0$, and $f(n)\asymp g(n)$ means that $f(n) = O(g(n))$ and $g(n) = O(f(n))$. We write $f(n) \gtrsim g(n)$ when there exists a constant $c>1$ and a sequence $d(n)\to\infty$ such that $f(n) \geq c g(n)+d(n)$. For any set $A$, $A^C$ denotes the complement of $A$.

\section{Externally-informed penalties and transfer learning}\label{sec:intro_blockpen} 

We describe externally-informed $\ell_0$ penalization and their connection to Bayesian variable selection (Section~\ref{ssec:intro_blockpen}), and general assumptions used across our results (Section~\ref{ssec:genassumptions}).

\subsection{Externally-informed penalization}\label{ssec:intro_blockpen}

We recall the $\ell_0$-penalized approach to variable selection. Consider a set of candidate models $\M$ given by subsets $M\subseteq V$ such that $|M|\leq n$ and their parameter spaces
\[
\mathcal{L}_M := \{\beta\in \R^p: \beta_i=0 \text{ if } i\notin M\}.
\]
A standard $\ell_0$ procedure with penalty $\eta(M)=\kappa |M|$ for some $\kappa>0$ selects the model
\begin{equation}\label{eq:stdl0}
\hat{S}^{sd}= \arg\max_{M\in \M} \left\{\max_{\beta\in \mathcal{L}_M}\ell(\by; \bbeta)-\kappa |M|\right\},
\end{equation}
where $\ell(\by; \bbeta)$ denotes the log-likelihood function. 
Both $p$ and the number of models $|\mathcal{M}|$ may far exceed $n$, and we may only consider models of size $|M| \leq n$ without loss of generality (if $\kappa>0$ such models are never selected).

A useful way to interpret $\ell_0$ penalization is via its connection to Bayesian variable selection under discrete spike-and-slab priors where, for each variable $i$, one has a binary indicator $m_i=I(\beta_i\,\neq\,0)$ for the inclusion of variable $i$ in the model. Each model $M \in \M$ corresponds then to a particular value of $(m_1,\ldots,m_p)$. The joint prior on parameters and models is 
\begin{equation}\label{eq:bayesianframework}
    p(\bbeta,M \mid \theta)\;=\;p( \bbeta \mid  M) p(M \mid \theta)
\end{equation}
where $p( \bbeta \mid  M)$ is a prior on $\bbeta$ given model $M$, and 
variable inclusions are independent a priori with inclusion probability $\theta$, that is
\begin{equation} \label{eq:modelprior}
p( M \mid \theta) \propto
        \prod_{i=1}^p \operatorname{Bern}\left(m_i ; \theta\right).
\end{equation}
Posterior model probabilities are $p(M \mid \by, \theta) \propto p(\by \mid M) p(M \mid \theta)$, where $p(\by \mid M)$ is the so-called marginal likelihood of model $M$. The BIC approximation~\citep{BIC} to $p(M\mid \by, \theta)$ gives that, for a wide family of priors $p(\bbeta \mid M)$, 
\begin{equation}\label{eq:dfbapon}
\ln p(M \mid\by, \theta) \;\approx\; \ln p(\by \mid \tilde{\bbeta}^{(M)}) - \Big(\frac{1}{2}\ln(n) + \ln \big(1/\theta-1\big)\Big)|M| + c_M,
\end{equation}
where $\tilde{\bbeta}^{(M)}$ is the maximum likelihood estimator (MLE) under model $M$, and $c_M$ is a constant that may depend on $M$.
This approximation makes explicit the close relationship between Bayesian model selection and $\ell_0$ penalties. Neglecting $c_M$ and retaining only terms that grow with $n$, simple algebra shows that the $\ell_0$ selector in~\eqref{eq:stdl0} with $\kappa = \frac12 \ln(n) + \ln(1/\theta - 1)$ approximately maximizes $p(M \mid \by, \theta)$. This gives a correspondence between a choice of prior inclusion probability $\theta$ and a choice of $\ell_0$ penalty $\kappa$. 

This Bayesian perspective also suggests how external information can be incorporated into variable selection. Assume that one has external information $\bz=(z_1,\ldots,z_p)$ on the relevance of variables for association with $\by$. In a spike-and-slab prior, this information can be naturally reflected by setting prior inclusion probabilities that depend on $\bz$, that is replacing $\theta$ in~\eqref{eq:modelprior} by $\theta(z_i)$, $i=1,\ldots,p$. When $\bz$ is discrete, it defines blocks of variables that share similar prior relevance. In this work, we assume that $\bz$ partitions the variables into blocks $B_1,\ldots,B_b$, 
obtaining block-dependent prior probabilities $\theta(z_i)= \theta_j$ for all $i \in B_j$,
 \begin{eqnarray} \label{eq:block_modelprior}
 p( M \mid \btheta) \propto
         \prod_{i=1}^p \operatorname{Bern}\left(m_i ; \theta_j\right) I(i \in B_j),
\end{eqnarray}
The BIC approximation for this model prior gives
\begin{eqnarray}\label{eq:block_modelprob}
\ln p(M \mid\by, \mathbf{\theta}) \;\approx\; \ln p(\by \mid \tilde{\bbeta}^{(M)}) - \sum_{j=1}^b\Big(\frac{1}{2}\ln(n) +\ln \big(1/\theta_j-1\big)\Big)|M_j| + c_M,
\end{eqnarray}
where $M_j \subseteq B_j$ are
the selected variables in block $j$ and $M=\cup_{j=1}^b M_j$ the model.

Motivated by this correspondence, we now return to a frequentist perspective. We study in this work the $\ell_0$ selector analogue to maximizing $p(M \mid \by, \mathbf{\theta})$ in~\eqref{eq:block_modelprob}. That is, an informed $\ell_0$ selector that penalizes differently the blocks $B_1,\ldots,B_b$,
\begin{equation}\label{eq:Mhat}
  \hat S^{ei}\;\in\;\arg\max_{M\in \M} \left\{\max_{\bbeta\in \mathcal{L}_M}\ell(\by; \bbeta)-\sum_{j=1}^b\kappa_j |M_j| \right\},
\end{equation}
where $\kappa_1>0,\ldots,\kappa_b>0$. Under the BIC approximation in~\eqref{eq:block_modelprob}, $\hat S^{ei}$ corresponds to the posterior mode with block prior inclusion probabilities $\theta_j=(\exp(\kappa_j-\ln(n)/2)+1)^{-1}$. In the case that $b=1$ (all variables form a single block), $\hat S^{ei}$ recovers $\hat{S}^{sd}$ in \eqref{eq:stdl0}. Throughout, the superscripts “sd” and “ei” stand for “standard” and “externally-informed,” respectively.

\subsection{General assumptions of theoretical results}\label{ssec:genassumptions}

We study the properties of $\hat{S}^{ei}$ as the sample size $n \to \infty$.
Although not explicitly denoted, $p$, $s$ and $\beta_{\min}^*$ are functions of $n$, and so are $s_j$, $p_j$ and $\beta_{\min,j}^*$. In particular the numbers of inactive and active variables ($p_j-s_j$ and $s_j$) can be either fixed or grow with $n$. We make no assumption about the regime linking $n$ and $p$, but the main interest is in $n= o(p)$ settings. For convenience, we assume that in each block $j$, $p_j-s_j\geq1$ and $s_j\geq1$, but our results can easily be adapted to $p_j-s_j=0$ and $s_j=0$.

We assume that $\M$ is a well-specified set of candidate models, that is $S$, the support of $\bbeta^*$, lies in $\M$, and that one constrains attention to models with full-rank design, such that for any model $M \in \M$, $\operatorname{rank}(\bX_M^{\top}\bX_M)=|M|$. This assumption implies $s\pch{=|S|}\leq n$. We also assume that $\hat S^{ei}$ is unique and that the error variance $\sigma^2$ is known. To simplify notation, without loss of generality, we set $\sigma^2=1$. Non-uniqueness of $\hat{S}^{ei}$, non-full rank models, and unknown $\sigma^2$ can be accommodated, at the expense of a slightly more involved treatment.

\section{Oracle externally-informed penalties}\label{sec:properties_legression}

We discuss the properties of the externally-informed model selector $\hat S^{ei}$ in \eqref{eq:Mhat} in the linear model~\eqref{eq:linearmodel}. Our goal is to show that $P(\hat S^{ei}=S)\to 1$ as $n\to\infty$, and that this occurs under milder conditions or at a faster rate than for the standard selector $\hat S^{sd}$ in \eqref{eq:stdl0}, which uses a common penalty across blocks. In Section~\ref{ssec:oracleprop}, we prove the consistency of $\hat S^{ei}$ under sufficient conditions. In Section \ref{ssec:necessary_conditions} we discuss matching necessary conditions that are important not only to ensure the tightness of our analysis, but also to describe the benefits of $\hat S^{ei}$ over $\hat{S}^{sd}$. 
In Section~\ref{ssec:regression_gains}, we discuss these benefits in relation to the relevance of the external information, that is to what extent the blocks discriminate between truly active and inactive signals. In this section we assume that penalties are set by an oracle that knows the true pattern of sparsity and signal strength. In Section~\ref{ssec:nobetaminhyp}, we establish a generalized convergence result for $\hat S^{ei}$ under weak conditions on signal strength, which plays an important role for the data-based procedures in Section~\ref{sec:informed_l0}.

\subsection{Oracle properties}\label{ssec:oracleprop}

We first outline our proof strategy. 
We remark that, although the strategy builds upon that in \cite{rossell2022concentration}, there are significant innovations. \cite{rossell2022concentration} did not consider block penalties and, when applied to the single penalty case, our sufficient conditions to attain consistency are milder and the associated rates are sharper. We also obtain new necessary conditions (Section \ref{ssec:necessary_conditions}) that nearly match the sufficient conditions. 
We also compared our results to those obtained under a Gaussian sequence model, a simplified setting where a sharp characterization is possible, obtaining analogous results to our regression setting. Finally, the proof and results in Section \ref{ssec:nobetaminhyp} are new, and these are crucial for the theory of our data-based procedure in Section \ref{sec:informed_l0}.

For any model $M\in \M$, let $\tilde{\bbeta}^{(M)}= (\bX_{\!\!M}^\top \bX_{\!\!M})^{-1}\bX_{\!\!M}^\top\by \in \R^{|M|}$ be the MLE under model $M$, and 
\begin{equation}\label{eq:CM}
    C(M)\;: =\;\tfrac{1}{2}\|\bX_M\tilde{\bbeta}^{(M)}\|^2-\sum_{j=1}^b\kappa_j |M_j|\quad\text{and}\quad NC(M)\;:=\;\frac{\exp(C(M))}{\underset{M' \in \M}{\sum} \exp(C(M'))}.
\end{equation}
\begin{lem}\label{lem:usefulreformulation} 
$\hat{S}^{ei}$ satisfies $\hat{S}^{ei}\;=\;\arg\max_{M\in \M} NC(M)$ 
\end{lem}

By Lemma~\ref{lem:usefulreformulation}, $\hat{S}^{ei}$ selects
$M\in \M$ with the largest $NC(M)$, which relates to the sum-of-squares explained by $M$. By the BIC approximation in~\eqref{eq:block_modelprob}, $C(M)$ and $NC(M)$ can be understood as, respectively, the unnormalized and normalized pseudo-posterior probability for model $M$ in a Bayesian framework. 
We refer to $C(M)$ and $NC(M)$ as the unnormalized and normalized scores for $M$, respectively.
Lemma~\ref{lem:L0toL1convergence}, reproduced from~\cite{rossell2022concentration}, proves that the expected sum of the normalized scores for $M\neq S$ bounds the probability of an incorrect model selection. 
\begin{lem}\label{lem:L0toL1convergence}
\begin{enumerate}[label= (\roman*)]\tightlist
\item
$P(\hat{S}^{ei} \neq S) \,\,\leq \,\, 2 \sum_{M \in \M\setminus \{S\}} \E\left(NC(M)\right)$.
\item  For any $M,M'\in\M$, such that $M\neq M'$, 
$NC(M)\;\leq \;\big(1+\exp(C(M')-C(M))\big)^{-1}$.
\end{enumerate}
\end{lem}
Using Lemma~\ref{lem:L0toL1convergence}(i), we control $P(\hat{S}^{ei} \neq S)$ by showing that
$\sum_{M\neq S}\E\{NC(M)\}$ vanishes asymptotically.
That is, we establish strong consistency in the sense that $NC(S) \stackrel{L_1}{\longrightarrow} 1$.
By Lemma~\ref{lem:L0toL1convergence}(ii) with $M'=S$, it suffices to control
$\E\!\left\{(1+\exp(C(S)-C(M)))^{-1}\right\}$, which reduces to bounding quadratic forms
of least-squares estimators.
We obtain these bounds using concentration inequalities and then summing over $M\neq S$.

We now state two assumptions that are sufficient to asymptotically recover $S$.

\hypertarget{lab:A1}{(A1)} \quad For each $j$, there exists $f_j\to \infty$ such that, 
\begin{equation*}
    \kappa_j \;=\;\ln(p_j-s_j) + f_j
\end{equation*} 
\hypertarget{lab:A2}{(A2)} \quad For each $j$, there exists $g_j\to \infty$ such that,
    \begin{equation*}
        \sqrt{\frac{(1-\gamma) n\rho(\bX)}{6}}{\beta_{\min,j}^*} \,- \,\sqrt{\kappa_j} \;=\; \sqrt{\ln(s_j)} + g_j,
    \end{equation*}
where $\rho(\bX)=\min_{M \in \M:\, M\not \supseteq S} \lambda_{\min}\big({n}^{-1} \bX_{S \setminus M}^{\top}\left(I_n- \bX_{M}\left(\bX_{M}^{\top} \bX_{M}\right)^{-1} \bX_{M}^{\top}\right) \bX_{S \setminus M}\big)$, and $\gamma:=\tfrac12(1+\max_j\ln(p_j-s_j)/\kappa_j)\in (\tfrac12, 1)$.

Assumption~\hyperlink{lab:A1}{A1} sets a lower-bound the block penalties $\kappa_j$, whereas Assumption~\hyperlink{lab:A2}{A2} imposes a beta-min condition on the block-wise signal strengths $\beta_{\min,j}^*$.
In the simplest orthonormal setting where $\bX^{\top}\bX=n\I_p$, the quantity $\rho(\bX)=1$. More generally, 
$\rho(\bX)\ge 0$ measures the correlation between truly active variables $\bX_S$ and other variable sets $\bX_M$. For further discussion see, e.g.,~\cite{infotheowainwright}.
Our proofs allow slightly weaker conditions, but we present~\hyperlink{lab:A1}{A1}--\hyperlink{lab:A2}{A2} in this form to facilitate interpretation and comparison with existing results.

We next state our main theorems on strong variable selection consistency of $\hat{S}^{ei}$ and the associated rates. In Theorem~\ref{prop:fullrecovreg}, the constants $\delta\in(0,1)$ and $\gamma^*\in(1/2,1)$ should be taken arbitrarily close to 1 and $1/2$, respectively. After presenting these results, we examine the tightness of Assumptions \hyperlink{lab:A1}{A1}–\hyperlink{lab:A2}{A2} relative to necessary conditions for selection consistency.

\begin{thm}\label{theo:suffcondlinearmodel}

Assume~\hyperlink{lab:A1}{A1} and \hyperlink{lab:A2}{A2}, and recall that $\mathfrak{B}=\{\bbeta \in \R^p: \text{for all }j \in \{1,\ldots,b\} \text{ and }i \in B_j,\, |\beta_i|\geq  \beta^*_{\min,j} \text{ if } i\in S_j \text{ and }  \beta_i= 0 \text{ otherwise}\}$. Then,
\[\lim_{n \to \infty}\, \sup_{\bbeta^*\in\mathfrak{B}}\sum_{M \in \M\setminus \{S\}}\E\left(NC(M)\right)= 0 \qquad \mbox{and} \qquad \lim_{n \to \infty}\, \inf_{\bbeta^*\in\mathfrak{B}} P(\hat{S}^{ei} = S)= 1.\]
\end{thm}

Theorem \ref{theo:suffcondlinearmodel} guarantees that the correct selection probability $P(\hat{S}^{ei} = S)$ converges to 1, uniformly across the set of $(s_1,\ldots,s_b)$-sparse $\bbeta^* \in \mathfrak{B}$.

\begin{thm}\label{prop:fullrecovreg}
Assume~\hyperlink{lab:A1}{A1} and \hyperlink{lab:A2}{A2}.
Then, there exists a constant $c>0$ such that for all sufficiently large $n$ and any $\delta \in (0,1)$,
\begin{equation}\label{eq:rateconvreg}
\sup_{\bbeta^*\in\mathfrak{B}}P(\hat{S}^{ei} \neq S)\leq c\sum_{j=1}^b\exp\big[-\tfrac{\delta}{2}(\kappa_j-\ln(p_j-s_j))\big] + 
   \exp\Big[- \tfrac{\delta}{2}\Big(\Big(\sqrt{\tfrac{(1-\gamma) n\rho(\bX)\beta_{\min,j}^*}{6}} - \sqrt{\kappa_j}\Big)^2 -\ln(s_j)\Big) \Big].
\end{equation}
Moreover, suppose that for some fixed $\gamma^*\in(1/2,1)$ and all $j=1,\ldots,b$, the $\kappa_j$ are set at the oracle values $\kappa^*_j$ satisfying
	\begin{equation} \label{eq:oraclepenreg}
\sqrt{\kappa^*_j\,} \;=\; \frac12\sqrt{\frac{(1-\gamma^*) n\rho(\bX)}{6}}{\beta_{\min,j}^*}+\frac12\sqrt{\frac{6}{(1-\gamma^*) n\rho(\bX)}}\frac{1}{\beta_{\min,j}^*}\big(\ln({p_j-s_j})-\ln({s_j})\big),
\end{equation}
and that $\lim_{n\to \infty}g(\gamma^*)\sqrt{ n\rho(\bX)/6}\;\beta_{\min,j}^* /(\sqrt{\ln(p_j-s_j)} + \sqrt{\ln(s_j)}) >1$ where $g(\gamma^*)=\sqrt{(-1 + 2 \gamma^*)(1-\gamma^*)}/(1 + \sqrt{2 - 2 \gamma^*})$.
Then Assumptions~\hyperlink{lab:A1}{A1} and \hyperlink{lab:A2}{A2} hold, and, for all sufficiently large $n$ and any $\delta \in (0,1)$, we have
\begin{equation}\label{eq:oraclerateconvreg}
\sup_{\bbeta^*\in\mathfrak{B}}P(\hat{S}^{ei} \neq S) \; \leq \; 2c \sum_{j=1}^b \exp\big[-\tfrac{\delta}{2}\big(\tfrac{(1-\gamma^*) n\rho(\bX)}{24}{\beta_{\min,j}^*}^2-\ln \max\{p_j-s_j,s_j\}\big)\big].
\end{equation}
\end{thm}

Theorem \ref{prop:fullrecovreg} gives the rate at which $P(\hat{S}^{ei} \neq S)$ vanishes.
In~\eqref{eq:rateconvreg}, each summand features a 
first term that decays exponentially in the penalty $\kappa_j$, and a second term that decays exponentially in 
$
(\sqrt{\tfrac{(1-\gamma)n\rho(\bX)}{6}}\,\beta_{\min,j}^* - \sqrt{\kappa_j})^2
$. These two terms capture competing type I and II error contributions.
The choice $\kappa_j=\kappa_j^*$ balances these contributions by equating their exponential rates, and thus approximately minimizes \eqref{eq:rateconvreg}.
We refer to $\kappa_j^*$ as \emph{oracle penalties}, since they depend on the unknown $s_j$ and $\beta_{\min,j}^*$. In~\eqref{eq:oraclerateconvreg}, we further derive a simplified upper bound for the rate achieved by $\kappa_j^*$.
This bound is sharpest in regimes where the signal strength term $n\rho(\bX)\beta_{\min,j}^{*2}$ grows faster than $\ln(p_j/s_j-1)$.

The results of Theorem~\ref{prop:fullrecovreg}, when applied to uninformed $\ell_0$ selector $\hat{S}^{sd}$, nearly match the minimax rates for the orthonormal case, where $\bX^{\top}\bX=n\I_p$ and $\rho(\bX)=1$.~\cite{butucea2018} show that, in this case, the choice $\sqrt{\tilde{\kappa}}=(1/2)\sqrt{n/2}\beta_{\min}^*+(1/2)\sqrt{2/n}\ln(p/s-1)/\beta_{\min}^*$ is exactly minimax for the Hamming loss and asymptotically minimax for the probability of incorrect recovery, when $\beta_j^* \geq 0$ for all $j$. The minimax rate for the Hamming loss is then shown to be the sum of a term decreasing exponentially in $\tilde{\kappa}-\ln(p-s)$ and a term decreasing exponentially in $(\sqrt{n/2}\beta_{\min}^*-\sqrt{\tilde{\kappa}})^2-\ln(s)$, similarly to our rate in~\eqref{eq:rateconvreg}.

\subsection{Necessary conditions}
\label{ssec:necessary_conditions}
We compare our sufficient Assumptions~\hyperlink{lab:A1}{A1} and~\hyperlink{lab:A2}{A2} to assumptions on $\kappa_j$ and $\beta_{\min,j}$ that are necessary for consistent variable selection with $\hat{S}^{ei}$. To obtain tight necessary assumptions, it suffices to require that $S$ is preferred over models that differ by  one variable. Consider, for each $j$, the subset of models that over-fit by only one variable from block $B_j$,
\begin{align}
    &\OO_j \:=\,\{M\in \M\;|\; M=S\cup\{i\} \text{ where } i\in B_j\setminus S_j\} \label{eq:defOj}.
\end{align}
Our necessary conditions involve the next two quantities
\begin{equation}\label{eq:ovpsk}
\underline{\lambda}_j\;:=\;\lambda^2_{\min}(C^j) 
\quad\text{ and}\quad
\bar{\lambda}\,:=\,\lambda_{\max}\Big(\frac1n \bX_{S}^\top\bX_{S}\Big), 
\end{equation}
where $C^j$ is the matrix in $\R^{(p_j-s_j) \times (p_j-s_j)}$ such that for any $M^k,M^l \in \OO_j$, $C^j_{k,l}=\operatorname{corr}\big(Z_{M^k},Z_{M^l}\big)$ and $Z_{M^k}^2=\|\bX_{M^k}\tilde{\bbeta}^{(M^k)}\|^2-\|\bX_{S}\tilde{\bbeta}^{(S)}\|^2$. That is, $Z_{M^k}^2$ is an increase in explained sum-of-squares measuring the extent to which $M^k$ overfits the observed data relative to $S$,
and hence small $\underline{\lambda}_j$ indicates that truly inactive variables in block $j$ are highly correlated with each other, whereas $\underline{\lambda}_j=1$ that they are uncorrelated. Similarly, $\bar{\lambda}$ relates to the correlation between active variables. 

\begin{prop}\label{prop:nectypeItypeII}
 \begin{enumerate}[label=(\roman*)]\tightlist
    \item If for some $j=1,\ldots,b$, $p_j-s_j\to\infty$ and $\lim_{n\to\infty}\kappa_j/(\underline{\lambda}_j\ln(p_j-s_j)) <1$, then 
     $\lim_{n \to \infty} \, \sup_{\bbeta^*\in\mathfrak{B}}\, P(\hat{S}^{ei}=S)= 0$.
    \item If for some $j=1,\ldots,b$, $\kappa_j\to\infty$  and
      $\lim_{n\to \infty}\sqrt{n\bar{\lambda}}\beta_{\min,j}^*-\sqrt{2\kappa_j} < \infty$ then there exists $\bbeta^*\in\mathfrak{B}$ such that $\lim_{n \to \infty} \, P(\hat{S}^{ei} = S) < 1$.
    \item If for some $j\in\{1,\ldots,b\}$,   $\lim_{n\to+\infty}(p_j-s_j)^{\underline{\lambda}_j}=+\infty$ and $\lim_{n\to\infty}\sqrt{n\bar{\lambda}}\beta_{\min,j}^*-\sqrt{2\underline{\lambda}_j\ln(p_j-s_j)} < \infty$, then there exists $\bbeta^*\in\mathfrak{B}$ such that $\lim_{n \to \infty} \,P(\hat{S}^{ei}=S) < 1$.
\end{enumerate}
\end{prop}

Proposition~\ref{prop:nectypeItypeII} implies that Assumptions \hyperlink{lab:A1}{A1}-\hyperlink{lab:A2}{A2} are fairly tight. By Part (i), the block penalty $\kappa_j$ must grow at least as fast as $\ln(p_j-s_j)$.
Since $\underline{\lambda}_j \in [0,1]$, in the worst case where $\underline{\lambda}_j=1$ Part (i) implies that it is necessary that 
$\kappa_j \geq \ln(p_j-s_j)$ 
nearly matching \hyperlink{lab:A1}{A1}. 
Part (ii) shows that the signal strength $\beta_{\min,j}^*$ must be sufficiently large relative to $\kappa_j$, and closely resembles the requirement that $\sqrt{(1-\gamma) n \rho(\bX) \beta_{\min,j}^*/6} - \sqrt{\kappa_j} - \sqrt{\ln(s_j)} \to \infty$ made by \hyperlink{lab:A2}{A2},
up to logarithmic term in $s_j$, and to replacing $\rho(\bX) (1 - \gamma)$ by $\bar{\lambda}$, where, recall, $1-\gamma \in (0,1/2)$ and it is bounded away from 0 if $\ln(p_j-s_j)=O(f_j)$ for all $j$.
Part (iii) follows by combining Parts (i)-(ii), and implies that $\beta_{\min,j}^*$ must satisfy
$\lim_{n\to\infty}\sqrt{n\bar{\lambda}}\beta_{\min,j}^*-\sqrt{ 2\underline{\lambda}_j\ln(p_j-s_j)}$,
which again resembles \hyperlink{lab:A2}{A2} in the worst case where $\underline{\lambda}_j=1$.
That is, Assumption~\hyperlink{lab:A2}{A2} and Parts (ii)-(iii) imply the same scalings of $(n,p_j,s_j,\beta_{\min,j}^*)$, up to logarithmic terms.

\subsection{Benefits of externally-informed penalties}\label{ssec:regression_gains}

This section compares the oracle-informed selector $\hat S^{ei}$ with the standard $\ell_0$ selector $\hat S^{sd}$, which possesses the best-known properties 
in terms of variable selection consistency \citep{gao2025optimalityl0}.
We highlight two types of improvements in $\hat S^{ei}$: 
(i) weaker conditions for variable selection consistency and (ii) faster convergence rates when both $\hat S^{ei}$ and $\hat S^{sd}$ are consistent.
Both improvements depend on how the blocks differ in their sparsity levels $s_j$ and signal strengths $\beta_{\min,j}^*$, that is on how informative the blocks truly are.

By Theorem~\ref{prop:fullrecovreg}, Assumptions~\hyperlink{lab:A1}{A1}–\hyperlink{lab:A2}{A2} are sufficient for consistent support recovery, both for $\hat S^{ei}$ and $\hat S^{sd}$ (as the particular case where $b=1$).
To compare $\hat S^{ei}$ and $\hat S^{sd}$, we plug in the smallest penalties $\kappa_j$  allowed by \hyperlink{lab:A1}{A1}–\hyperlink{lab:A2}{A2} and inspect the implied scaling of the sample size $n$.

Consider a high-dimensional setting where for all $j$, $p_j-s_j\to\infty$. For the informed selector $\hat S^{ei}$, with block penalties $\kappa_j \asymp \ln(p_j-s_j)$, consistency requires that
\begin{equation}\label{eq:rangeblockpenex}
\sqrt{n}
\;\gtrsim\;
\frac{\sqrt{\ln(p_j-s_j)}+\sqrt{\ln(s_j)}}{\sqrt{\rho(\bX)}\,\beta^*_{\min,j}}\quad\text{for each }j=1,\ldots,b.
\end{equation}
For the standard $\hat S^{sd}$, with a single penalty $\kappa \asymp \ln(p-s)$, consistency requires
\begin{equation}\label{eq:rangepenex}
\sqrt{n}
\;\gtrsim\;
\frac{\sqrt{\ln(p-s)}+\sqrt{\ln(s)}}{\sqrt{\rho(\bX)}\,\beta^*_{\min}}.
\end{equation}
Since $p_j-s_j \le p-s$, $s_j \le s$, and $\beta^*_{\min,j} \ge \beta^*_{\min}$ for all $j$, ~\eqref{eq:rangeblockpenex} is always weaker than ~\eqref{eq:rangepenex}.

The magnitude of this improvement depends on how strongly $(\beta_{\min,j}^*, s_j, p_j-s_j)$ vary across blocks, whether the blocks truly capture variables with different signal strengths or sparsity levels. To illustrate, we discuss next the requirement on signals strength implied by \eqref{eq:rangeblockpenex} and \eqref{eq:rangepenex}. That is,
for the informed selector $\hat S^{ei}$,
\begin{equation}\label{eq:betaminblockpenex}
\beta^*_{\min,j}
\;\gtrsim\;
\frac{\sqrt{\ln(p_j-s_j)}+\sqrt{\ln(s_j)}}{\sqrt{\rho(\bX)n}}\quad\text{for each }j=1,\ldots,b,
\end{equation}
and for the standard $\hat S^{sd}$,
\begin{equation}\label{eq:betaminpenex}
\beta^*_{\min}
\;\gtrsim\;
\frac{\sqrt{\ln(p-s)}+\sqrt{\ln(s)}}{\sqrt{\rho(\bX)n}}.
\end{equation}
The right-hand sides in \eqref{eq:betaminblockpenex} and \eqref{eq:betaminpenex} are smallest signal strength allowed by the sufficient assumptions for consistency (\hyperlink{lab:A1}{A1}–\hyperlink{lab:A2}{A2}) with $\hat S^{ei}$ and $\hat S^{sd}$ respectively. 
Consider a $b=2$ blocks example with $p-s=n^2$ inactive variables:
$p_1-s_1=n^{2}-n^{1/2}$ in block 1 and $p_2-s_2=n^{1/2}$ in block 2.
Suppose also that $s_1=s_2=O(1)$ and that the design is orthonormal, so that $\rho(\bX)=1$. For the standard selector $\hat S^{sd}$, consistency requires $\beta^*_{\min}\gtrsim\sqrt{2\ln n / n}$.
In contrast, for the informed selector $\hat S^{ei}$, consistency requires $\beta^*_{\min,1}\gtrsim\sqrt{2\ln n / n}$ in block~1, but only $\beta^*_{\min,2}\gtrsim\sqrt{\ln n / (2n)}$ in block~2. That is, $\hat S^{ei}$ can recover signals twice smaller than $\hat S^{sd}$ in block 2. Additional scenarios illustrating how gains depend on block sparsity and signal heterogeneity across blocks are provided in Supplement~\ref{app:exbenef}.

The comparison above is based on sufficient conditions.
Proposition~\ref{prop:nectypeItypeII}(iii) shows that a similar comparison holds for necessary conditions.
In particular, there exist regimes in which the sample size $n$ grows insufficiently quickly for consistent recovery using $\hat S^{sd}$, but sufficiently for $\hat S^{ei}$.
Corollary~\ref{cor:selconstimpossiblereg} formalizes this result.
Corollary~\ref{cor:selconstimpossiblereg} assumes a high-dimensional setting, where $\underline{\lambda} \ln(p-s)$ diverges, as in Proposition~\ref{prop:nectypeItypeII} (iii). 

\begin{cor}\label{cor:selconstimpossiblereg}
Assume \hyperlink{lab:A1}{A1}, \hyperlink{lab:A2}{A2} and $\lim_{n\to\infty}(p-s)^{\underline{\lambda}}=\infty$ for $\underline{\lambda}$ as in~\eqref{eq:ovpsk} for the $b=1$ block case. If
$\lim_{n\to\infty}\sqrt{n\bar{\lambda}}\beta_{\min}^*-\sqrt{2\underline{\lambda}\ln(p-s)}< \infty$,
then there exists $\bbeta^*\in\mathfrak{B}$ such that
$\lim_{n \to \infty} \, P(\hat{S}^{sd} = S) < 1$ and $\lim_{n \to \infty} \, \inf_{\bbeta^*\in\mathfrak{B}}P(\hat{S}^{ei} = S) = 1$.
\end{cor}
Supplement~\ref{app:exbenef} provides an illustration (simulation scenario 2).

We next compare the probabilities of correct selection when both procedures are consistent. Assume that block penalties are set to their oracle values $\kappa_j^*$ in Theorem \ref{prop:fullrecovreg}, and that the standard penalty is set to $\kappa^*$. Let $OR^{ei}$ and $OR^{sd}$ denote the corresponding oracle convergence rates in \eqref{eq:oraclerateconvreg}. Up to constant factors, their ratio satisfies
\begin{equation}\label{eq:fewkn}
\frac{OR^{ei}}{OR^{sd}}
=
\sum_{j=1}^b
\exp\!\left(
-\frac{\delta}{2}
\Big[
\frac{n\rho(\bX)(1-\gamma^*)}{24}
\big({\beta^*_{\min,j}}^2-{\beta_{\min}^*}^2\big)
+
\ln\!\frac{\max\{p-s,s\}}{\max\{p_j-s_j,s_j\}}
\Big]
\right),
\end{equation}
where $\delta$ can be taken arbitrarily close to~1.  Since $\beta^*_{\min,j}\ge\beta^*_{\min}$ and $\max\{p_j-s_j,s_j\}\le\max\{p-s,s\}$, it follows that $OR^{ei}=O(OR^{sd})$ for any block partition.
Moreover, if for every $j$, $\beta^*_{\min,j}>\beta^*_{\min}$ or $(p-s) / (p_j - s_j) \to \infty$, then $OR^{ei}=o(OR^{sd})$, implying faster convergence. Consider a sparse setting where $s_j < p_j - s_j$ for all $j$. There are then two sources of improvement in convergence rates: in any block $j$, a smaller number of truly zero parameters, that is $p_j-s_j< p-s$, or a larger signal, $\beta^*_{\min,j} > \beta_{\min}^*$. In blocks where $\beta^*_{\min,j} \approx \beta_{\min}^*$, improvement comes mainly from having a smaller number of truly zero means. In those blocks,
the oracle sets $\kappa^*_j < \kappa^*$, so that smaller signals are detected. If $\beta^*_{\min,j}$ is much greater than $\beta_{\min}^*$, the oracle sets $\kappa^*_j > \kappa^*$ and the probability of false positives is reduced. Further illustrations are provided in Supplement~\ref{app:exbenef}.

\subsection{Generalized convergence}\label{ssec:nobetaminhyp}

We now discuss a convergence result for $\hat S^{ei}$ that is a key ingredient for the data-driven procedures in Section~\ref{sec:informed_l0}. This result generalizes Theorem~\ref{theo:suffcondlinearmodel} by dropping the beta-min Assumption~\hyperlink{lab:A2}{A2}. Instead, active signals are separated into three categories---\emph{small}, \emph{intermediate}, and \emph{large}---relative to the penalties $\kappa_j$. We define next these three categories.

For fixed block penalties $\kappa_1,\ldots,\kappa_b$, and for $\gamma$ and $g_j$ as in Assumption~\hyperlink{lab:A2}{A2}, define
\begin{eqnarray}\label{eq:defSes}
    S_j^{L}(\kappa_j)
    &:=&
    \Big\{ i \in S_j \,\Big|\, \sqrt{\frac{(1-\gamma) n\rho(\bX)}{6}}\,|\beta_i^*|-\sqrt{\kappa_j}
    \;\geq\;\sqrt{\ln(s_j)} + g_j\Big\},\label{eq:defslsi}\\
    S_j^{S}(\kappa_j)
    &:=&
    \Big\{ i \in S_j \,\Big|\, \sqrt{n\bar{\lambda}}\,|\beta_i^*|=o\!\big(\sqrt{\kappa_j}\big) \Big\},\nonumber\\
    S_j^{I}(\kappa_j)
    &:=&
    S_j \setminus \big(S_j^{L}(\kappa_j) \cup S_j^{S}(\kappa_j)\big).\nonumber
\end{eqnarray}
where $\bar{\lambda}=\lambda_{\max}\Big(\bX_{S}^\top\bX_{S}/n\Big)$.
The set $S_j^{L}(\kappa_j)$ contains the large parameters, in that they meet the beta-min Assumption~\hyperlink{lab:A2}{A2}. The set $S_j^{S}(\kappa_j)$ contains small signals relative to the penalty $\kappa_j$. The remaining active variables form the intermediate set $S_j^{I}(\kappa_j)$.

Let $\bkappa=(\kappa_1,\ldots,\kappa_b)$ and define
$S^{L}(\bkappa)=\cup_{j=1}^b S_j^{L}(\kappa_j)$,
$S^{S}(\bkappa)=\cup_{j=1}^b S_j^{S}(\kappa_j)$, and
$S^{I}(\bkappa)=\cup_{j=1}^b S_j^{I}(\kappa_j)$.
Theorem \ref{thm:convtoT} shows that, under Assumption~\hyperlink{lab:A1}{A1}, the normalized scores $NC(M)$ in \eqref{eq:CM} concentrate on models that include all large signals, exclude all inactive variables and all small signals, and differ only in which intermediate variables they include:
\begin{equation} \label{eq:Tkappa}
\T(\bkappa)
\;:=\;
\Big\{M \in \M \,\Big|\,
S^{L}(\bkappa)\subseteq M \subseteq S^{L}(\bkappa)\cup S^{I}(\bkappa)
\Big\}.
\end{equation}
We require the following assumption.

\hypertarget{lab:A3}{(A3)} \quad For each block $j$, $|S_j^{I}(\kappa_j)|=O(1)$ and $|S_j^{S}(\kappa_j)|=O(p_j-s_j)$.

Assumption~\hyperlink{lab:A3}{A3} is mild and can be relaxed:  
$|S_j^{I}(\kappa_j)|$ can grow moderately with $p_j-s_j$ in our proof,
and one expects $|S_j^{I}(\kappa_j)|$ to be small because there is a relatively small gap between the ranges of $\beta_i^*$'s included in $S_j^S$ and $S_j^L$. 
Further, if $|S_j^{S}(\kappa_j)|/ (p_j-s_j) \to \infty$ then Theorem \ref{thm:convtoT} still holds, replacing, $p_j-s_j$ by $|S_j^{S}(\kappa_j)|$ in \hyperlink{lab:A1}{A1}.
\begin{thm}\label{thm:convtoT}
    Assume \hyperlink{lab:A1}{A1} and \hyperlink{lab:A3}{A3}, for every $j=1,\ldots,b$,  then, \\$\lim_{n \to \infty} \sum_{M \in \M\setminus \T(\bkappa)} \, \sup_{\bbeta^*\in\mathfrak{B}}\, \E\left(NC(M)\right)= 0 \mbox{  and  } 
    \lim_{n \to \infty} \, \inf_{\bbeta^*\in\mathfrak{B}}\, P(\hat{S}^{ei} \in \T(\bkappa))= 1.$
\end{thm}

\section{Data-based externally-informed penalties}\label{sec:informed_l0}

We now propose a data-driven method to set the block penalties $\kappa_1,\ldots,\kappa_b$ without relying on an oracle.
The main idea is to adapt $\kappa_1,\ldots,\kappa_b$ to the sparsity in each block. 
Our construction uses the Bayesian interpretation of $\ell_0$ penalties via the BIC approximation in~\eqref{eq:block_modelprob}.
This link suggests an empirical Bayes strategy for estimating, in each block $B_j$, the number active coefficients $s_j$.
Although the motivation is Bayesian, the procedure is fully data-based and does not require specifying a prior on $(\bbeta,M)$.

Recall that under the BIC approximation in~\eqref{eq:block_modelprob}, choosing $\kappa_j$ is equivalent to choosing a block-wise prior inclusion probability 
$\theta_j=(\exp(\kappa_j-\ln(n)/2)+1)^{-1}$,
which is 
a prior guess for $s_j/p_j$, the proportion of active coefficients in block $B_j$. A standard empirical Bayes strategy to set $\hat{\theta}_j$ is to maximize the marginal likelihood of $\by$ given $\btheta=(\theta_1,\ldots,\theta_b)$, i.e.
\begin{eqnarray}
\hat{\btheta}
\,:=\, \arg\max_{\btheta} p(\by \mid \btheta)=
\arg\max_{\btheta} \int p(\by \mid \bbeta, M) d P(\bbeta, M \mid \btheta).
\nonumber
\end{eqnarray}

Lemma~\ref{lem:empiricalbayesinterpret} shows that $\hat{\btheta}$ satisfies a fixed point equation, whereby the prior inclusion probability $\hat{\theta}_j$ for block $j$ equals the average posterior inclusion probability across variables in that block given $\hat{\btheta}$.
\begin{lem}\label{lem:empiricalbayesinterpret} 
Let $p(M \mid \btheta)= \prod_{i=1}^p \operatorname{Bern}\left(m_i ; \theta_j\right) I(i \in B_j)$ as in~\eqref{eq:block_modelprior} and $P(\beta_j \neq 0 \mid \by, \btheta)= \sum_{M\in\M: j\in M} P(M \mid \by,\btheta)$ be the posterior inclusion probability of variable $j$. Then $\hat{\btheta}$ satisfies 
\[\hat{\theta}_j= \frac{1}{p_j}\sum_{i\in B_j} P(\beta_i \neq 0 \mid \by, \hat{\btheta}).
\]
\end{lem}

Motivated by Lemma~\ref{lem:empiricalbayesinterpret} we estimate $s_j$, the number of active variables in block $B_j$ by the sum of pseudo-posterior inclusion probabilities,
computed using $NC(M)$ in \eqref{eq:CM} as a proxy for $P(M\mid \by,\btheta)$:
\begin{equation}\label{eq:shat}
    \hat{s}_j \,:=\, \sum_{i\in B_j}\,\,\sum_{M\in \M : i \in M}NC(M).
\end{equation}
We have the following asymptotic bounds on the expectation of $\hat{s}_j/p_j$. 
\begin{prop}
    \label{prop:shatconsistence}
Assume \hyperlink{lab:A1}{A1} and \hyperlink{lab:A3}{A3}. Then
\[ \,\, \frac{|S_j^{L}(\kappa_j)|}{p_j}\leq \lim_{n\to\infty}\E\bigg(\frac{\hat{s}_j}{p_j}\bigg)\leq \frac{s_j-|S_j^{S}(\kappa_j)|}{p_j}\quad\text{for all }j=1,\ldots,b \text{ and }\bbeta^*\in\mathfrak{B}.\]

\end{prop}
As a consequence, if the betamin condition in Assumption~\hyperlink{lab:A2}{A2} also holds, then $\hat{s}_j/p_j$ is an asymptotically unbiased
estimator of $s_j/p_j$ because $S_j=S_j^{L}(\kappa_j)$ and $S_j^{S}(\kappa_j)=\emptyset$ for all $j$. 
The proposition also shows that, when Assumption~\hyperlink{lab:A2}{A2} is not met, $\hat{s}_j/p_j$ is asymptotically downward biased by at least the proportion of small signals $|S_j^{S}(\kappa_j)|/p_j$, but it is guaranteed to be larger than the proportion of large signals $|S_j^L(\kappa)|/p_j$ (i.e., those satisfying \hyperlink{lab:A2}{A2}).

We propose a two-step procedure that uses $\hat{s}_j$ in~\eqref{eq:shat} to set adaptive block penalties that leverage the difference in sparsity across blocks. 
In the first step,  we use a single common penalty across blocks $\kappa^{\circ}$ and obtain the corresponding $\hat{s}_j$.
In the second step, we use $\hat{s}_j$ to set block penalties $\kappa_j^{A}= \ln(p_j-\hat{s}_j)+ \tfrac{1}{2}\ln(n)$.
This choice of penalties amounts to setting $\theta_j= (p_j-\hat{s}_j+1)^{-1}$ in the BIC approximation~\eqref{eq:block_modelprob}.
We briefly discuss alternative choices of $\theta_j$ at the end of the section.

\textbf{\hypertarget{lab:AlgoAIS}{Algorithm 1} - Adaptive informed selector $\hat{S}^{A,ei}$}
\begin{enumerate}\tightlist
    \item Set $\kappa_j=\kappa^{\circ}=\ln(p)+\frac{1}{2}\ln(n)$ 
    for $j=1,\ldots,b$. Compute $\hat{s}_j/p_j$ in~\eqref{eq:shat} for $j=1,\ldots,b$.
    \item Obtain $\hat{S}^{A,ei}$ solving~\eqref{eq:Mhat} with
    $\kappa_j^{A} = \ln(p_j-\hat{s}_j)+ \tfrac{1}{2}\ln(n)$.
\end{enumerate}

We refer to $\hat{S}^{A,ei}$ as the adaptive informed selector.  
Step 1 can be carried out with any Bayesian computational method that returns posterior inclusion probabilities $P(\beta_i \neq 0 \mid \by, \btheta)$. 
Step 2 can be done with any computational method to find the model with best informed $\ell_0$ criterion.
In our examples in Section~\ref{sec:examplesinforml0} we use MCMC for both steps. 

In the oracle analysis of Section~\ref{sec:properties_legression}, the selection properties depend on the number of inactive variables $p_j-s_j$ in each block. In practice, in Algorithm~1, Step~1 may fail to detect small signals, so $p_j - \hat{s}_j$ captures the number of variables that remain effectively indistinguishable from zero in Step~1. That is, $p_j-|S_j^{L}(\kappa^\circ)|$ where $S_j^{L}(\kappa^\circ)$ is the set of large signals as defined in~\eqref{eq:defSes}, evaluated at the Step~1 penalty $\kappa^\circ$.
This leads to a beta-min condition stated in terms of $p_j-|S_j^{L}(\kappa^\circ)|$ for $j=1,\ldots,b$,

\hypertarget{lab:A4}{(A4)}\quad For each block $j$, there exists $d_j\to\infty$, such that
\begin{equation*}
        \sqrt{\frac{ (1-\xi)n\rho(\bX)}{2}}\beta_{\min,j}^* \,- \,\sqrt{\ln\big(p_j-|S^{L}_j(\kappa^{\circ})|\big)+\frac12\ln(n)} \;=\; \sqrt{\ln(s_j)} + d_j,
    \end{equation*}
where $\xi=\tfrac12(1+\max_j \ln(p_j-s_j)/(\ln(p_j-s_j)+\tfrac 12 \ln(n)))$.

Under Assumptions~\hyperlink{lab:A3}{A3} and \hyperlink{lab:A4}{A4}, $\hat{S}^{A,ei}$ consistently recovers the true support $S$.

\begin{thm}\label{theo:informedl0consist}
    Assume that $\kappa^{\circ}$ satisfies \hyperlink{lab:A3}{A3} and that \hyperlink{lab:A4}{A4} holds. Then 
$\lim_{n \to \infty} \,\inf_{\bbeta^*\in\mathfrak{B}} P(\hat{S}^{A,ei} = S) = 1$.
\end{thm}

We compare next $\hat{S}^{A,ei}$ to its uninformed counterpart $\hat{S}^{A,sd}$ that sets in Step 2 a single common penalty $\kappa^{A}=\ln(p-\hat{s})+ \tfrac{1}{2}\ln(n)$. By Theorem~\ref{theo:informedl0consist}, $\hat{S}^{A,sd}$ is variable selection consistent under the betamin assumption
\begin{equation}\label{eq:betaminA}
        \sqrt{\frac{ (1-\xi')n\rho(\bX)}{2}}\beta_{\min}^* \,- \,\sqrt{\ln\big(p-|S^{L}(\kappa^{\circ})|\big)+\frac{\ln(n)}{2}} \;=\; \sqrt{\ln(s)} +  d,
    \end{equation}
where $\xi'=\tfrac12(1+\ln(p-s)/(\ln(p-s)+\tfrac12 \ln(n)))$ and $d\to \infty$. 
Since
$\ln(p-|S^{L}(\kappa^{\circ})|)\geq \ln(p_j-|S_j^{L}(\kappa^{\circ})|)$ and $\ln(s)\geq \ln(s_j)$, Assumption \hyperlink{lab:A4}{A4} for $\hat S^{A,ei}$ is milder than~\eqref{eq:betaminA} for $\hat S^{A,sd}$, in particular in less sparse blocks where $p_j-|S_j^{L}(\kappa^{\circ})|$ is smaller than $p-|S^{L}(\kappa^{\circ})|$. Although not shown here, necessary betamin assumptions are also milder for $\hat S^{A,ei}$ than for $\hat S^{A,sd}$.

A natural alternative is to set prior inclusion probabilities to $\theta_j=\hat{s}_j/p_j$, the empirical Bayes estimate of the proportion of active parameters in block $j$, which leads to setting $\kappa^{EB}_j=\ln(p/\hat{s}_j-1)$ in Step 2 of Algorithm 1.
The consistency of such a procedure can be proven under a mild additional assumption that bounds the number of active parameters. Compared to the uninformed counterpart with common penalty  $\kappa^{EB}=\ln(p/\hat{s}-1)$, the resulting informed selector requires milder conditions for consistency in denser blocks where $\hat s_j/p_j>\hat s/p$, and harsher conditions in the sparser blocks where $\hat s_j/p_j<\hat s/p$. For example, if the global $\beta_{\min}^*$ is located in a sparse block (small $\hat s_j/p_j$), then this alternative procedure results in low power relative to Algorithm 1.

\section{Structural transfer learning}\label{sec:TL}

We now study a structural transfer learning method built on the adaptive informed selector $\hat S^{A,ei}$.
We assume that for each source dataset $k=1,\ldots,K$ we are given a set of selected variables $\hat{S}^{(k)}$, and we want to use these sets to improve variable selection in a target dataset.
The hope is that the sources and target are similar regarding the \emph{support} (truly active variables),
in contrast to transfer learning for estimation and prediction~\citep{Li2021,Tian2023,abba2024} where similarity is expected in the parameter values.

A naive strategy would force the inclusion of any variable selected in at least one source dataset.
This can lead to \emph{negative transfer}: if a variable is truly active in a source dataset but not in the target dataset, one commits a type I error.
In the worst case, the set of truly active variables in the source and target may not overlap at all.
Instead, we use the source selections $\hat{S}^{(k)}$ to define  external information $\bz$ that partitions variables into blocks. In a first step, for each variable $i=1,\ldots,p$, we define $z_i= I (i\in \cup_{k=1}^K\hat{S}^{(k)})$. 
That is, variables in the target dataset are split into a first block containing those that were selected in $\geq 1$ source datasets, and a second block gathering the rest.
In a second step, we use the adaptive informed selector $\hat{S}^{A,ei}$ with that partition to select variables in the target dataset. This procedure provably prevents negative transfer (asymptotically).

\textbf{\hypertarget{lab:AlgoTL}{Algorithm 2} - Tran-s-ell0 (a simple two-block transfer rule)}
\begin{enumerate}[label=(\roman*)]\tightlist
    \item Construct $B_1= \cup_{k=1}^K\hat{S}^{(k)}$ and $B_2=V\setminus B_1$.
    \item Run \hyperlink{lab:AlgoAIS}{Algorithm 1} on the target dataset using the block partition $(B_1,B_2)$, and denote the output by $\hat{S}^{\rm tr}$.
\end{enumerate}

Alternative strategies that construct more than two blocks from the source selections are possible, e.g., by grouping variables by how many source datasets selected them. Our theory continues to apply once a block partition is fixed.
More flexible approaches regress inclusion probabilities on the full vector $(I\{i\in \hat S^{(1)}\},\ldots,I\{i\in \hat S^{(K)}\})$~\citep{vandeWiel2019}. Such approaches require a separate theoretical treatment and are left for future research.

The theoretical properties of our algorithm
follow directly from Theorem~\ref{theo:informedl0consist}, applied to the blocks defined using the $\hat{S}^{(k)}$'s. Let
\begin{equation}\label{eq:afihgr}
\begin{aligned}
   & s_1 = |B_1 \cap S| = |\cup_{k=1}^K\hat S^{(k)}\cap S|,\\
   & s_2 = |B_2 \cap S| = |S\setminus \cup_{k=1}^K  \hat S^{(k)}|,\\
   & p_1-s_1 = |B_1 \setminus S| = |(\cup_{k=1}^K\hat S^{(k)})\setminus S |,\\
   & p_2-s_2 = |B_2 \setminus S| = |V \setminus (\cup_{k=1}^K \hat S^{(k)}\cup S)|,\\
   & p_1-|S_1^L(\kappa^{\circ})| = |B_1 \setminus S^L(\kappa^{\circ})| = |(\cup_{k=1}^K\hat S^{(k)})\setminus S^L(\kappa^{\circ}) |,\\
   & p_2-|S_2^L(\kappa^{\circ})| = |B_2 \setminus S^L(\kappa^{\circ})| = |V \setminus (\cup_{k=1}^K \hat S^{(k)}\cup S^L(\kappa^{\circ}))|,\\
\end{aligned} 
\end{equation}
where $S^L(\kappa^{\circ})$ is the subset of large active signals in $S$ as defined in~\eqref{eq:defSes} for $\kappa^{\circ}$ as in Step 1 of \hyperlink{lab:AlgoAIS}{Algorithm 1}.  

Consider the betamin assumption\\
\hypertarget{lab:A5}{(A5)}\quad There exist $d_1,d_2\to \infty$ such that,
\begin{align*}
        &\sqrt{\frac{ (1-\xi)n\rho(\bX)}{6}}\beta_{\min,1}^*- \,\sqrt{\ln\big(|(\cup_{k=1}^K\hat S^{(k)})\setminus S^L(\kappa^{\circ})|\big)+\frac{\ln(n)}{2}}= \sqrt{\ln\big(|\cup_{k} S^{(k)}\cap S|\big)} + d_1,\\\text{ and,}\\
    &\sqrt{\frac{ (1-\xi)n\rho(\bX)}{6}}\beta_{\min,2}^*- \sqrt{\ln\big(|V \setminus (\cup_{k=1}^K \hat S^{(k)}\cup S^L(\kappa^{\circ}))|\big)+\frac{\ln(n)}{2}}= \sqrt{\ln\big(|S\setminus \cup_{k=1}^K  \hat S^{(k)}|\big)} + d_2,
\end{align*} where $\xi=\tfrac12(1+\max_j \ln(p_j-s_j)/(\ln(p_j-s_j)+\tfrac12\ln(n)))$.
Assumption~\hyperlink{lab:A5}{A5} is Assumption~\hyperlink{lab:A4}{A4}, plugging in the number of active ($s_j$) and inactive ($p_j-s_j$) parameters in~\eqref{eq:afihgr}.

\begin{cor}\label{cor:TLconsist}
    Assume that $\kappa^{\circ}$ satisfies \hyperlink{lab:A3}{A3} and that \hyperlink{lab:A5}{A5} holds. Then $\lim_{n \to \infty} \,\inf_{\bbeta^*\in\mathfrak{B}}P(\hat{S}^{tr} = S) = 1$.
\end{cor}
Corollary~\ref{cor:TLconsist} shows that Tran-s-ell0 is selection consistent under milder conditions than a standard, uninformed, $\ell_0$ selector $\hat S^{A,sd}$, which, recall, is consistent under~\eqref{eq:betaminA}. The larger the overlap between the set of truly active variables $S$ in the target dataset and those selected in the source datasets, $\cup_{k=1}^K\hat S^{(k)}$,
the milder \hyperlink{lab:A5}{A5} becomes relative to~\eqref{eq:betaminA}.
Corollary~\ref{cor:TLconsist} continues to apply even when $\cup_{k=1}^K\hat S^{(k)}$ provides a poor guess for $S$, making Tran-s-ell0 robust to negative transfer. In that case however, Trans-s-ell0 provides no advantage over standard $\ell_0$ criteria.

Remarkably, Tran-s-ell0 achieves gains under weaker requirements than transfer learning algorithms for parameter estimation and prediction. The latter either assume that the sets of variables are the same between source and target datasets \citep{abba2024}, or make strong assumptions on the similarity of design matrices \citep{Li2021,Tian2023}. This is required to avoid negative transfer because parameter estimates are highly sensitive to covariate correlation structures. The gains in estimation that are achievable through transfer learning were in fact shown to be inherently limited by the dissimilarity between source and target design matrices~\citep{liu2025}, an issue also known as the \textit{covariate shift} problem. Our results on structural transfer learning require no assumption on the similarity of design matrices: the variable sets in the source and target datasets may differ largely, even in size.

In addition, many (but not all) transfer learning algorithms for parameter estimation and prediction, e.g.~\cite{Jiang2016} and~\cite{Li2021}, require access to the source datasets. Tran-s-ell0 only requires sets of selected variables as input from the source. It is then more widely applicable and enjoys stronger privacy and communication efficiency guarantees. 

\section{Numerical illustrations}\label{sec:examplesinforml0}

We illustrate the performance of the adaptive informed selector $\hat{S}^{A,ei}$ (Algorithm~1) on simulated data under different levels of relevance of the external information (Section~\ref{sec:simulationsss}).
We then illustrate the performance of Tran-s-ell0 (Algorithm~2) on simulated and empirical data (Section~\ref{sec:TLexample}).
In Step~1 of Algorithm~1 (and hence also within Algorithm~2), we use $\kappa^{\circ}=\tfrac12\ln(n)+\ln(p)$ and
the MCMC algorithm in function \texttt{bestIC} in R package \texttt{mombf}. Briefly, each model is assigned a
pseudo-posterior probability given by the BIC approximation~\eqref{eq:dfbapon}, and an MCMC is used to sample from this distribution. 
In Step~2 of Algorithm~1, we obtain $\hat S^{A,ei}$ by re-scoring the models visited by the MCMC in Step~1 and selecting the best-scoring one. We obtain similarly, for comparison purposes, $\hat S^{A,sd}$, the uninformed counterpart of $\hat{S}^{A,ei}$ (using a single common penalty). We also compare $\hat S^{A,ei}$ and Tran-s-ell0 to the standard $\ell_0$ penalty EBIC~\citep{EBIC}, the LASSO~\citep{lasso} and SCAD~\citep{Fan01SCAD}. For the last two methods, the penalization parameter is chosen by cross-validation. The R code to reproduce our analyses is at 
\url{https://github.com/paulrognonvael/varsel.extdata}.

\subsection{Adaptive informed selector}\label{sec:simulationsss}

We simulate data from the linear model~\eqref{eq:linearmodel} with standard Gaussian errors ($\sigma^2=1$), sample size $n \in [20,700]$, and variable values drawn from a multivariate Gaussian with zero means, unit variances, and pairwise correlations $\mbox{cor}(x_{ij}, x_{il})= 0.5$ for $j \neq l$.

\begin{table*}[ht]
\caption{Sparsity and block assumptions in simulation scenarios 1--3}
\label{tab:examples_bis}
\begin{center}
\begin{tabular}[]{@{}cccccccc@{}}
\hline
Scenario & $p-s$ & $s$ & $s_1/(p_1-s_1)$ &$s_2/(p_2-s_2)$ & $\beta_{\min,1}^*$ & $\beta_{\min,2}^*=\beta_{\min}^*$ \\
\hline
$1$    & $3n/2$ & $3\ln(n)$ & $3\ln(n)/(3n-2\sqrt{n})$ & $3\ln(n)/(2\sqrt{n})$  & 0.5 & 0.33\\
$2$    & $n$ & $3\ln(n)$ & $3\ln(n)/(2n-2\ln(n))$ & $3/2$   & 0.5 & 0.33 \\
$3$    & $n$ & $3\ln(n)$ & $3\ln(n)/n$ & $3\ln(n)/n$  & 0.5 & 0.5\\
\hline
\end{tabular}
\end{center}
\end{table*}

Table~\ref{tab:examples_bis} summarizes our three simulation scenarios, which differ in the sparsity and the informativeness of the blocks. In Scenarios 1 and 2, the blocks are informative: block $B_1$ is sparser than $B_2$. In Scenario 1, the blocks are moderately informative, in that $[s_2/(p_2-s_2)]/[s_1/(p_1-s_1)]= O(\sqrt{n})$. In Scenario 2, the blocks are highly informative: $\bbeta^*$ is sparse overall, but it is non-sparse within block 2 ($s_2 > p_2-s_2)$. In Scenario 3, the blocks are non-informative: each
has half of the inactive parameters. In all scenarios, the nonzero $\beta_i^*$, $i\in S$, are set uniformly spaced in $(1/2, 1)$ except for the smallest parameter in each block, which takes the values $\beta_{\min,1}^*$ and $\beta_{\min,2}^*$ specified in the last two columns of Table~\ref{tab:examples_bis}.

\begin{figure}[ht]
\begin{tabular}{M{49mm}M{49mm}M{49mm}}
        Scenario 1& Scenario 2 & Scenario 3\\
         \includegraphics{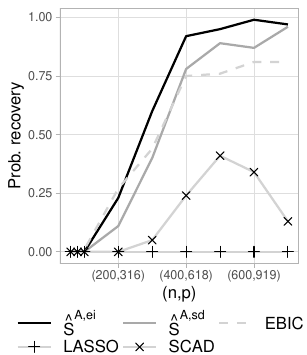} & \includegraphics{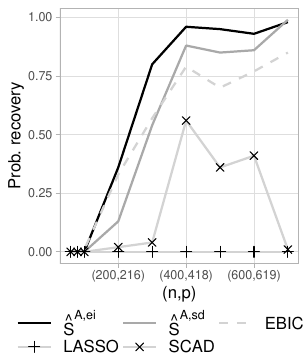} & \includegraphics{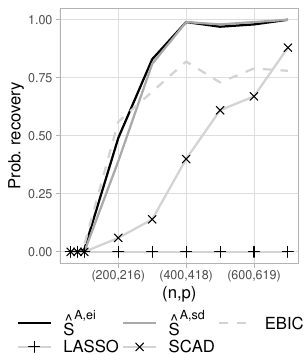}
\end{tabular}
\caption{Probability of correct selection with externally-informed $\hat S^{ei}$, non-informed $\hat S^{A,sd}$, EBIC, LASSO and SCAD in Scenario 1 (left), 2 (middle), 3 (right).}
\label{fig:linregprobrec}
\end{figure}

\begin{figure}[ht]
\begin{tabular}{M{49mm}M{49mm}M{49mm}}
         FDR& Power & MSE\\
         \includegraphics{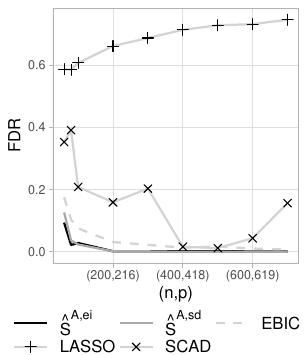}  & 
         \includegraphics{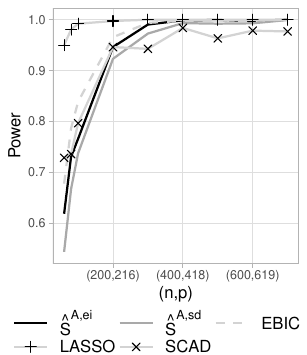} & 
         \includegraphics{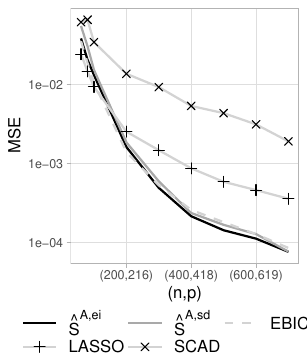} 
\end{tabular}
\caption{Simulation scenario 2. FDR (left), power (middle) and estimation MSE (right) with $\hat S^{A,ei}$, $\hat S^{A,sd}$, EBIC, LASSO and SCAD}
\label{fig:linregfdrpowmse}
\end{figure}

For each sample size, we generate 100 datasets and report in Figure~\ref{fig:linregprobrec} the empirical probability of correct recovery for $\hat S^{A,ei}$, $\hat S^{A,sd}$, EBIC, LASSO and SCAD. In Scenario 1 and 2 where blocks are discriminative, our externally-informed $\hat S^{A,ei}$ outperforms the non-informed $\hat S^{A,sd}$, and the EBIC.
In Scenario 3, where blocks are non-discriminative, $\hat S^{A,ei}$ performs similarly to $\hat S^{A,sd}$. LASSO and SCAD have significantly lower probability of recovery than the $\ell_0$ methods. The average false discovery rate (FDR) and power in Scenario 2, reported in Figure~\ref{fig:linregfdrpowmse}, shed light on these differences. Compared to the uninformed $\hat S^{A,sd}$, our informed $\hat S^{A,ei}$ has equally low or lower FDR and higher power. LASSO and EBIC have higher power than $\hat S^{A,ei}$ but higher FDR. In terms of mean squared error, $\hat S^{A,ei}$, $\hat S^{A,sd}$ and EBIC perform similarly, and better than SCAD and LASSO.
Similar findings are observed in Scenarios 1 and 3 
(Figure~\ref{fig:linregfdrpowmsesc13} in Section~\ref{app:fdrpowersim}). 
We also report in Figure~\ref{fig:runtime} in Section~\ref{app:fdrpowersim} the average computation time to obtain $\hat{S}^{A,sd}$ as a function of dimension $p$. In all three settings, the average time was below 20 seconds for all $(n,p)$.

\subsection{Structural transfer learning with Tran-s-ell0}\label{sec:TLexample}
\begin{figure}[ht]
\begin{tabular}{M{52mm}M{52mm}M{52mm}}
         Scenario 1, $h=0$& Scenario 2, $h=0.4$ & Scenario 3, $h=1$\\
         \includegraphics{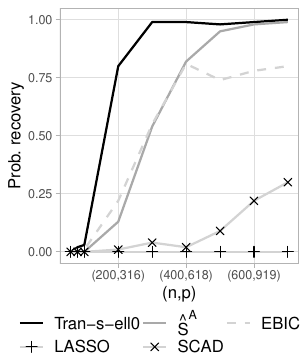} &          \includegraphics{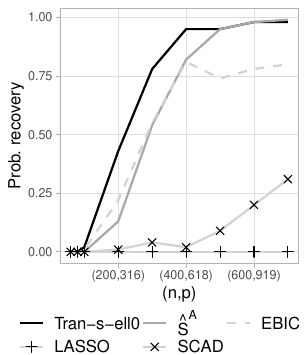} &      \includegraphics{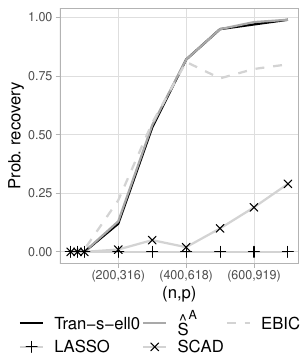} 
\end{tabular}
\caption{Probability of correct selection with Tran-s-ell0 algorithm, $\hat S^{A,sd}$, EBIC, LASSO and SCAD in Scenario 1 (left), 2 (middle), 3 (right).}
\label{fig:TLpr}
\end{figure}

We apply next the \hyperlink{lab:AlgoTL}{Tran-s-ell0} algorithm to simulated data. We consider target dataset sample sizes $n \in [20,700]$ and, for each sample size, obtain 100 simulations.
In each target dataset, the data-generating $\bbeta^*$ has $p-s=3n/2$ zero entries and $s=3\ln(n)$ non-zero entries uniformly spaced in $(1/3,1)$. For each target dataset, we generate two source datasets with sample size 700 and data-generating parameters that are equal to the target $\bbeta^*$, except that a proportion $h$ of the nonzero $\beta_i^*$'s are set to 0 and an equal number of the zero $\beta_i^*$'s are set to be uniformly spaced in $(1/3,1)$. 
We consider three scenarios. In Scenario 1, the truly active variables in the source and target sets are the same, that is $h=0$. In Scenario 2, there is a moderate overlap, with $h=0.4$. Scenario 3 is an example of completely irrelevant transfer where $h=1$. In the target datasets, variable values are drawn from a multivariate Gaussian with zero means, unit variances and pairwise correlations $\mbox{cor}(x_{ij}, x_{il})= 0.5$. In the source datasets, variable values are drawn from a multivariate t-distribution with 8 degrees of freedom and pairwise correlations $\mbox{cor}(x_{ij}, x_{il})= 0.5^{|i-j|}$.
We select variables in the source datasets with a standard $\ell_0$ penalty $\kappa=\tfrac12\ln(n)+\ln(p)$ and feed the selected sets to Tran-s-ell0. 

Figure~\ref{fig:TLpr} shows the average probability of recovery.
In Scenarios 1 and 2, \hyperlink{lab:AlgoTL}{Tran-s-ell0 algorithm} outperforms all other methods. In Scenario 3, Tran-s-ell0 performs similarly to $\hat S^{A,sd}$, better than LASSO and SCAD, slightly worse than EBIC for small $n$ but significantly better for large $n$.
These results support that Tran-s-ell0 is robust to negative transfer.

\subsection{Colon cancer data}

In molecular biology, it is common to use animal models (e.g. mice) to study genomic mechanisms. Findings from such studies often serve as a basis to study human diseases, for example, to assess whether genes that played key roles in the animal model are also important in humans. Specifically,~\cite{calon:2012} found a list of 172 genes that were associated with gene TGFB in mice, and showed that TGFB plays a crucial role in colon cancer progression. 
Our goal is to identify what genes within a set of $p=1000$ genes analyzed in~\cite{rossell:2017} are associated with TGFB in humans, using data from $n=262$ patients, by regressing TGFB expression on that of the $p=1000$ genes.
To this end, we run Tran-s-ell0 defining two blocks of variables: one with the 172 genes found in the mouse study, and another with the remaining genes. All variables are standardized to zero mean and unit variance. In this dataset, the runtime of Tran-s-ell0 was 49 seconds. 
\begin{table*}[ht]
\caption{Selected genes in each block}
\label{tab:selgenes}
\centering
\begin{tabular}{lC{0.1\linewidth}C{0.22\linewidth}C{0.1\linewidth}C{0.22\linewidth}}
\hline
 & \multicolumn{2}{c}{In mouse list} & \multicolumn{2}{c}{Outside of mouse list} \\ [3pt] \cline{2-5} 
 Method & \multicolumn{1}{C{0.15\linewidth}}{Number of discoveries} & Genes & \multicolumn{1}{C{0.15\linewidth}}{Number of discoveries} & Genes\\ 
\hline
Tran-s-ell0 & \multicolumn{1}{C{0.1\linewidth}}{4} & ESM1, GAS1, HIC1, CILP & \multicolumn{1}{C{0.1\linewidth}}{0} & - \\
\hline
$\hat S^{A,sd}$ & \multicolumn{1}{C{0.1\linewidth}}{2} & IGFBP3, CILP & \multicolumn{1}{C{0.1\linewidth}}{2} & LGALS1, ARHGEF1  \\
\hline
EBIC & \multicolumn{1}{C{0.1\linewidth}}{4} &  ESM1, GAS1, HIC1, CILP & \multicolumn{1}{C{0.1\linewidth}}{3} & KCNJ5-AS1, ZBTB46, OXER1 \\
\hline
LASSO& \multicolumn{1}{C{0.1\linewidth}}{22} & GAS1, HIC1, IGFBP3, CILP, DNAJB5, \ldots & \multicolumn{1}{C{0.1\linewidth}}{55} &  KCNJ5-AS1, ZBTB46, OXER1, LGALS1, \ldots \\
\hline
SCAD & \multicolumn{1}{C{0.1\linewidth}}{11} & GAS1, HIC1, IGFBP3, CILP, DNAJB5, \ldots & \multicolumn{1}{C{0.1\linewidth}}{5} &  POLQ, CRIP2, ZBTB46, VSIG4, LGALS1 \\
\hline
\end{tabular}
\end{table*}
\begin{table*}[ht]
\caption{Tran-s-ell0  vs EBIC: cross-validated MSE and pseudo-posterior probabilities}
\label{tab:cvmse.pseudoprob}
\centering
\begin{tabular}{lccccc}
\hline
 & & \multicolumn{4}{c}{Pseudo-posterior probabilities}  \\[3pt] \cline{3-6} 
 Method & CV-MSE & ESM1 & GAS1 & HIC1 & CILP\\ 
\hline
EBIC & 0.54& 0.55 & 0.88 & 0.81 & 0.95\\
\hline
Tran-s-ell0 & 0.49 & 0.74 & 0.90 & 0.95 & 0.99\\
\hline
\end{tabular}
\end{table*}

Table~\ref{tab:selgenes} reports the genes selected with Tran-s-ell0, $\hat S^{A,sd}$, EBIC, LASSO and SCAD. Compared to the uninformed selector $\hat{S}^{A,sd}$, Tran-s-ell0
selects more genes that were in the mouse list, and fewer genes from outside the list. LASSO and SCAD select a much larger number of genes than Tran-s-ell0, although recall that in our simulations in Section~\ref{sec:simulationsss} these methods exhibited high  false discovery rates (see also \cite{slope}). The extra discoveries of LASSO and SCAD should hence be considered with caution. Finally, the EBIC selects the same genes from the mouse list as Tran-s-ell0 but also three more genes from outside the list. Recall that in our simulations, EBIC had a slightly larger FDR than our informed selector but also a slightly larger power, for small sample sizes. 

To assess the relevance of the extra discoveries of EBIC and the quality of the models selected by EBIC and Tran-s-ell0, we report in Table~\ref{tab:cvmse.pseudoprob} the associated cross-validated mean squared error (CV-MSE) in estimation.
As a measure of variable relevance, we also report pseudo-posterior inclusion probabilities $\pi(\beta_i)=\sum_{M\in\M} NC(M)I(i\in M)$. 
Compared to the EBIC, Tran-s-ell0 leads to lower mean squared error and to higher pseudo-posterior inclusion probabilities for the genes in the mouse list. These findings suggest that there is value in structural transfer learning from mice to humans.

\section{Discussion}
We studied how 
variable selection in high-dimensional regression can be improved when one has external information that suggests that variables in some blocks are more likely than those in other blocks to be associated to the response. 
We first characterized the gains achievable by an oracle choice of block penalties, which serves as a benchmark for what external information can provide. 
We then proposed a data-based procedure that learns block penalties from the data and retains the main qualitative benefits of the oracle selector. 
As an important special case, we analyzed a structural transfer learning setting in which sets of variables selected in source studies are used as external information to guide selection in a target study, yielding a method that is robust to negative transfer.

Several directions remain open. 
First, it would be interesting to understand whether improvements can be obtained for estimation (not only support recovery) when such external information is available; while our simulations suggest potential gains, we did not develop estimation procedures or theory in this direction. 
Second, our results focus on linear regression. Extending the analysis to other structural learning problems—such as graphical model selection—could clarify when and how external information can relax the conditions required for consistent structure recovery.
Finally, it would be interesting to develop theory for cases where the external information does not simply split variables into blocks, but rather provides a richer description such as continuous measures of variable relevance found in several previous studies. 

\bibliography{references.bib}

\phantomsection\label{supplementary-material}
\bigskip

\begin{center}

{\large\bf SUPPLEMENTARY MATERIAL}

\end{center}

\renewcommand{\thesection}{S\arabic{section}}
\renewcommand{\theequation}{S\arabic{equation}}
\renewcommand{\thefigure}{S\arabic{figure}}
\renewcommand{\thetable}{S\arabic{table}}

\begin{description}
\item[\ref{app:exbenef} Examples of benefits of $\hat S^{ei}$:]
Illustrations in four examples of sparsity and informativeness of the blocks of the benefits of $\hat S^{ei}$ in oracle selection properties.
\item[\ref{app:fdrpowersim} Complementary material on numerical illustrations:]
Reports of false discovery rates, power, mean squared error and computation time in the numerical illustrations of the performance of $\hat S^{ei}$ and Tran-s-ell0.
\item[\ref{app:proofs} Proofs]
\end{description}

\setcounter{section}{0}
\setcounter{equation}{0}
\setcounter{figure}{0}
\setcounter{table}{0}

\section{Examples of benefits of $\hat S^{ei}$}\label{app:exbenef}

We illustrate here the benefits of $\hat S^{ei}$ over $\hat S^{sd}$ in four concrete examples contemplating different scenarios of sparsity regimes and informativeness of the blocks, summarized in Table~\ref{tab:examples}. We assume an orthonormal design such that $\rho(\bX)=1$ for simplicity and focus on a sparse setting where $\ln(s)=o(\ln(p-s))$, with two blocks ($b=2$). The smallest non-zero $|\beta_i^*|$, $\beta_{\min}^*$, is assumed to be located in block $2$, without loss of generality. 

\begin{table*}[ht]
\caption{Sparsity and block assumptions in Scenarios 1--4}
\label{tab:examples}
\begin{center}
\begin{tabular}[]{@{}cccccc@{}}
\hline
Scenario & $p-s$ & $p_1-s_1$ & $p_2-s_2$ & $s_1$ 
& $s_2$\\
\hline
$1$    & $3n/2$ & $3n/2-\sqrt{n}$ & $\sqrt{n}$  & $3\ln(n)/2$  & $3\ln(n)/2$ \\
$2$    & $\exp(n/16)$ & $\exp(n/16)-n^2$ & $n^2$  & $3\ln(n)/2$  & $3\ln(n)/2$   \\
$3$    & $n$ & $n-\ln(n)$ & $\ln(n)$  & $3\ln(n)/2$  & $3\ln(n)/2$   \\
$4$    & $n$ & $n/2$ & $n/2$  & $3\ln(n)/2$  & $3\ln(n)/2$\\
\hline
\end{tabular}
\end{center}
\end{table*}

\begin{figure}[ht]
\begin{tabular}{M{70mm}M{70mm}}
        Scenario 1& Scenario 2\\
         \includegraphics{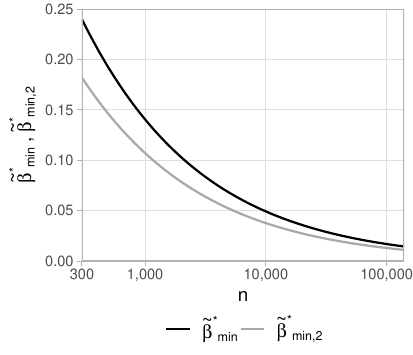} & \includegraphics{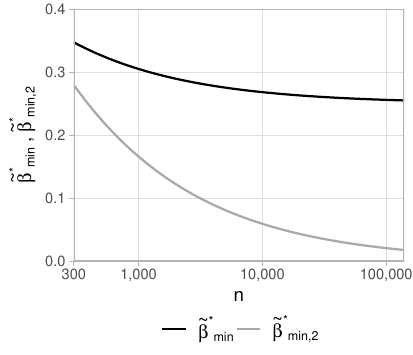} \\
        Scenario 3& Scenario 4\\
         \includegraphics{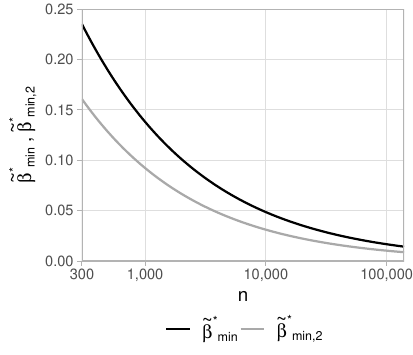} & \includegraphics{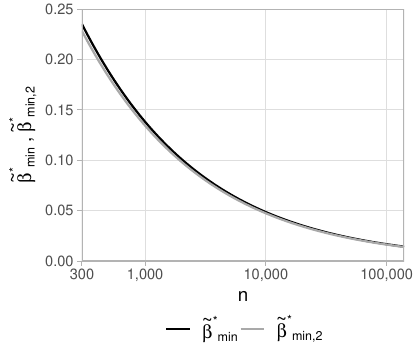} \\
\end{tabular}
\caption{Smallest signal strength $\tilde{\beta}^*_{\min,2}$ and $\tilde{\beta}^*_{\min}$ in Scenario 1 (top left), 2 (top right), 3 (bottom left), and 4 (bottom right) as in Table~\ref{tab:examples} (Section~\ref{app:exbenef})}
\label{fig:boundbetamin}
\end{figure}

\begin{figure}[ht]
\begin{center}
\includegraphics{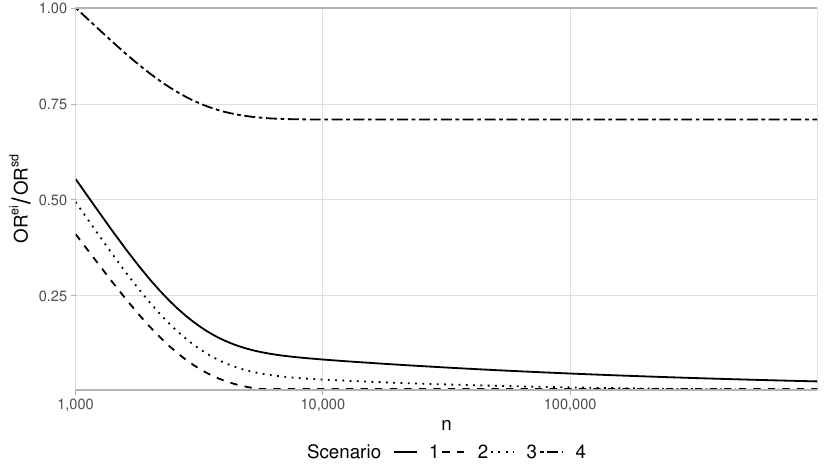}    
\end{center}
\caption{Ratio $OR^{ei}/OR^{sd}$ in Scenario 1 (solid), 2 (dashed), 3 (dotted), and 4 (two-dashed) as in Table~\ref{tab:examples} (Section~\ref{app:exbenef})}
\label{fig:ratioorconvrate}
\end{figure}

In Scenarios 1 to 3, blocks are informative as block $B_1$ is sparser than block $B_2$. In Scenario 1 the blocks are moderately informative, in that $[s_2/(p_2-s_2)]/[s_1/(p_1-s_1)]= O(\sqrt{n})$. In Scenario 2 they are highly discriminative, in that $[s_2/(p_2-s_2)]/[s_1/(p_1-s_1)]= O(\exp(n/16))$. In Scenario 3, they are highly discriminative in that $\bbeta^*$ is sparse overall, but it is non-sparse within block 2 ($s_2 > p_2-s_2)$. In Scenario 4, the block partition is non-informative of the distribution of inactive parameters: each block has half of them. 

Recall that the right-hand sides in \eqref{eq:betaminblockpenex} and \eqref{eq:betaminpenex} are the smallest signal strength allowed by the sufficient assumptions for consistency (\hyperlink{lab:A1}{A1}–\hyperlink{lab:A2}{A2}) with $\hat S^{ei}$ and $\hat S^{sd}$ respectively. Denote them $\tilde{\beta}^*_{\min,j}$ and $\tilde{\beta}^*_{\min}$ respectively, that is, for the informed selector $\hat S^{ei}$,
\begin{equation*}
\tilde{\beta}^*_{\min,j}
=
\frac{\sqrt{\ln(p_j-s_j)}+\sqrt{\ln(s_j)}}{\sqrt{\rho(\bX)n}}\quad\text{for each }j=1,\ldots,b,
\end{equation*}
and for the standard $\hat S^{sd}$,
\begin{equation*}
\tilde{\beta}^*_{\min}
=
\frac{\sqrt{\ln(p-s)}+\sqrt{\ln(s)}}{\sqrt{\rho(\bX)n}}.
\end{equation*}
Figure~\ref{fig:boundbetamin} plots $\tilde{\beta}^*_{\min,2}$ and $\tilde{\beta}^*_{\min}$ in Scenarios 1--4. The smallest signal strength allowed for $\hat S^{ei}$ in block 2 is smaller than that allowed for $\hat S^{sd}$, with the magnitude of the difference depending on the informativeness of the block partition. It is the largest in the highly informative Scenario 2 and essentially null in the non-informative Scenario 4. Scenario 2 also provides a straightforward example where incorporating external information makes possible consistent recovery. For any constant $\beta_{\min}^*=\beta_{\min,2}^*<0.25$, consistent recovery is impossible with $\hat S^{sd}$, but possible with $\hat S^{ei}$. 

We compare next the convergence rates associated with $\hat S^{ei}$ and $\hat S^{sd}$ for penalties chosen by an oracle, that is, ${OR}^{ei}$ and $OR^{sd}$ respectively (see Theorem~\ref{prop:fullrecovreg}). Figure~\ref{fig:ratioorconvrate} plots the ratio ${OR}^{ei}/OR^{sd}$ as given in~\eqref{eq:fewkn} in Scenario 1 to 4 where $\beta_{\min }^*=\beta_{\min,2}^*=1/2$ and  $\beta_{\min,1}^*=1.3\beta_{\min }^*$ to guarantee asymptotic recovery is possible in all four scenarios. Recall that $\gamma^*\in(1/2,1)$ and $\delta>0$ is arbitrarily close to 1. We set $\gamma^*=3/4$ and $\delta=0.99$ for illustration purposes. Changes in $\gamma^*$ and $\delta$ impact how much the ratios differ between scenarios but not their asymptotic behavior. In the informative Scenarios 1 to 3 convergence with $\hat S^{ei}$ is much faster than with $\hat S^{sd}$ and the ratios vanish. The vanishing is the fastest in Scenario 3, followed by 2 and 1. In Scenario 4, the ratio converges to a constant just below 0.75. In that scenario, where the partition is non-informative of the distribution of inactive parameters, with $\hat S^{ei}$, the oracle is still capable to leverage the difference in signal strength between blocks.

\section{Complementary material on numerical illustrations}\label{app:fdrpowersim}

\begin{figure}[ht]
\begin{tabular}{M{49mm}M{49mm}M{49mm}}
         Sc.1 FDR& Sc. 1 Power & Sc.1 Est. MSE\\
         \includegraphics{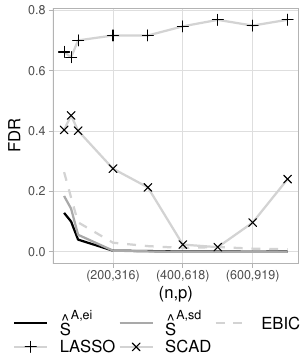}  & 
         \includegraphics{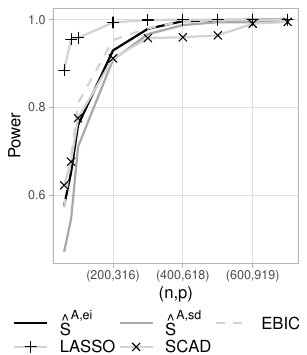} & 
         \includegraphics{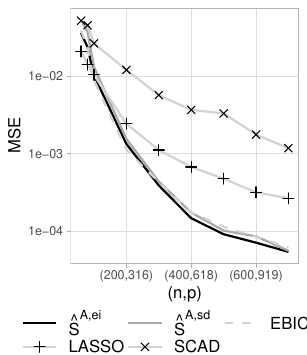} \\
         Sc.3 FDR& Sc.3 Power & Sc.3 Est. MSE\\
         \includegraphics{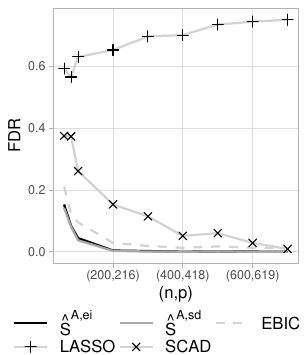}  & 
         \includegraphics{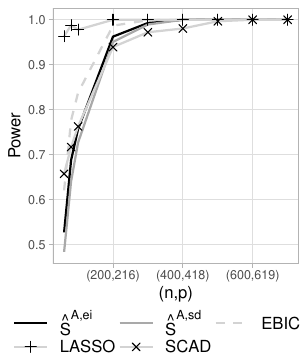} & 
         \includegraphics{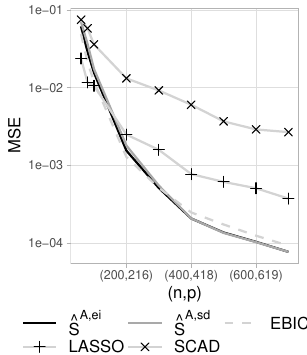} 
\end{tabular}
\caption{FDR (left), power (middle) and estimation MSE (right) with $\hat S^{A,ei}$, $\hat S^{A,sd}$, EBIC, LASSO and SCAD in Scenario 1 (top) and 3 (bottom) as in Table~\ref{tab:examples_bis} (Section~\ref{sec:simulationsss})}
\label{fig:linregfdrpowmsesc13}
\end{figure}

Figure~\ref{fig:linregfdrpowmsesc13} shows the average FDR, power and mean squared error in Scenario 1 and 3 of the simulations comparing the adaptive informed selector $\hat S^{A,ei}$ to other selection methods in Section~\ref{sec:simulationsss}. Recall, Scenario 1 is in informative block partition setting and Scenario 3 a non-informative block partition setting. See Table~\ref{tab:examples_bis} for a complete description of the scenarios. The results in Scenario 1 confirm the findings in Scenario 2 discussed in Section~\ref{sec:simulationsss}. That is, when the block partition is informative and compared to the uninformed $\hat S^{A,sd}$, $\hat S^{A,ei}$ has equal or lower FDR and higher power. LASSO and EBIC have higher power than $\hat S^{A,ei}$ but higher FDR. In terms of mean squared error, $\hat S^{A,ei}$, $\hat S^{A,sd}$ and EBIC perform similarly, and better than SCAD and LASSO. The results in Scenario 3 show that when the block partition is non-informative, $\hat S^{A,ei}$ and $\hat S^{A,sd}$ perform very similarly. 

\begin{figure}[ht]
\begin{tabular}{M{49mm}M{49mm}M{49mm}}
         Sc.1 Runtime & Sc. 2 Runtime & Sc.3 Runtime\\
         \includegraphics{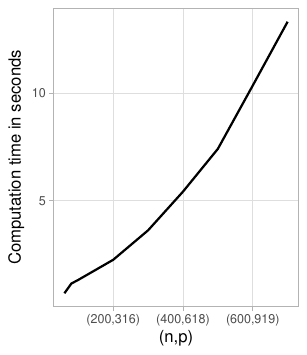}  & 
         \includegraphics{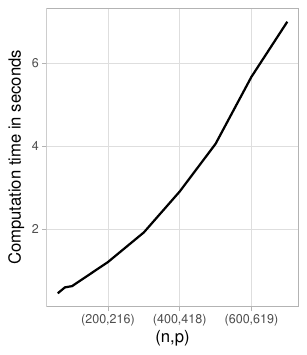} & 
         \includegraphics{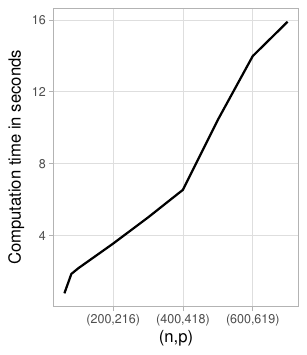} 
\end{tabular}
\caption{Average runtime of Algorithm 1 to obtain $\hat S^{A,ei}$ in Scenario 1 (left), 2 (middle) and 3 (right) as in Table~\ref{tab:examples_bis} (Section~\ref{sec:simulationsss})}
\label{fig:runtime}
\end{figure}

Figure~\ref{fig:runtime} shows the average runtime of Algorithm 1 to obtain $\hat S^{A,ei}$ in Scenario 1--3 as a function of $(n,p)$ (Windows laptop, Intel Core Ultra i7, 32Gb RAM). Runtime increased with the dimension but remained below 20 seconds in all the asymptotic regimes considered.

\begin{figure}[ht]
\begin{tabular}{M{49mm}M{49mm}M{49mm}}
         Sc.1 FDR& Sc. 1 Power & Sc.1 Est. MSE\\
         \includegraphics{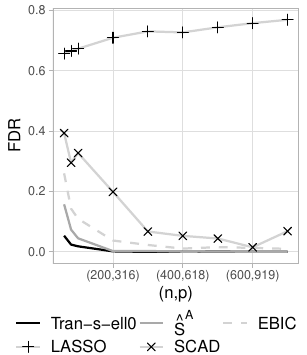}  & 
         \includegraphics{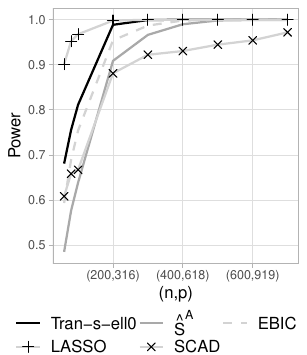} & 
         \includegraphics{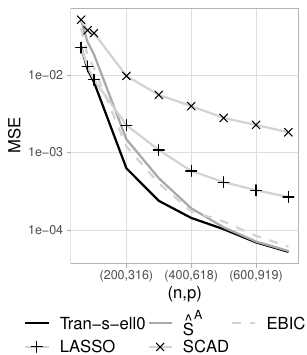} \\
         Sc.2 FDR& Sc.2 Power & Sc.2 Est. MSE\\
         \includegraphics{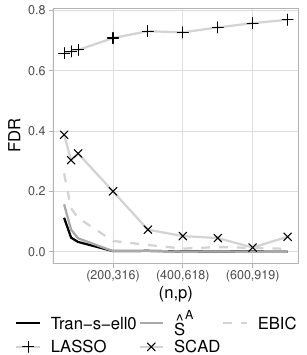}  & 
         \includegraphics{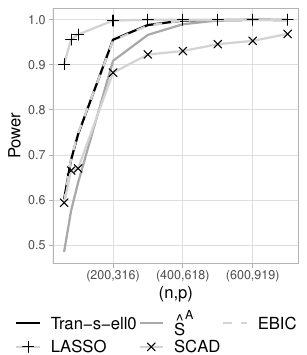} & 
         \includegraphics{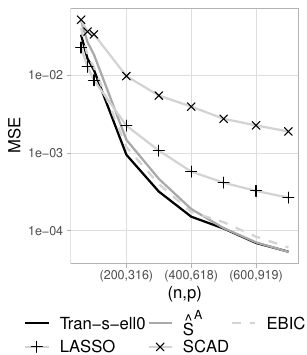}  \\
         Sc.3 FDR& Sc.3 Power & Sc.3 Est. MSE\\
         \includegraphics{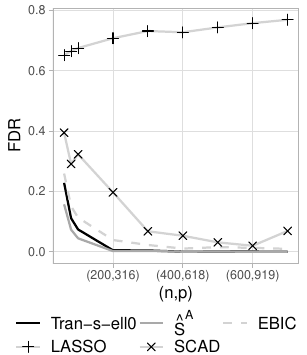}  & 
         \includegraphics{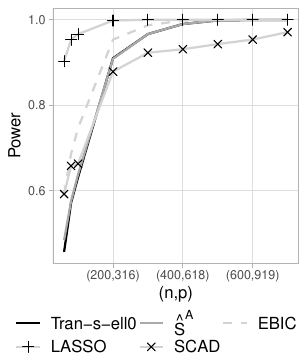} & 
         \includegraphics{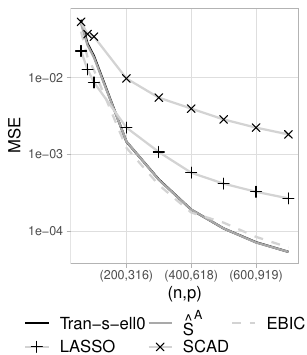} 
\end{tabular}
\caption{FDR (left), power (middle) and estimation MSE (right) with Tran-s-ell0, $\hat S^{A,sd}$, EBIC, LASSO and SCAD in Scenario 1 (top), 2 (middle) and 3 (bottom) as in Section~\ref{sec:TLexample}}
\label{fig:TLfdrpowmsesc123}
\end{figure}

Figure~\ref{fig:TLfdrpowmsesc123} shows the average FDR, power and mean squared error in Scenario 1 to 3 of the simulations comparing Trans-s-ell0 to other selection methods in Section~\ref{sec:TLexample}. Recall, Scenario 1 and 2 are examples of relevant transfer with respectively complete and moderate overlap of source and target supports. Scenario 3 is an example of irrelevant transfer with no overlap between source and target supports. Again, when the transfer is relevant and compared to the uninformed $\hat S^{A,sd}$, Trans-s-ell0 has equally low or lower FDR and higher power. Compared to EBIC, Trans-s-ell0 has higher power in Scenario 1 and similar power in Scenario 2 and has lower FDR in both scenarios. Compared to LASSO, Trans-s-ell0 has lower power but lower FDR in both Scenario 1 and 2. In terms of MSE, Trans-s-ell0 outperforms all the other selection methods in both Scenario 1 and 2. In Scenario 3, Trans-s-ell0 performs very similarly to $S^{A,sd}$ except it has slightly more FDR for small sample sizes.

\section{Proofs}\label{app:proofs}

Section~\ref{suppsec:auxiliary} presents results auxiliary to the rest of the proofs, as well as the proofs of that auxiliary results. Section~\ref{suppsec:proofseclinreg} the proofs of Section~\ref{sec:properties_legression} and Section~\ref{suppsec:proofsempbayes} the proofs of Sections~\ref{sec:informed_l0} and~\ref{sec:TL}. 

\subsection{Auxiliary results}\label{suppsec:auxiliary}

In this section we collect some technical results auxiliary to the rest of the proofs.

\subsubsection{Tail bounds}

\begin{lem}\label{lemma:3tailbounds}
If $y\sim N(\mu,\sigma^2)$ and $a\leq|\mu|$, $$P(|y|>a)  = P\left(z<\frac{|\mu|-a}{\sigma}\right) + P\left(z<-\frac{a+|\mu|}{\sigma}\right)$$
\end{lem}

\begin{proof} We have
$$P(|y|>a) = P(y>a) + P(y<-a) = P\Big(z>\frac{a-\mu}{\sigma}\Big) + P\Big(z<-\frac{a+\mu}{\sigma}\Big)$$
where $z\sim N(0,1)$. If $\mu>0$, $P(z>\frac{a-\mu}{\sigma})  =  P(z>\frac{a-|\mu|}{\sigma})= P(z<\frac{|\mu|-a}{\sigma})$ by symmetry of the standard Gaussian, and $P(z<-\frac{a+\mu}{\sigma})=P(z<-\frac{a+|\mu|}{\sigma})$, then $P(|y_i|>a) =  P(z<\frac{|\mu|-a}{\sigma}) + P(z<-\frac{a+|\mu|}{\sigma})$.
If $\mu\leq0$, then $P(z>\frac{a-\mu}{\sigma})=P(z>\frac{a+|\mu|}{\sigma})=P(z<-\frac{a+|\mu|}{\sigma})$, by symmetry of the standard Gaussian, and $P(z<-\frac{a+\mu}{\sigma})=P(z<\frac{|\mu|-a}{\sigma})$. Then for any $\mu$, $P(|y_i|>a)  = P(z<\frac{|\mu|-a}{\sigma}) + P(z<-\frac{a+|\mu|}{\sigma}) $.
\end{proof}

\begin{lem}\label{lemma:s1s3}
Let $W\sim\chi_{\nu}^2(\mu)$ with $\mu \geq 0$, then for any $w>\mu+\nu$
$$P(W > w) \leq e^{-\big(\frac{w+\mu}{2}  - \sqrt{2w( 2 \mu+ \nu)-2 \mu \nu - \nu^2 } \big)}. $$
Moreover, assume $w$, $\nu$ and $\mu$ are functions of $n$ such that $w$ is increasing, $\nu=o(w)$, and $\mu=o(w)$. Then, for any $\phi\in(0,1)$ and $n$ large enough
$$P(W > w)  \leq  e^{-\phi\frac{w}{2}}.$$
\end{lem}

\begin{proof}
By~\cite{birge}, Lemma 8.1 we have that for any $x>0$
$$P(W> (\nu+\mu)+2 \sqrt{(\nu+2 \mu) x}+2 x) \leq e^{-x}.$$
The function $f:x\mapsto (\nu+\mu)+2 \sqrt{(\nu+2 \mu) x}+2 x$ is  one-to-one between $\R^+$ and $( \nu+
\mu,\infty)$. 
Hence, we have that for any $w>\mu+\nu$,
$$P(W > w) \leq e^{-f^{-1}(w)} \:=\; e^{-\big(\frac{w+\mu}{2}  - \frac{1}{2}\sqrt{2w(2 \mu+\nu)-2 \mu \nu + \nu^2 }\big) } .$$
Observe that
$$\frac{w+\mu}{2}  - \frac{1}{2}\sqrt{2w( 2 \mu+ \nu)-2 \mu \nu - \nu^2 }= \frac{w}{2}\bigg(1+\frac{\mu}{w} -\frac{1}{2}\sqrt{\frac{8(2 \mu+\nu)}{w}\bigg(1-\frac{2\nu\mu-\nu^2}{2(2w\mu+w\nu)} \bigg)}\bigg).$$
 Since $\nu=o(w)$ and $\mu=o(w)$ by assumption, we have $\frac{w+\mu}{2}  - {2}^{-1}\sqrt{2w( 2 \mu+ \nu)-2 \mu \nu - \nu^2 }= \frac{w}{2}(1+o(1))$. Therefore, for any $\phi\in(0,1)$ and every $n$ large enough.
$$P(W > w) \leq  e^{-\phi\frac{w}{2}}.$$
\end{proof}

\begin{lem}\label{lemma:s2} Let $W \sim \chi_\nu^2(\mu)$ with $\mu>0$. For any $w<\mu$, 
$$
P(W<w) \leq \frac{e^{-\frac{1}{2}(\sqrt{\mu}-\sqrt{w})^2}}{(\mu / w)^{\nu / 4}}.$$
\end{lem}

\begin{proof}
The result follows directly from~\cite{rossell2022concentration}, Lemma S2.
\end{proof}

\begin{lem}\label{lemma:s18} 
    Let $W\sim \chi_\eta^2(\mu)$ with $\mu \geq 0$. Assume that $h$, $c$, $\eta$ and $\mu$ are functions of $n$ such that $h$ is positive and increasing, $c$ is positive with $\lim_{n\to\infty}c=l \in(0,\infty)$, $\eta=o(c\ln(h))$, and $\mu=o(c\ln(h))$. Let $\bar{u},\underline{u}$ in $(0,1)$ such that $1>\bar{u}>\underline{u}\geq\left(1+h^{\frac{\psi c}{2}}e^{-(\eta+\mu)}\right)^{-1}$ where $\psi\in(0,\min\{1,2/l\})$, then for every $n$ large enough, we have 
$$
\int_{\underline{u}}^{\bar{u}} P\left(W>c\ln \left(\frac{h}{1 / u-1}\right)\right) d u \leq  \frac{1}{h^{\frac{\psi c}{2}}}\Big(\bar{u}-\underline{u}+\ln\Big(\frac{\bar{u}}{\underline{u}}\Big)\Big).
$$
\end{lem}

\begin{proof}
For any $u\in [\underline{u},\bar{u}]$, we have $c\ln \left(\frac{h}{1 / u-1}\right) \geq c\ln \left(\frac{h}{1 / \underline{u}-1}\right)$. Since $\underline{u}\geq (1+h^{\frac{\psi c}{2}}e^{-(\eta+\mu)})^{-1}$ by assumption  we also have that, for any $u\in (\underline{u},\bar{u})$,
$$c\ln \left(\frac{h}{1 / u-1}\right) \geq c\big(1-\frac{\psi c}{2}\big)\ln(h)+c(\eta+\mu).$$
It follows,
$$\frac{\eta}{c\ln \left(\frac{h}{1 / u-1}\right)}\leq \frac{\eta}{c\big(1-\frac{\psi c}{2}\big)\ln(h)+c(\eta+\mu)}
\quad\text{and}\quad
\frac{\mu}{c\ln \left(\frac{h}{1 / u-1}\right)}\leq\frac{\mu}{ c\big(1-\frac{\psi c}{2}\big)\ln(h)+c(\eta+\mu)}.$$
By assumption $\eta=o\left(c\ln (h)\right)$ and $\mu=o\left(c\ln (h)\right)$, then for any $u\in (\underline{u},\bar{u})$, $\eta=o\left(c\ln \left(\frac{h}{1 / u-1}\right)\right)$ and $\mu=o\left(c\ln \left(\frac{h}{1 / u-1}\right)\right)$. By Lemma~\ref{lemma:s1s3} with $\phi=\psi\in(0,1)$, for every $n$ large enough,
\begin{equation}\label{eq:randomm}
    \int_{\underline{u}}^{\bar{u}} P\left(W>c\ln \left(\frac{h}{1 / u-1}\right)\right) d u< \frac{1}{h^{\frac{\psi c}{2}}} \int_{\underline{u}}^{\bar{u}}(1 / u-1)^{\frac{\psi c}{2}} d u .
\end{equation}
Applying the change of variables $v=1 / u-1$ to the integral on the right-hand side above gives
\begin{equation}\label{eq:rpoig}
    \int_{\underline{u}}^{\bar{u}}(1 / u-1)^{\frac{\psi c}{2}} d u=\int_{1 / \bar{u}-1}^{1 / \underline{u}-1} \frac{v^{\frac{\psi c}{2}}}{(v+1)^2} d v.
\end{equation}
Rewrite $v^{\frac{\psi c}{2}}=(v-1 + 1)^{\frac{\psi c}{2}}$ where, since $\psi<2/l$ by assumption, $\psi c/2\leq1$ for $n$ large enough. Since $\bar{u}<1$, we have that for any $v>1 / \bar{u}-1$, $v-1>-1$. Note that for any $x \geq-1$ and $r\in[0,1]$ $(1+x)^r \leq 1+r x$. Then, for any $v>1 / \bar{u}-1$, $v^{\frac{\psi c}{2}}=(v-1 + 1)^{\frac{\psi c}{2}}\leq 1 + \frac{\psi c}{2} (v-1)\leq 1 + \frac{\psi c}{2}(v+1)$. Applying this last inequality to the right-hand side in~\eqref{eq:rpoig} gives
\begin{equation}\label{eq:random2}
\int_{\underline{u}}^{\bar{u}}(1 / u-1)^{\frac{\psi c}{2}} d u<\int_{1 / \bar{u}-1}^{1 / \underline{u}-1} \frac{1}{(v+1)^2} +\frac{\psi c}{2(v+1)} d v = \bar{u}-\underline{u}+\frac{\psi c}{2}\ln\Big(\frac{\bar{u}}{\underline{u}}\Big).
\end{equation}
The result follows inputting the bound from~\eqref{eq:random2} in~\eqref{eq:randomm} and using that $\psi c/2\leq1$ for every $n$ large enough and $\ln(\bar{u}/\underline{u}) \geq 0$ ($\underline{u} \leq \bar{u}$).
\end{proof}

\begin{lem}\label{lemma:s20}
 Let $W\sim \chi_\eta^2(\mu)$ with $\mu \geq 0$. Assume that $h$, $c$, $\eta$ and $\mu$ are functions of $n$ such that $h$ is positive and increasing, $c$ is positive such that $\lim_{n\to\infty}c = l \in (0,\infty)$, $\eta=o(\ln(h))$, and $\mu=o(\ln(h))$. Then for any $\alpha\in(0,\min\{1,l/2\})$ and every large enough $n$, we have
$$
\int_0^1 P\left(W>c \ln \left(\frac{h}{1 / u-1}\right)\right) d u = o\big(  h^{-\alpha}\big).$$
\end{lem}

\begin{proof}
Since a probability is bounded by $1$, for any $a\in(0,1)$,
\begin{equation}\label{eq:basebound}
    \int_0^1 P\left(W>c \ln \left(\frac{h}{1 / u-1}\right)\right) d u \;\leq\; 2 a+\int_{a}^{1-a} P\left(W>c\ln\left(\frac{h}{1 / u-1}\right)\right) d u.
\end{equation}

Take $a=\left(1+h^{\frac{\psi c}{2}}e^{-(\eta+\mu)}\right)^{-1}$ where $\psi\in(2\alpha/l,\min\{1,2/l\})$. Since, by assumption, $\eta=o(c\ln(h))$ and $\mu=o(c\ln(h))$, by Lemma~\ref{lemma:s18} with $\underline{u}=a$ and $\bar{u}=1-a$, we have
$$
\int_0^1 P\left(W>c \ln \left(\frac{h}{1 / u-1}\right)\right) d u \;\leq\; \frac{2}{1+h^{\frac{\psi c}{2}}e^{-(\eta+\mu)}}+ 
\frac{1-2a+\ln(h^{\frac{\psi c}{2}}e^{-(\eta+\mu)})}{h^{\frac{\psi c}{2}}}. $$
Since $1+h^{\frac{\psi c}{2}}e^{-(\eta+\mu)}>h^{\frac{\psi c}{2}}e^{-(\eta+\mu)}$ and $1-2a-\eta+\mu\leq \ln\big(h^{\frac{\psi c}{2}}\big)$ for every $n$ large enough, we have
$$
\int_0^1 P\left(W>c \ln \left(\frac{h}{1 / u-1}\right)\right) d u \;\leq\; h^{-\frac{\psi c}{2}}\big(2e^{\eta+\mu}+2\ln(h^{\frac{\psi c}{2}})\big) $$
We have that
\begin{equation}\label{eq:random24}
    \frac{h^{-\frac{\psi c}{2}}2e^{\eta+\mu}}{h^{-\alpha}} 
= e^{-\big(\frac{\psi c}{2}-\alpha\big) \ln(h) + \ln(2) + \eta + \mu}
    \;=\;e^{-\big(\frac{\psi c}{2}-\alpha\big)\ln(h)\big(1-\frac{\ln(2)}{(\psi c/2-\alpha)\ln(h)}-\frac{\eta+\mu}{(\psi c/2-\alpha)\ln(h)}\big)}
\end{equation}
and similarly that
\begin{equation}\label{eq:random25}
    \frac{h^{-\frac{\psi c}{2}}2\ln(h^{\frac{\psi c}{2}})}{h^{-\alpha}} \;=\;e^{-\big(\frac{\psi c}{2}-\alpha\big)\ln(h)\big(1-\frac{\ln(\psi c\ln(h)/2)}{(\psi c/2-\alpha)\ln(h)}- \frac{\ln(2)}{(\psi c/2-\alpha)\ln(h)}\big)}.
\end{equation}
Since $\psi>2\alpha/l$, for $n$ large enough $\alpha<\psi c/2$, and by assumption $h$ is increasing, $\eta=o(\ln(h))$ and $\mu=o(\ln(h))$, both expressions in~\eqref{eq:random24} and~\eqref{eq:random25} then 
vanish as $n$ grows. Hence, 
$$\int_0^1 P\left(W>c \ln \left(\frac{h}{1 / u-1}\right)\right) d u \;=\; o(h^{-\alpha}).$$
\end{proof}

\subsubsection{Bounds related to model normalized scores}

Let $NC(M)$ be the normalized score for model $M$ defined in~\eqref{eq:CM}, $\M$ the set of models under consideration. Lemma~\ref{lem:L0toL1convergencegen} generalizes Lemma~\ref{lem:L0toL1convergence} and shows that the probability of not selecting a set of models is bounded above by the expected sum of the normalized scores of the models outside the set.  Lemma~\ref{lemma:s7} shows the difference in squared errors for two nested models is chi-squared distributed. Lemma~\ref{lem:boundrhogen} and Lemma~\ref{lem:upperboundnoncentralgen} give bounds on the noncentrality parameter of that chi-squared distribution. Lemma~\ref{lem:boundexpnc} provides a bound on the expected normalized score of any $M \in \M$ based on the pairwise comparison $C(T)-C(M)$ (\emph{cf}~\eqref{eq:CM}) for some $T\subseteq S$. It essentially shows that, if the block penalties diverge and the signals in $M\setminus T$ are small, $NC(M)$ is small. Prior to stating Lemma~\ref{lem:boundexpnc} we outline the proof strategy.

\begin{lem}\label{lem:L0toL1convergencegen} For $\hat{S}^{ei}$ as in~\eqref{eq:Mhat} and any $k$ models $M^1,\ldots,M^k$
$$P(\hat{S}^{ei}\not\in \{M^1,\ldots,M^k\})\;\leq\; (k+1) \sum_{M \in \M\setminus \{M^1,\ldots,M^k\}}\E\left( NC(M)\right).$$
\end{lem}

\begin{proof}
Suppose that $NC(M^1)+\ldots+NC(M^k)>\frac{k}{k+1}$ then for any $M\not\in \{M^1,\ldots,M^k\}$, we have that
$$NC(M)= 1-\sum_{M' \neq M} NC(M') < 1 - \sum_{M' \in \{M^1,\ldots,M^k\}} NC(M') < \frac{1}{k+1}.$$ 
In addition, if $NC(M^1)+\ldots+NC(M^k)>\frac{k}{k+1}$ then necessarily $\max_{l=1,\ldots,k}NC(M_l)> \frac{1}{k+1} > NC(M) $ for any $M \not\in \{M^1,\ldots,M^k\}$, and therefore $\hat{S}^{ei}\in \{M^1,\ldots,M^k\}$. Consequently, 
$$
P(\hat{S}^{ei}\not\in \{M^1,\ldots,M^k\})\; \leq\; P\Big(NC(M^1)+\ldots+NC(M^k)\leq\frac{k}{k+1}\Big).$$
Moreover, we have
$$P\Big(NC(M^1)+\ldots+NC(M^k)\leq\frac{k}{k+1}\Big)\;=\;P\Big(\sum_{M \in \M\setminus \{M^1,\ldots,M^k\}} NC(M)\geq \frac{1}{k+1}\Big).$$
The result follows from the Markov's inequality applied to the right-hand side above.
\end{proof}

\begin{lem}\label{lemma:s7}
    Let $M,Q$ be any two nested models such that $M \subseteq Q$. Then
$$
L_{QM}\;=\;\|\bX_{Q}\tilde{\bbeta}^{(Q)}\|^2-\|\bX_{M}\tilde{\bbeta}^{(M)}\|^2\;\sim\; \chi_{|Q\setminus M|}^2\left(\mu_{Q M}\right), 
$$
where $\chi_{k}^2(\mu)$ denotes, when $\mu>0$, the noncentral chi-squared distribution with $k$ degrees of freedom and noncentrality parameter $\mu$ and, when $\mu=0$, the chi-squared distribution with $k$ degrees of freedom $\chi_{k}^2$. The parameter $\mu_{QM}$ is given by \begin{equation}\label{eq:muQM}
    \mu_{QM}\;\;:=\;\;\|\big(I_n-P_M\big) \bX_{Q\setminus M} \bbeta^*_{Q\setminus M}\|^2
\end{equation}
where $P_M=\bX_{M}\left(\bX_{M}^{\top} \bX_{M}\right)^{-1} \bX_{M}^{\top}$.
\end{lem}
\begin{proof}
    This result follows directly from Lemma~S7 in~\cite{rossell2022concentration} taking $\phi^*=1$.
\end{proof}

\begin{lem}\label{lem:boundrhogen}
 For any $T\subseteq S$ and $M\in\M$ such that $T\not\subseteq M$, let $Q_T=M \cup T$. The non-centrality parameter defined in~\eqref{eq:muQM} satisfies: 
\begin{equation}
    \mu_{Q_TM}\; \geq \;n\,\rho(\bX)\,\sum_{j=1}^b|T_j\setminus M_j|\,\min_{i\in T_j\setminus M_j}{\beta_{i}^*}^2.
\end{equation}
\end{lem}

\begin{proof} 
The non-centrality parameter $\mu_{Q_TM}$, as defined in~\eqref{eq:muQM}, satisfies
\begin{eqnarray*}
    \mu_{Q_TM}&=&\|\big(I_n-P_M\big) \bX_{Q_T\setminus M} \bbeta^*_{Q_T\setminus M}\|^2\\
    &=&\|\big(I_n-P_M\big) \bX_{T\setminus M} \bbeta^*_{T\setminus M}\|^2\\
    &= & n {\beta^*_{T\setminus M}}^\top \left(\tfrac1n \bX_{T\setminus M}^\top(I_n-P_M)\bX_{T\setminus M}\right)\beta^*_{T\setminus M}\\
    &\geq & n\lambda_{\min}\left(\tfrac1n \bX_{T\setminus M}^\top(I_n-P_M)\bX_{T\setminus M}\right)\|\beta^*_{T\setminus M}\|^2,
\end{eqnarray*}
where the second equality follows from observing that $Q_T\setminus M=T\setminus M$. Since $T$ is a subset of $S$, by reordering columns, $X_{S\setminus M}=[X_{T\setminus M},X_{S\setminus (M\cup T)}]$ and therefore $\tfrac1n \bX_{T\setminus M}^\top(I_n-P_M)\bX_{T\setminus M}$ is a principal submatrix of $\tfrac1n \bX_{S\setminus M}^\top(I_n-P_M)\bX_{S\setminus M}$. Hence, Cauchy's interlacing theorem gives that $\lambda_{\min}\left(\tfrac1n \bX_{T\setminus M}^\top(I_n-P_M)\bX_{T\setminus M}\right) \geq \lambda_{\min}\left(\tfrac1n \bX_{S\setminus M}^\top(I_n-P_M)\bX_{S\setminus M}\right)$. Finally, by definition of $\rho(\bX)$ in Assumption~\hyperlink{lab:A2}{(A2)} we have that
$\lambda_{\min}\left(\tfrac1n \bX_{S\setminus M}^\top(I_n-P_M)\bX_{S\setminus M}\right) \geq \rho(\bX)$, and further noting that $\|\beta^*_{T\setminus M}\|^2\geq \sum_{j=1}^b|T_j\setminus M_j|\,\min_{i\in T_j\setminus M_j}{\beta_{i}^*}^2$ gives the desired result.
\end{proof}

\begin{lem}\label{lem:upperboundnoncentralgen}
    For any $T\subseteq S$ and $M\in \M$, let $Q_T=M\cup T$, and $\mu_{Q_T T}$ as defined in~\eqref{eq:muQM}, then, 
    \begin{equation}\label{eq:upperbounnnnnd}
      \mu_{Q_TT} \leq n\,\bar{\lambda}\,\sum_{j=1}^b|(S_j\cap M_j )\setminus T_j|\,\max_{i\in (S_j\cap M_j )\setminus T_j}{\beta_{i}^*}^2 \quad \text{where}\quad \bar{\lambda}\,:=\,\lambda_{\max}\Big(\frac1n \bX_{S}^\top\bX_{S}\Big)  
    \end{equation}
    and
    \begin{equation}\label{eq:lowerbounnnnnd}
    \mu_{Q_TT} \geq n\,\lambda_{\min}\Big(\frac1n \bX_{S}^\top\bX_{S}\Big)\,\sum_{j=1}^b|(S_j\cap M_j )\setminus T_j|\,\min_{i\in (S_j\cap M_j )\setminus T_j}{\beta_{i}^*}^2
    \end{equation}
\end{lem}

\begin{proof}
Using the definition of $\mu_{Q_T T}$ in~\eqref{eq:muQM}, we have that 
\begin{eqnarray}
    \mu_{Q_TT}&=&{\bbeta^*_{Q_T\setminus T}}^{\top}\bX_{Q_T\setminus T}^{\top}  \big(I_n-P_T\big) \bX_{Q_T\setminus T} \bbeta^*_{Q_T\setminus T}\nonumber\\
    &=&{\bbeta^*_{M\setminus T}}^{\top}\bX_{M\setminus T}^{\top}  \big(I_n-P_T\big) \bX_{M\setminus T} \bbeta^*_{M\setminus T}\nonumber\\
    &=&{\bbeta^*_{(S\cap M)\setminus T}}^{\top}\bX_{(S\cap M)\setminus T}^{\top}  \big(I_n-P_T\big) \bX_{(S\cap M)\setminus T} \bbeta^*_{(S\cap M)\setminus T}\label{eq:poihe}
\end{eqnarray}
where the second equality follows from $Q_T\setminus T =M\setminus T $ and the third equality from $\bbeta^*_{M\setminus S}=0$.
We start by showing the upper bound in~\eqref{eq:upperbounnnnnd}. Denote for any square matrix $A$, its largest eigenvalue 
$\lambda_{\max}(A)$. By~\eqref{eq:poihe}, we have that
$$\mu_{Q_TT}\leq n\lambda_{\max}\Big(\frac{1}{n}\bX_{(S\cap M)\setminus T}^{\top}  \big(I_n-P_T\big) \bX_{(S\cap M)\setminus T}\Big) \|\bbeta^*_{(S\cap M)\setminus T}\|^2.$$
Let $B:=\tfrac{1}{n}\bX_{(S\cap M)\setminus T}^{\top}  \big(I_n-P_T\big) \bX_{(S\cap M)\setminus T}$, $C:=\tfrac{1}{n}\bX_{(S\cap M)\cup T}^{\top}\bX_{(S\cap M)\cup T}$ and $D:=\tfrac{1}{n}\bX_T^{\top} \bX_T$. $D$ is a principal submatrix of $C$ and $B$ is the Schur complement of $D$ of $C$. The inverse $B^{-1}$ is then a principal submatrix of $C^{-1}$, and by Cauchy's interlacing theorem we have that   $\lambda_{\min}(B^{-1})\geq \lambda_{\min}(C^{-1})$ and then $\lambda_{\max}(B)\leq \lambda_{\max}(C)$. Since $T \subseteq S$ by assumption, we also have that $(S\cap M)\cup T \subseteq S$, then by interlacing again $\lambda_{\max}(C)\leq \lambda_{\max}\Big(\tfrac{1}{n}\bX_S^{\top}\bX_S\Big)=\bar{\lambda}$. The upper bound in~\eqref{eq:upperbounnnnnd} follows from the latter inequality and also observing that $\|\bbeta^*_{(S\cap M)\setminus T}\|_2^2\leq \sum_{j=1}^b |(S_j\cap M_j)\setminus T_j|\max_{i\in (S_j\cap M_j)\setminus T_j}{\beta_i^*}^2$.

We now derive the lower bound in~\eqref{eq:lowerbounnnnnd}. By~\eqref{eq:poihe}, we have that 
$$\mu_{Q_TT}\geq n\lambda_{\min}\Big(\frac{1}{n}\bX_{(S\cap M)\setminus T}^{\top}  \big(I_n-P_T\big) \bX_{(S\cap M)\setminus T}\Big) \|\bbeta^*_{(S\cap M)\setminus T}\|_2^2.$$

Recall that $B^{-1}$ is a principal submatrix of $C^{-1}$, hence by interlacing $\lambda_{\max}(B^{-1})\leq \lambda_{\max}(C^{-1})$ and $\lambda_{\min}(B)\geq \lambda_{\min}(C)$. Since $T \subseteq S$ by assumption, we have that $(S\cap M)\cup T \subseteq S$, and hence $\lambda_{\min}(C)\geq\lambda_{\min}(\frac{1}{n}\bX_S^{\top} \bX_S)$.
The bound in~\eqref{eq:lowerbounnnnnd} is obtained by using the latter inequality and noting that also $\|\bbeta^*_{(S\cap M)\setminus T}\|_2^2\geq \sum_{j=1}^b |(S_j\cap M_j)\setminus T_j|\min_{i\in (S_j\cap M_j)\setminus T_j}{\beta_i^*}^2.$
\end{proof}

Lemma~\ref{lem:boundexpnc} next gives a bound on the expected normalized score of any $M \in \M$. We outline now the strategy to prove it. Recall that by Lemma~\ref{lem:L0toL1convergence}(ii), for any $M,M'$, the normalized score $NC(M)$ is bounded by a simple function of $C(M')-C(M)$. In the next lemma we take $M'= T$ where $T\subseteq S$. By the definition in~\eqref{eq:CM}, $C(T)-C(M)=\tfrac12 L_{TM}+\Delta_{MT}$, where for any two models $M,T\subseteq V$, we denote
\begin{equation}\label{eq:DeltaM}
\Delta_{MT}\;:=\;\sum_{j=1}^b \kappa_j(|M_j|-|T_j|) \quad\text{and} \quad L_{TM}\;:=\;\|\bX_{T}\tilde{\bbeta}^{(T)}\|^2-\|\bX_{M}\tilde{\bbeta}^{(M)}\|^2.
\end{equation}
To lower bound $C(T)-C(M)$, we use that $L_{TM}$ can be expressed in terms of chi-squared variables.
The idea is to take the union model $Q_T=T\cup M$, and to note that $L_{TM}=L_{Q_TM}-L_{Q_TT}$. Lemma~\ref{lemma:s7} is used to bound $L_{Q_TM}$, $L_{Q_TT}$, and hence also $L_{TM}$.

If $M\supset T$, then $Q_T=M$ and $-L_{TM}=L_{Q_TT}$ and by Lemma~\ref{lemma:s7}, $-L_{TM}\sim \chi^2_{|Q_T\setminus T|}(\mu_{Q_TT})$, which has expectation $|Q_T\setminus T|+\mu_{Q_TT}$. It also holds that $\Delta_{MT}>0$. Using properties of chi-squared distribution, we show that if one sets large enough $\kappa_j$ (and thus $\Delta_{MT}$), then $C(T)-C(M)$ is also large and $NC(M)$ vanishes. If $M\subset T$, then $Q_T=T$, $L_{TM}=L_{Q_TM}$, and by Lemma~\ref{lemma:s7}
$L_{TM}\sim \chi^2_{|Q_T\setminus M|}(\mu_{Q_TM})$, which has expectation $|Q_T\setminus M|+\mu_{Q_TM}$.
If the noncentrality parameter $\mu_{Q_TM}$ is large enough, then $L_{TM}$ and $C(T)-C(M)$ are also large, and, by the properties of the noncentral chi-squared distribution, $NC(M)$ vanishes. By Lemma~\ref{lem:boundrhogen}, large $\mu_{Q_TM}$ can be achieved by setting a condition on signal strength and $\rho(\bX)$.

If $M\neq T$ is such that $M\not\supset T$ and $M\not\subset T$, simultaneously large enough $\kappa_j$ and $\mu_{Q_T M}$ guarantee that $C(T)-C(M)$ is also large, and that $NC(M)$ vanishes.

\begin{lem}\label{lem:boundexpnc} For any $T\subseteq S$ and $M\in \M\setminus\{T\}$, denote $Q_T= M\cup T$ and $A_T:=\gamma\Delta_{MT}+\tfrac{1-\gamma}{6}\mu_{Q_TM}$ (\emph{cf}~\eqref{eq:DeltaM} and~\eqref{eq:muQM}). Suppose that, for some $\gamma \in (1/2,1]$, it holds that $A_T>0$, $|M\setminus T|=o(A_T)$, and $\mu_{Q_T T}=o(A_T)$. For any $\psi\in(0,1)$ and every $n$ large enough,
$$\E\left(NC(M)\right)  \;\;\leq\;\; e^{-\psi A_T}.$$
\end{lem}

\begin{proof}
Since $M\in \mathcal M\setminus \{T\}$, by Lemma~\ref{lem:L0toL1convergence} (ii),  $NC(M)\leq (1+e^{C(T)-C(M)})^{-1} \in [0,1]$.  
The first step of the proof is to use that for any random variable $Z \geq 0$ we have $\E(Z)=\int_0^\infty P(Z>u) du$, so that
\begin{eqnarray*}
  \E\left(NC(M)\right)  &\leq& \int_0^1 P\left((1+e^{C(T)-C(M)})^{-1}\geq u\right){\rm d} u\\
  &=&\int_0^1 P\left(C(T)-C(M) \leq \ln{(\tfrac1u-1)}\right){\rm d}u\\
  &=&\int_0^1 P\Big(-\tfrac12 L_{TM} \geq \Delta_{MT}-\ln{(\tfrac1u-1)}\Big) {\rm d}u.
\end{eqnarray*}

The second step of the proof is to use the union bound to upper bound the probability in the integrand above.
Let $Q_T=M \cup T$ and recall that $L_{TM}=L_{Q_TM}-L_{Q_TT}$. For any $\gamma$ 
$$
-\tfrac12 L_{TM} -(\Delta_{MT}-\ln{(\tfrac1u-1)})\;=\;\left(\tfrac12 L_{Q_TT}-\ln{\Bigl(\frac{e^{\gamma\Delta_{MT}}}{\tfrac1u-1}\Bigl)}\right)-\Bigl(\tfrac12 L_{Q_TM}+(1-\gamma)\Delta_{MT}\Bigl).
$$
Observe that for any random variables $U$, $V$, and any $\epsilon,\gamma' \geq 0$, the event $
\{U-V \geq 0\}$ implies $\{U \geq \gamma' \epsilon\}\cup
\{V  < \gamma' \epsilon\}$. Let $U=\tfrac12 L_{Q_TT}-\ln{\Bigl(\frac{e^{\gamma\Delta_{MT}}}{1/u-1}\Bigl)}$ and $V=\tfrac12 L_{Q_TM}+(1-\gamma)\Delta_{MT}$. Take $\epsilon=\mu_{Q_TM}$ and $\gamma'=\tfrac16({1-\gamma})$, and observe that
$A_T:=\gamma\Delta_{MT}+\tfrac{1-\gamma}{6}\mu_{Q_TM}= \gamma\Delta_{MT}+\gamma'\mu_{Q_TM}$. We then have
\begin{align*}
\{U\geq \gamma'\epsilon\}\;&=\;\{\tfrac12 L_{Q_TT}\geq \ln{\Bigl(\frac{e^{\gamma\Delta_{MT}}}{\tfrac1u-1}}\Bigl)+\gamma'\mu_{Q_TM}\}\;=\;\{\tfrac12 L_{Q_TT}\geq \ln{\Bigl(\frac{e^{A_T}}{\tfrac1u-1}}\Bigl)\}\\
\{V<\gamma'\epsilon\}\;&=\;\{\tfrac12 L_{Q_TM}<-(1-\gamma)\Delta_{MT}+\gamma'\mu_{Q_TM}\}\;=\;\{\tfrac12 L_{Q_TM}<\gamma'(\mu_{Q_TM}-6\Delta_{MT})\}.
\end{align*}
By the union bound we have that
\begin{equation}\label{eq:prooflemmas20nonspu}
    \E\left(NC(M)\right)  \leq \int_0^1P\Bigl(\tfrac12 L_{Q_TT}\geq \ln{\Bigl(\frac{e^{A_T}}{\tfrac1u-1}\Bigl)}\Bigl){\rm d}u+ P\Bigl(\tfrac12 L_{Q_TM}<\gamma'(\mu_{Q_TM}-6\Delta_{MT})\Bigl).
\end{equation}

The third and final step of the proof is to upper bound each of the terms in the right-hand side of~\eqref{eq:prooflemmas20nonspu}. The intuition is that both $T$ and $M$ are nested within $Q_T$, and therefore $L_{Q_T T}$ and $L_{Q_T M}$ follow chi-squared distributions.
We first bound the first term. 
If $M \subset T$ then $Q_T=T$, $L_{Q_TT}=0$, and this term is zero. Suppose now that $M \not\subset T$.  
Then, by Lemma~\ref{lemma:s7}, 
$L_{Q_TT} \sim \mathcal{\chi}_{|Q_T\setminus T|}^2(\mu_{Q_TT})$ with $|Q_T\setminus T|=|M\setminus T|$. By assumption, $A_T>0$, $|M\setminus T|=o(\ln(e^{A_T}))$ and $\mu_{Q_TT}=o(\ln(e^{A_T}))$, then by Lemma~\ref{lemma:s20}, with $c=2$ and $\alpha\in(\psi,1)$, and every $n$ large enough,
\begin{equation}\label{eq:boundLQS}
\int_0^1P\Bigl(L_{Q_TT}>2\ln{\Bigl(\frac{e^{A_T}}{1 / u-1}\Bigl)}\Bigl){\rm d}u \;<\; 
e^{-\alpha A_T}.
\end{equation}

We now bound the second term in~\eqref{eq:prooflemmas20nonspu}. If $M \supset T$, then $Q_T=M$, $L_{Q_TM}=0$, $\mu_{Q_TM}=0$, and this term is zero. Alternatively, if $M \not\supset T$
then, by Lemma~\ref{lemma:s7}, $L_{Q_TM} \sim \chi_{|Q_T\setminus M|}^2(\mu_{Q_TM})$ with $|Q_T\setminus M|=|T\setminus M|$.
Clearly, when $\mu_{Q_TM}\leq 6\Delta_{MT}$ this term is also zero, so suppose that $\mu_{Q_TM}> 6\Delta_{MT}$.  
We have  
$$P\left(L_{Q_TM}<2\gamma'(\mu_{Q_TM}-6\Delta_{MT})\right)\; \leq \; P\left(L_{Q_TM}<2\gamma'\mu_{Q_TM}\right)$$ 
and, by Lemma~\ref{lemma:s2} and using that $\gamma'\in (0,\tfrac{1}{12})$, we obtain that 
$$P\left(L_{Q_TM}<2\gamma'\mu_{Q_TM}\right) \;\;\leq\;\; (2\gamma')^{\frac{|T\setminus M|}{4}} e^{-\frac{1}{2}(1-\sqrt{2\gamma'})^2\mu_{Q_TM}}\;\;\leq\;\;  (\tfrac16)^{\frac{|T\setminus M|}{4}}e^{-\frac{1}{2}(1-\sqrt{2\gamma'})^2\mu_{Q_TM}}.$$
Since $\mu_{Q_TM}>6\Delta_{MT}$, we get
$$
A_T\;=\;\gamma'\mu_{Q_TM}+\gamma\Delta_{MT}\;<\;\frac{1-\gamma}{6}\mu_{Q_TM}+\frac{\gamma}{6}\mu_{Q_TM}\;=\;\frac16\mu_{Q_TM}\;\leq\;\frac12(1-\sqrt{2\gamma'})^2\mu_{Q_TM},
$$
where the last inequality follows from the fact that $\gamma'\in(0, \tfrac{1}{12})$. We then get
\begin{equation}\label{eq:boundLQM}
    P\left(L_{Q_TM}<2\gamma'(\mu_{Q_TM}-6\Delta_{MT})\right) \;\leq\; \left(\frac16\right)^{\frac{|T\setminus M|}{4}}e^{-A_T} \leq e^{-A_T}\leq e^{-\alpha A_T}
\end{equation}
Summing the bounds in~\eqref{eq:boundLQS} and~\eqref{eq:boundLQM} gives that for every $n$ large enough $\E\left(NC(M)\right) < 2e^{-\alpha A_T}$. Since $\psi<\alpha$, for every $n$ large enough, we have that $\E\left(NC(M)\right) < e^{-\psi A_T}$ as claimed.
\end{proof}

\subsection{Proofs of Section~\ref{sec:properties_legression}}\label{suppsec:proofseclinreg}
\subsubsection{Proof of Lemma~\ref{lem:usefulreformulation}}
	First note that by discarding constant terms 
 $$\arg\max_{M\in \M} \left\{\max_{\beta\in \mathcal{L}_M}\ell(\by; \bbeta)-\sum_{j=1}^b\kappa_j |M_j| \right\}=\arg\min_{M\in \M} \left\{\min_{\beta\in \mathcal{L}_M}\bigg\{ \tfrac{1}{2}\|\by-\bX\bbeta\|^2\right\}+\sum_{j=1}^b\kappa_j |M_j| \bigg\}.$$
 
 We also have that 
 $$\min_{\beta\in \mathcal{L}_M} \tfrac{1}{2}\|\by-\bX\bbeta\|^2=\tfrac{1}{2}\|\by-\bX_{\!\!M}\tilde{\bbeta}^{(M)}\|^2=\tfrac{1}{2}\|\by\|^2-\tfrac{1}{2}\|\bX_{\!\!M}\tilde{\bbeta}^{(M)}\|^2,$$
where in the last equality we used that $\tilde{\bbeta}^{(M)}= (\bX_M^T \bX_M)^{-1} \bX_M^T \by$.
 The maximization in~\eqref{eq:Mhat} can be then replaced with the maximization of $C(M)=\tfrac{1}{2}
 \|\bX_{\!\!M}\tilde{\bbeta}^{(M)}\|^2-\sum_{j=1}^b \kappa_j|M_j|$ which is equivalent to maximizing $NC(M)$.

\subsubsection{Proof of Lemma~\ref{lem:L0toL1convergence}}
Part (i) follows directly from Lemma~\ref{lem:L0toL1convergencegen} by taking $\{M^1,\ldots,M^k\}=\{S\}$.
Part (ii) follows from
$$NC(M)=\Big(1+\sum_{N \neq M} e^{C(N)-C(M)}\Big)^{-1}
< \big(1+ e^{C(M')-C(M)}\big)^{-1}.$$ 

\subsubsection{Proof of Theorem~\ref{theo:suffcondlinearmodel}} 
The proof strategy is to first use Lemma~\ref{lem:boundexpnc} with $T=S$ to show that for every $M\neq S$, $\E(NC(M)) \leq e^{-\psi A_S}$ for every large enough $n$ and any $\psi \in (0,1)$, where $A_S=\gamma\Delta_{MS}+\tfrac{1-\gamma}{6}\mu_{Q_SM}$ (\emph{cf}~\eqref{eq:DeltaM} and~\eqref{eq:muQM}), $\gamma\in(1/2,1)$ is defined in Assumption~\hyperlink{lab:A2}{A2} and $Q_S= M\cup S$. Assumptions~\hyperlink{lab:A1}{A1}--\hyperlink{lab:A2}{A2} and the fact that $M\setminus S\subseteq S^C$ ensure the 
assumptions of Lemma~\ref{lem:boundexpnc} are met for any $\bbeta^*\in \bB$. 
The second step is to obtain a lower bound for $A_S$, which gives a new upper bound for $\E(NC(M))$.
The final step is to use these bounds to get an upper-bound on $\sum_{M \in \M\setminus \{S\}} \E\left(NC(M)\right)$ that holds for any $\bbeta^*\in \bB$ and asymptotically vanishes under Assumptions~\hyperlink{lab:A1}{A1} and \hyperlink{lab:A2}{A2}. We then use Lemma~\ref{lem:L0toL1convergence} to conclude on the vanishing of $\sup_{\bbeta^*\in \bB}P(\hat{S}^{ei} \neq S)$. 

First, to show that $\E(NC(M)) \leq e^{-\psi A_S}$
for any $M\in\mathcal{M}\setminus \{S\}$, we show that $A_S$ satisfies the conditions of Lemma~\ref{lem:boundexpnc}, taking $T=S$. 
That is, we wish to show that, $A_S>0$, $|M\setminus S|=o(A_S)$, and $\mu_{Q_S S}=o(A_S)$. 
Observe that $\Delta_{MS}$, defined in~\eqref{eq:DeltaM}, can be rewritten as $\Delta_{MS}=\sum_{j=1}^b(|M_j\setminus S_j|-|S_j\setminus M_j|)\kappa_j$. By  Lemma~\ref{lem:boundrhogen} with $T=S$, for any $\bbeta^*\in\bB$ and every $n\in \N$, we have 
\begin{equation}\label{eq:nonoverfittingcond}
\begin{aligned}
     A_S & \;=\; \gamma \Delta_{MS}+\frac{1-\gamma}{6}\mu_{Q_SM}\\
     & \;\geq\;  \gamma\sum_{j=1}^b|M_j \setminus S_j|\kappa_j + \sum_{j=1}^b |S_j\setminus M_j |\left(\tfrac{1-\gamma}{6} n \rho(\bX) {\beta^*_{\min,j}}^{2} - \gamma\kappa_j\right) 
\end{aligned}
\end{equation} 
Since $M\neq S$, $|M\setminus S|\neq0$ or $|S\setminus M|\neq0$, then by Assumptions~\hyperlink{lab:A1}{A1} and \hyperlink{lab:A2}{A2}, for every $n$ large enough, $A_S>0$. We immediately have $\mu_{Q_S S}=o(A_S)$ because $\beta^*_{Q_S\setminus S}=\beta^*_{M\setminus S}=0$ (any parameter outside the true support $S$ is by definition 0) and hence $\mu_{Q_S S}=0$.
If $|M\setminus S|=0$, $|M\setminus S|=o(A_S)$ also immediately. Consider now the case $|M\setminus S|\neq 0$. By Assumption~\hyperlink{lab:A2}{A2}, the last term in~\eqref{eq:nonoverfittingcond} is nonnegative, and hence
$$\frac{|M\setminus S|}{A_S}
= \frac{|M\setminus S|}{\gamma \Delta_{MS}+\frac{1-\gamma}{6}\mu_{Q_SM}}
\leq \bigg[ \gamma\sum_{j=1}^b\frac{|M_j \setminus S_j|}{|M\setminus S|}\kappa_j \bigg]^{-1} \leq \Big[ \gamma\min_{j=1,\ldots,b}\kappa_j \Big]^{-1} $$
where the last inequality follows from $\sum_{j=1}^b\frac{|M_j \setminus S_j|}{|M\setminus S|}=1$. By Assumption~\hyperlink{lab:A1}{A1} we have that $\min_j \kappa_j \to \infty$ as $n \rightarrow \infty$, and hence $|M\setminus S|=o(A_S)$. Thus,
by Lemma~\ref{lem:boundexpnc}, for any $\psi\in(0,1)$ and $\bbeta^*\in \bB$, and all $n$ large enough,
$
\E(NC(M))\;\leq\;e^{-\psi A_S}.
$

For the second step of the proof, let $A_S^*$ be the lower bound for $A_S$ given in~\eqref{eq:nonoverfittingcond}. That is 
$$A_S^*\;:=\;\gamma\sum_{j=1}^b|M_j \setminus S_j|\kappa_j + \sum_{j=1}^b |S_j\setminus M_j |\left(\tfrac{1-\gamma}{6} n \rho(\bX) {\beta^*_{\min,j}}^{2} - \gamma\kappa_j\right).$$
By~\eqref{eq:nonoverfittingcond}, we have, for any $\bbeta^*\in \bB$ and all $n$ large enough,
\begin{equation}\label{eq:boundexp}
    \E(NC(M))\;\leq\;e^{-\psi A_S^*}.
\end{equation}
Assumption~\hyperlink{lab:A2}{A2} implies there exist $g'_j \to \infty$ such that
\begin{equation}\label{eq:ogbea}
    \frac{(1-\gamma) n\rho(\bX)}{6}{\beta_{\min,j}^*}^2 \,- \,\kappa_j \;=\; \ln(s_j) + g'_{j}.
\end{equation}
Let $\delta \in (0,1)$ and denote $\bar{m}_j=\max\big\{\frac{2\ln(p_j-s_j)}{f_j},\frac{2\ln(s_j)}{g'_{j}}\big\}$, where $f_j$ is given in Assumption~\hyperlink{lab:A1}{A1}. 
Take $\psi=\max_{j=1,\ldots,b}\frac{\xi + \delta + \bar{m}_j}{1+\bar{m}_j}$ for some $\xi \in(0, 1- \delta)$ then $\psi \in (0,1)$ and we have, for every $j=1,\ldots,b$,
\begin{align}
    &\psi \;>\; \frac{\delta + \frac{2\ln(p_j-s_j)}{f_j}}{1+\frac{2\ln(p_j-s_j)}{f_j}} \;=\; \frac{\delta f_j/2 + \ln(p_j-s_j)}{f_j/2 +\ln(p_j-s_j)} \label{eq:psif}\\
    &\psi \;>\; \frac{\delta + \frac{2\ln(s_j)}{g'_j}}{1+\frac{2\ln(s_j)}{g'_j}} \;=\;  \frac{\delta g'_j/2 + \ln(s_j)}{g'_j/2 +\ln(s_j)} \;\geq\; \frac{\delta g'_j/2 + \ln(s_j)}{g'_j +\ln(s_j)}.\label{eq:psig}
\end{align}
In Assumption~\hyperlink{lab:A2}{A2}, $\gamma$ is defined as $\gamma:=\tfrac12(1+\max_j\ln(p_j-s_j)/\kappa_j)$, we then have 
$$
\gamma\kappa_j\;\geq\;\frac12 \Big(1+\frac{\ln(p_j-s_j)}{\kappa_j}\Big)\kappa_j\;=\;\ln(p_j-s_j)+\frac12(\kappa_j-\ln(p_j-s_j)) \;=\; \ln(p_j-s_j)+ \frac12f_j.
$$
Hence, by~\eqref{eq:psif}, we have
\begin{equation}\label{eq:ooofjfj}
\psi\gamma\kappa_j \;\geq\; \psi\big(\ln(p_j-s_j)+ \frac12f_j\big) \;\geq\;  \ln(p_j-s_j)+ \delta\frac12f_j.
\end{equation}
Further,
\begin{equation}\label{eq:ooofjfj2}
\psi\left(\tfrac{1-\gamma}{6} n \rho(\bX) {\beta^*_{\min,j}}^{2} - \gamma\kappa_j\right) \;\geq\; \psi\left(\ln(s_j)+g_j'\right) \;\geq\; \ln(s_j)+\delta \frac{1}{2}g'_j
\end{equation}
where the first inequality follows from~\eqref{eq:ogbea} and the second inequality from~\eqref{eq:psig}.

In~\eqref{eq:boundexp}, $\psi A^*_S=\sum_{j=1}^b|M_j \setminus S_j|\psi\gamma\kappa_j + \sum_{j=1}^b |S_j\setminus M_j |\psi\left(\tfrac{1-\gamma}{6} n \rho(\bX) {\beta^*_{\min,j}}^{2} - \gamma\kappa_j\right)$. Then by~\eqref{eq:ooofjfj} and~\eqref{eq:ooofjfj2}, we get, for any $\bbeta^*\in \bB$,
\begin{equation}\label{eq:fyv}
\E(NC(M))\;\leq\;\exp\left\{- \sum_{j=1}^b |M_j\setminus S_j|(\ln(p_j-s_j)+\delta \tfrac{f_j}{2})-\sum_{j=1}^b |S_j\setminus M_j|(\ln(s_j)+\delta \tfrac{g'_j}{2})\right\}.    
\end{equation}

For the final step of the proof, denote $\mathcal{S}= \sum_{M \in \M\setminus \{S\}} \E\left(NC(M)\right)$ for convenience. Since~\eqref{eq:fyv} holds for any $\bbeta^*\in \bB$, we have,
\begin{align*}
   \sup_{\bbeta^*\in \bB} \mathcal{S} &\leq \sum_{M \in \M\setminus \{S\}}e^{- \sum_{j=1}^b |M_j\setminus S_j|\big(\ln(p_j-s_j)+\delta \tfrac{f_j}{2}\big)-\sum_{j=1}^b |S_j\setminus M_j|\big(\ln(s_j)+\delta \tfrac{g'_j}{2}\big)}.
\end{align*}
Observe that if $|M_j\setminus S_j|=0$ and $|S_j\setminus M_j|=0$ for all $j$, then $M=S$ and the summand in the right-hand side above is 1. Then by adding and subtracting 1 we get
$$\sup_{\bbeta^*\in \bB} \mathcal{S} \leq -1+ \sum_{M \in \M}e^{- \sum_{j=1}^b |M_j\setminus S_j|\big(\ln(p_j-s_j)+\delta \tfrac{f_j}{2}\big)-\sum_{j=1}^b |S_j\setminus M_j|\big(\ln(s_j)+\delta \tfrac{g'_j}{2}\big)}.$$
We can split the sum in the right-hand side above into sums over the models that have the same number of inactive variables and missing the same number of truly active variables in every block. That is, the models $M$ such that for all $j$, $|M_j\setminus S_j|=u_j$ and $|S_j\setminus M_j|=w_j$ with $u_j \in \{0,\ldots,p_j-s_j\}$ and $w_j \in \{0,\ldots,s_j\}$. Denote those sums
$$S^{\bu}_{\bw}=\sum_{M \in \M:\forall \,j |M_j\setminus S_j|=u_j, |S_j\setminus M_j|=w_j } e^{- \sum_{j=1}^b u_j\big(\ln(p_j-s_j)+\delta \tfrac{f_j}{2}\big)-\sum_{j=1}^b w_j\big(\ln(s_j)+\delta \tfrac{g'_j}{2}\big)}.$$
We get
\begin{equation}\label{eq:uygfew}
    \sup_{\bbeta^*\in \bB} \mathcal{S} \leq -1+\sum_{w_1=0}^{s_1}\cdots\sum_{w_b=0}^{s_b}\sum_{u_1=0}^{p_1-s_1}\cdots\sum_{u_b=0}^{p_b-s_b}S^{\bu}_{\bw}.
\end{equation}
The number of models having, for all $j$, $u_j$ inactive parameters and missing $w_j$ out of the $s_j$ active parameters
is $\prod_{j=1}^b {\binom{p_j-s_j}{u_j}}{\binom{s_j}{w_j}}$. We thus have
\begin{align*}
    S^{\bu}_{\bw}&= \bigg(\prod_{j=1}^b {\binom{p_j-s_j}{u_j}}{\binom{s_j}{w_j}}\bigg)e^{- \sum_{j=1}^b u_j\big(\ln(p_j-s_j)+\delta \tfrac{f_j}{2}\big)-\sum_{j=1}^b w_j\big(\ln(s_j)+\delta \tfrac{g'_j}{2}\big)} \\
    &= \prod_{j=1}^b {\binom{p_j-s_j}{u_j}}e^{- u_j\big(\ln(p_j-s_j)+\delta \tfrac{f_j}{2}\big)}{\binom{s_j}{w_j}}e^{-w_j\big(\ln(s_j)+\delta \tfrac{g'_j}{2}\big)}.
\end{align*}
Inputting the expression above in~\eqref{eq:uygfew} gives
\begin{align*}
    \sup_{\bbeta^*\in \bB} \mathcal{S} 
    &\leq -1+\sum_{w_1=0}^{s_1}\cdots\sum_{w_b=0}^{s_b}\sum_{u_1=0}^{p_1-s_1}\cdots\sum_{u_b=0}^{p_b-s_b}\prod_{j=1}^b {\binom{p_j-s_j}{u_j}}e^{- u_j\big(\ln(p_j-s_j)+\delta \tfrac{f_j}{2}\big)}{\binom{s_j}{w_j}}e^{-w_j\big(\ln(s_j)+\delta \tfrac{g'_j}{2}\big)}\\
    &\leq - 1 + \prod_{j=1}^b\left(1+\sum_{u_j=1}^{p_j-s_j}{\binom{p_j-s_j}{u_j}}e^{- u_j(\ln(p_j-s_j)+\delta \tfrac{f_j}{2})}\right)\left(1+
    \sum_{w_j=1}^{s_j}{\binom{s_j}{w_j}}
    e^{-w_j(\ln(s_j)+\delta \tfrac{g'_j}{2})}\right) 
\end{align*}
where the second inequality follows from first factorizing over terms in $u_j$ and $w_j$ and then taking the term in 0 out of every sum. A standard bound on binomial coefficient for $1\leq k\leq n$ is
\begin{equation}\label{eq:upperboudbinom}
    \binom{n}{k} \leq \left(\frac{n\,e}{k}\right)^k \leq \left(n\,e\right)^k = e^{k(\ln(n)+1)}.
\end{equation}
Then
\begin{equation}\label{eq:upperbound_sumprob}
    \sup_{\bbeta^*\in \bB} \mathcal{S} \;\;\leq\;\; - 1+\prod_{j=1}^b\left(1+\sum_{u_j=1}^{p_j-s_j}e^{-u_j\left(\delta \tfrac{f_j}{2}-1\right)}\right)\left(1+
    \sum_{w_j=1}^{s_j}    e^{-w_j\left(\delta \tfrac{g'_j}{2}-1\right)}\right).
\end{equation}
Denote
$$
d_j\;=\;e^{1-\delta \tfrac{f_j}{2}},\qquad h_j\;=\;e^{1-\delta \tfrac{g'_j}{2}}
$$
where both expressions go to zero as $n$ increases since $f_j\to\infty$ and $g'_j\to\infty$.
For every $j$, by the properties of geometric sums, we have
\begin{align*}
    &1+\sum_{u_j=1}^{p_j-s_j}e^{-u_j\left(\delta \tfrac{f_j}{2}-1\right)}
    =\frac{1-d_j^{p_j-s_j+1}}{1-d_j} \\
    &1+\sum_{w_j=1}^{s_j}    e^{-w_j\left(\delta \tfrac{g'_j}{2}-1\right)} =\frac{1-h_j^{s_j+1}}{1-h_j}.
\end{align*}
Since both expressions converge to $1$ as $n$ grows, and $b$ is fixed we get that
\begin{equation}\label{eq:NPFENK}
\lim_{n \to \infty} \sup_{\bbeta^*\in \bB} \mathcal{S}=\lim_{n \to \infty} \sup_{\bbeta^*\in \bB}\sum_{M \in \M\setminus \{S\}}\E\left(NC(M)\right)= 0.
\end{equation}
By Lemma~\ref{lem:L0toL1convergence}, for any $\bbeta^*\in \bB$, $P(\hat{S}^{ei}\neq S)\leq \,2\, \mathcal{S}$, then $\lim_{n \to \infty} \inf_{\bbeta^*\in \bB}P(\hat{S}^{ei}= S)=1$.

 \subsubsection{Proof of Theorem~\ref{prop:fullrecovreg}}

First, we prove the upper bound on $P(\hat{S}^{ei} \neq S)$ in~\eqref{eq:rateconvreg}. We assume here \hyperlink{lab:A1}{A1} and \hyperlink{lab:A2}{A2}. Under the same assumptions, in the proof of Theorem~\ref{theo:suffcondlinearmodel}, the following bound was shown in~\eqref{eq:upperbound_sumprob} for any $\bbeta^*\in \bB$:
\begin{equation}\label{eq:prodconstreg}
 \sum_{M \in \M\setminus \{S\}} \E\left(NC(M)\right)
    \;\;\leq\;\; - 1+\prod_{j=1}^b\left(1+\sum_{u_j=1}^{p_j-s_j}e^{-u_j\left(\delta \tfrac{f_j}{2}-1\right)}\right)\left(1+
    \sum_{w_j=1}^{s_j}    e^{-w_j\left(\delta \tfrac{g'_j}{2}-1\right)}\right) .
\end{equation}
Denote $S(u_j)=\sum_{u_j=1}^{p_j-s_j}e^{-u_j\left(\delta \tfrac{f_j}{2}-1\right)}$,  $S(w_j)=\sum_{w_j=1}^{s_j}    e^{-w_j\left(\delta \tfrac{g'_j}{2}-1\right)}$, $d_j\;=\;e^{1-\delta \tfrac{f_j}{2}}$, and $h_j\;=\;e^{1-\delta \tfrac{g'_j}{2}}$. For every $j$, we have, by the properties of geometric sums:
\begin{align*}
    &S(u_j)
    =d_j\,\frac{1-d_j^{p_j-s_j}}{1-d_j} \\
    &S(w_j) =h_j\,\frac{1-h_j^{s_j}}{1-h_j}.
\end{align*}
Developing the product in the right-hand side in~\eqref{eq:prodconstreg} and reordering the resulting terms gives
$$\sup_{\bbeta^*\in \bB}\sum_{M \in \M\setminus \{S\}} \E\left(NC(M)\right)\;\;\leq\;\;-1+1+\sum_{j=1}^b\big[S(u_j)+S(w_j)\big]+\mathcal{R}$$
where all the terms in $\mathcal{R}$ are product of two or more of the sums $S(u_1),\ldots,S(u_b),S(w_1),\ldots,S(w_b)$. Given that $\delta>0$, $f_j \to \infty$ and $g_j' \to \infty$ by assumption, and hence $d_j \to 0$ and $h_j \to 0$, the $S(u_j)$ and $S(w_j)$ are smaller than 1 for all sufficiently large $n$ for all $j$. Then each of the $2^{2b}-2b-1$ terms in $\mathcal{R}$ is bounded above by $\sum_{j=1}^b\big[S(u_j)+S(w_j)\big]$ and we get, for every $n$ large enough,
 \begin{equation} \label{eq:fierg}
 \sup_{\bbeta^*\in \bB}\sum_{M \in \M\setminus \{S\}} \E\left(NC(M)\right)
 \leq (2^{2b}-2b)\sum_{j=1}^b \bigg[d_j\,\frac{1-d_j^{p_j-s_j}}{1-d_j}+ h_j\,\frac{1-h_j^{s_j}}{1-h_j}\bigg ]. \end{equation} 
Denote
\begin{equation} \label{eq:fierg2}r=\max_{j=1,\ldots,b}\bigg\{\frac{1-d_j^{p_j-s_j}}{1-d_j} \,,\,\frac{1-h_j^{s_j}}{1-h_j} \bigg\}.\end{equation} 
 By Lemma~\ref{lem:L0toL1convergence},~\eqref{eq:fierg} and~\eqref{eq:fierg2}, we then obtain:
$$\sup_{\bbeta^*\in \bB}P(\hat{S}^{ei} \neq S)\;\leq\;  2 \sup_{\bbeta^*\in \bB}\sum_{M \in \M\setminus \{S\}} \E\left(NC(M)\right)
     \;\leq\; 2 (2^{2b}-2b) r\,e\,\,\sum_{j=1}^b e^{-\delta \tfrac{f_j}{2}} +e^{-\delta \tfrac{g'_j}{2}}$$
By the definition of $f_j$ in Assumption~\hyperlink{lab:A1}{A1}, that of $g'_j$ in~\eqref{eq:ogbea}, and the fact that $2e<6$,  we get, for every $n$ large enough, that
\begin{equation}\label{eq:gpkr}
    \sup_{\bbeta^*\in \bB} P(\hat{S}^{ei} \neq S)\;\leq\; (2^{2b}-2b)6r\,\,\sum_{j=1}^b\,\, e^{-\tfrac{\delta}{2}\big[\kappa_j-\ln(p_j-s_j)\big]} + e^{-\tfrac{\delta}{2}\big[\frac{(1-\gamma) n\rho(\bX)}{6}{\beta_{\min,j}^*}^2 - \kappa_j -\ln(s_j) \big]}.
\end{equation}

We have $\frac{(1-\gamma) n\rho(\bX)}{6}{\beta_{\min,j}^*}^2 - \kappa_j > \Big(\sqrt{\frac{(1-\gamma) n\rho(\bX)}{6}}{\beta_{\min,j}^*} - \sqrt{\kappa_j}\Big)^2$ and then,
$$\sup_{\bbeta^*\in \bB} P(\hat{S}^{ei} \neq S)\;\leq\; (2^{2b}-2b)6r\,\,\sum_{j=1}^b\,\, e^{-\tfrac{\delta}{2}\big[\kappa_j-\ln(p_j-s_j)\big]} + e^{-\tfrac{\delta}{2}\big[(\sqrt{\frac{(1-\gamma) n\rho(\bX)}{6}}{\beta_{\min,j}^*} - \sqrt{\kappa_j})^2 -\ln(s_j) \big]}.
$$
We have $d_j \to 0$ and $h_j \to 0$, then for every $n$ large enough $\frac{1-d_j^{p_j-s_j}}{1-d_j}\to1$ and $\frac{1-h_j^{s_j}}{1-h_j}\to1$ for all $j$, and $r\to 1$. It follows that for any $c>(2^{2b}-2b)6$ for $n$ large enough the bound in~\eqref{eq:rateconvreg} holds. 

Second, we prove that, if for all $j=1,\ldots,b$, it holds that
\begin{equation}\label{betamincond:minimalreg}
    \lim_{n\to \infty}\frac{g(\gamma^*)\sqrt{ n\rho(\bX)/6}\;\beta_{\min,j}^*}{\sqrt{\ln(p_j-s_j)} + \sqrt{\ln(s_j)}} \geq 1
\end{equation}
with $g(\gamma^*)=\sqrt{(-1 + 2 \gamma^*)(1-\gamma^*)}/(1 + \sqrt{2 - 2 \gamma^*})$ then $\kappa^*_j$ defined in~\eqref{eq:oraclepenreg} satisfies Assumptions~\hyperlink{lab:A1}{A1} and \hyperlink{lab:A2}{A2}. Indeed, under~\eqref{betamincond:minimalreg}, for every $j$, there exists a sequence $x_j$ such that $\lim_{n\to\infty} x_j\geq1$ and $g(\gamma^*)\sqrt{\frac{n\rho(\bX)}{6}}\beta^*_{\min,j}=x_j\big(\sqrt{\ln(p_j-s_j)} + \sqrt{\ln(s_j)}\big)$. Let $v_j= x_j \sqrt{(1-\gamma^*)}g(\gamma^*)^{-1}$, then $\sqrt{\frac{(1-\gamma^*)n\rho(\bX)}{6}}\beta^*_{\min,j}=v_j\big(\sqrt{\ln(p_j-s_j)} + \sqrt{\ln(s_j)}\big)$. Denote $a=\sqrt{\ln(p_j-s_j)}$ and $a'=\sqrt{\ln(s_j)}$.
We have
$$
\sqrt{\kappa^*_j} \;=\;\frac{v_j}{2} (a+a')+\frac{(a^2-a'^2)}{2v_j(a+a')}\;=\;\frac{v_j}{2} (a+a')+\frac{a-a'}{2v_j}\;=\;\frac{v_j^2+1}{2v_j}a+\frac{v_j^2-1}{2v_j}a'.
$$
Since $\sqrt{(1-\gamma^*)}g(\gamma^*)^{-1}>1$ for $\gamma^*\in(1/2,1)$ and $\lim_{n\to\infty} x_j\geq1$, for $n$ large enough, $v_j>1$. We also have $a'\geq 0$. It follows that $a'(v_j^2-1)/2v_j \geq 0$ for $n$ large enough and
\begin{equation}\label{eq:ojdv}
\sqrt{\kappa^*_j} =\frac{v_j^2+1}{2v_j}a+\frac{v_j^2-1}{2v_j}a'  \geq  \Big(\frac{v_j^2+1}{2v_j}+1-1\Big)a = \Big(1+\frac{(v_j-1)^2}{2v_j}\Big)\sqrt{\ln(p_j-s_j)} 
\end{equation}
which implies Assumption~\hyperlink{lab:A1}{A1} since $v_j>1$ for $n$ sufficiently large. We also have
\begin{equation}\label{eq:oabpx}
\frac{\sqrt{(1-\gamma^*) n\rho(\bX)}}{6}{\beta_{\min,j}^*}-\sqrt{\kappa^*_j} = \frac{v_j^2-1}{2v_j}a+\frac{v_j^2+1}{2v_j}a' \geq \Big(1+\frac{(v_j-1)^2}{2v_j}\Big)\sqrt{\ln(s_j)}.
\end{equation}
Let $\gamma = \tfrac12(1+\max_j\ln(p_j-s_j)/\kappa^*_j)$ as defined in Assumption \hyperlink{lab:A2}{A2}. By~\eqref{eq:ojdv}, $\gamma = \tfrac12(1+\max_j(1+(v_j-1)^2/(2v_j))^{-2})$. Since $\lim_{n\to\infty} x_j\geq 1$, for sufficiently large $n$, $v_j\geq \sqrt{(1-\gamma^*)}g(\gamma^*)^{-1}$. Simple algebra shows the latter inequality implies $\gamma=\tfrac12(1+\max_j\ln(p_j-s_j)/\kappa^*_j) < \gamma^*$ and $1-\gamma > 1-\gamma^*$. Then, by~\eqref{eq:oabpx}, 
\begin{equation*}
\frac{\sqrt{(1-\gamma) n\rho(\bX)}}{6}{\beta_{\min,j}^*}-\sqrt{\kappa^*_j} \geq \Big(1+\frac{(v_j-1)^2}{2v_j}\Big)\sqrt{\ln(s_j)}
\end{equation*}
for every sufficiently large $n$. Since $\lim_{n\to\infty} x_j\geq1$ and $\sqrt{(1-\gamma^*)}g(\gamma^*)^{-1}>1$ for $\gamma^*\in(1/2,1)$, $v_j>1$ for $n$ sufficiently large and  Assumption~\hyperlink{lab:A2}{A2} holds.

Finally, we prove~\eqref{eq:oraclerateconvreg}. Since Assumptions~\hyperlink{lab:A1}{A1} and \hyperlink{lab:A2}{A2} hold for the $\kappa^*_j$, by the first part of the theorem, there exists $c>0$ such that, for every $n$ large enough,
\begin{equation*}
 \sup_{\bbeta^*\in \bB} P(\hat{S}^{ei} \neq S)\;\leq\; c\,\sum_{j=1}^b\,\, e^{-\tfrac{\delta}{2}\big[\kappa_j^*-\ln(p_j-s_j)\big]} + 
    e^{-\tfrac{\delta}{2}\big[(\sqrt{\frac{(1-\gamma^*) n\rho(\bX)}{6}}{\beta_{\min,j}^*} - \sqrt{\kappa_j^*})^2 -\ln(s_j) \big]}.
\end{equation*}
Note that the $\kappa_j^*$ satisfy $\kappa^*_j-\ln(p_j-s_j)=\big(\sqrt{\frac{(1-\gamma^*) n\rho(\bX)}{6}}{\beta_{\min,j}^*} - \sqrt{\kappa_j^*}\big)^2 -\ln(s_j)$ for all $j$. We then have
\begin{equation}\label{eq:afi}
 \sup_{\bbeta^*\in \bB}P(\hat{S}^{ei} \neq S)\;\leq\; 2c\,\sum_{j=1}^b\,\, e^{-\tfrac{\delta}{2}\big[\kappa_j^*-\ln(p_j-s_j)\big]}.
\end{equation}
Denote \\$d =\frac12\sqrt{\frac{6}{(1-\gamma^*) n\rho(\bX)}}\frac{1}{\beta_{\min,j}^*}\big(\ln({p_j-s_j})-\ln({s_j})\big)$ such that $\sqrt{\kappa^*_j}=  \frac12\sqrt{\frac{(1-\gamma^*) n\rho(\bX)}{6}}{\beta_{\min,j}^*} + d$. Then
$$e^{-\frac{\delta}{2}\left(\kappa^*_j-\ln(p_j-s_j)\right)}
=e^{-\frac{\delta}{2}\left[\frac{(1-\gamma^*) n\rho(\bX)}{24}{\beta_{\min,j}^*}^2+d^2 - \frac{1}{2}\left(\ln (p_j- s_j)+\ln (s_j)\right)\right]}.$$
By considering separately the two possible maxima in $\ln \max\{p_j-s_j,s_j\}$, we get that
\begin{equation}\label{eq:oig}
    \frac{e^{-\frac{\delta}{2}\left(\kappa^*_j-\ln(p_j-s_j)\right)}}{e^{-\frac{\delta}{2}\left[\frac{(1-\gamma^*) n\rho(\bX)}{24}{\beta_{\min,j}^*}^2-\ln \max\{p_j-s_j,s_j\}\right]}}= 
e^{-\frac{\delta}{2}\left[d^2 + \frac{1}{2}\left|\ln (p_j- s_j)-\ln (s_j)\right|\right]} <1.
\end{equation}
It follows that, by~\eqref{eq:afi} and~\eqref{eq:oig},
$$\sup_{\bbeta^*\in \bB} P(\hat{S}^{ei} \neq S) \; \leq \; 2c\,\sum_{j=1}^b\,\, e^{-\frac{\delta}{2}\left[\frac{(1-\gamma^*) n\rho(\bX)}{24}{\beta_{\min,j}^*}^2-\ln \max\{p_j-s_j,s_j\}\right]}$$
which proves~\eqref{eq:oraclerateconvreg}.

\subsubsection{Proof of Proposition~\ref{prop:nectypeItypeII}}
Consistent recovery of $S$ implies that for any subset of models $\mathcal{N}$, $\max_{M\in \mathcal{N}}NC(M)/NC(S)<1$ with probability going to 1 as $n$ grows, where $NC$ is the normalized criterion defined in~\eqref{eq:CM}. 
In Part (i), for every $j=1,\ldots,b$, we obtain a necessary assumption on the $\kappa_j$ for selection consistency by considering $\mathcal{N}=\OO_j$ as defined in~\eqref{eq:defOj}. In Part (ii), we obtain a necessary assumption on the $\beta_{\min,j}^*$ by considering a model that under-fits by only one variable $i$ with $\beta^*_i=\beta_{\min,j}^*$. In Part (iii), we use (i) and (ii) to get, for every $j=1,\ldots,b$, a necessary assumption on the scaling of $(n,p_j,s_j,\beta_{\min,j}^*)$.

\textbf{Part (i).}\,\,
The event $\hat{S}^{ei}=S$ (correct recovery of $S$) requires the event 
$\max_{M\in \OO_j}\frac{NC(M)}{NC(S)}<1$. We have, for any $\bbeta^*\in \bB$,
$$P(\hat{S}^{ei}=S) 
\leq P\Big(\max_{M\in \OO_j}\frac{NC(M)}{NC(S)}<1\Big)
= P \left( \max_{M\in \OO_j} e^{C(M)-C(S)} <1 \right).$$
Using that $C(S)-C(M)=L_{SM}/2 + \Delta_{MS}$ for $L_{SM}$ and $\Delta_{MS}$ defined in~\eqref{eq:DeltaM}, we obtain
\begin{align}\label{eq:nOAJH1}
\sup_{\bbeta^*\in \bB}P(\hat{S}^{ei}=S) 
\leq P \left( \max_{M\in \OO_j} e^{-(\frac{1}{2} L_{SM} + \Delta_{MS})} < 1 \right)
=
P\Big(\max_{M\in \OO_j}L_{MS}<2\kappa_j\Big).
\end{align}
Taking square root on both sides of the inequality in the right-most probability in~\eqref{eq:nOAJH1}, we have
\begin{equation}\label{eq:oSBV}
    \sup_{\bbeta^*\in \bB}P(\hat{S}^{ei}=S) 
\leq P\Big(\max_{M\in \OO_j}\sqrt{L_{MS}}<\sqrt{2\kappa_j}\Big).
\end{equation}
By Lemma~\ref{lemma:s7}, for every $M\in \OO_j$, $L_{MS}\in \chi^2_1$ and there exists $Z_M\sim N(0,1)$ such that $\sqrt{L_{MS}}=|Z_M|$. Since for every $M\in \OO_j$, $|Z_M|\geq Z_M$, we have
\begin{equation}\label{eq:upperboundrecov}
    \sup_{\bbeta^*\in \bB}P(\hat{S}^{ei}=S) \leq  1-P\big(\max_{M\in \OO_j}Z_M\geq\sqrt{2\kappa_j}\Big).
\end{equation}
The set $\OO_j$ has cardinality $p_j-s_j$ which, by assumption, diverges, then by Theorem 3.4 in~\cite{hartigan}, for any $\varepsilon>0$,
$$P\big(\max_{M\in \OO_j}Z_M \geq  \sqrt{2\underline{\lambda}_j \ln(p_j-s_j)}(1-\varepsilon)\big) \to 1,$$
where $\underline{\lambda}_j$ is as defined in~\eqref{eq:ovpsk}.
If $\lim_{n\to\infty}\frac{\kappa_j}{\underline{\lambda}_j\ln(p_j-s_j)} <1$, then $\lim_{n \to \infty} P\big(\max_{M\in \OO_j}Z_M\geq\sqrt{2\kappa_j}\big) = 1$. Hence, by~\eqref{eq:upperboundrecov} we have that $\lim_{n \to \infty} \sup_{\bbeta^*\in \bB}P(\hat{S}^{ei}=S) = 0$, as claimed.

\textbf{Part (ii).}\,\,
Consider a $\bbeta^*\in \bB$ such that for some $i_{0}\in B_j$, $\beta^*_{i_{0}}=\beta_{\min,j}^*$ and the model $M$, that is under-fitting $S$ by variable $i_{0}$. The event $\hat{S}^{ei}=S$ (correct recovery of $S$) requires the event $\frac{NC(M)}{NC(S)}<1$ where $NC$ is the normalized criterion in~\eqref{eq:CM}. That is $S$ is preferred to model $M$. 
\begin{align*}
P(\hat{S}^{ei}=S) 
\leq P\Big(\frac{NC(M)}{NC(S)}<1\Big)
= P \left( e^{C(M)-C(S)} <1 \right)
\end{align*}
Using that $C(S)-C(M)=L_{SM}/2 + \Delta_{MS}$ for $L_{SM}$ and $\Delta_{MS}$ defined in~\eqref{eq:DeltaM}, and that $\Delta_{MS}=-\kappa_j$ we obtain
\begin{align}\label{eq:pokvfb}
P(\hat{S}^{ei}=S) 
\leq  P \left(e^{-(\frac{1}{2} L_{SM} + \Delta_{MS})} < 1 \right) =
P\Big(L_{SM}>2\kappa_j\Big).
\end{align}
Taking square root on both sides of the inequality in the right-most probability in~\eqref{eq:pokvfb}, we get
\begin{align}\label{eq:pokvfc}
P(\hat{S}^{ei}=S) 
\leq P\Big(\sqrt{L_{SM}}>\sqrt{2\kappa_j}\Big).
\end{align}
By Lemma~\ref{lemma:s7}, $L_{SM}\in \chi^2_1(\mu_{SM})$ and  there exists then $Z_M\sim N(\sqrt{\mu_{SM}},1)$ such that $\sqrt{L_{SM}}=|Z_M|$. We have
\begin{equation}\label{eq:upperboundrecov2}
    P(\hat{S}^{ei}=S) \leq  P\big(|Z_M|>\sqrt{2\kappa_j}\big).
\end{equation}
Consider first the case where $2\kappa_j>\mu_{SM}$. Using that $|Z_M|=|Z_M-\sqrt{\mu_{SM}}+\sqrt{\mu_{SM}}|\leq  |Z_M-\sqrt{\mu_{SM}}| + \sqrt{\mu_{SM}}$, we get that 
$$P\big(|Z_M|>\sqrt{2\kappa_j}\big)\leq P\big(|Z_M-\sqrt{\mu_{SM}}|>\sqrt{2\kappa_j}-\sqrt{\mu_{SM}}\big) = 2P\big(z>\sqrt{2\kappa_j}-\sqrt{\mu_{SM}}\big)$$
where $z\sim N(0,1)$ and the equality follows from the symmetry of the standard Gaussian distribution. Since $2\kappa_j>\mu_{SM}$, $P\big(z>\sqrt{2\kappa_j}-\sqrt{\mu_{SM}}\big)<1/2$, $P\big(|Z_M|>\sqrt{2\kappa_j}\big)<1$ and $P(\hat{S}^{ei}=S)<1$ too by~\eqref{eq:upperboundrecov2}.

Consider next the case where $2\kappa_j\leq\mu_{SM}$. By Lemma~\ref{lemma:3tailbounds}, we have
$$P\big(|Z_M|>\sqrt{2\kappa_j}\big)
    = P\left(z>\sqrt{2\kappa_j}-\sqrt{\mu_{SM}}\right) + P\left(z<-\sqrt{2\kappa_j}-\sqrt{\mu_{SM}}\right)$$ where $z\sim N(0,1)$. Since $\kappa_j\to+\infty$, $\sqrt{2\kappa_j}+\sqrt{\mu_{SM}}\to\infty$ and we have that $P(z<-\sqrt{2\kappa_j}-\sqrt{\mu_{SM}}) \to 0$.  
    Further, by Lemma~\ref{lem:upperboundnoncentralgen} with $T=S$, $\sqrt{\mu_{SM}}\leq \sqrt{n\bar{\lambda}}\beta_{\min,j}^*$. By assumption, we have $\lim_{n\to \infty}\sqrt{n\bar{\lambda}}\beta_{\min,j}^*-\sqrt{2\kappa_j} < \infty$. It follows that $\lim_{n\to \infty}\sqrt{\mu_{SM}}-\sqrt{2\kappa_j} < \infty$ and we have that $\underset{n \to \infty}{\lim}P\left(z>\sqrt{2\kappa_j}-\sqrt{\mu_{SM}}\right)<1$. Hence, we get $\underset{n \to \infty}{\lim}P\left(\hat{S}^{ei} = S \right)<1$, as claimed.

\textbf{Part (iii).}\,\,
 First, we re-write 
 $$\sqrt{n\bar{\lambda}}\beta_{\min,j}^*-\sqrt{2\underline{\lambda}_j\ln(p_j-s_j)} \,=\, \sqrt{n\bar{\lambda}}\beta_{\min,j}^*-\sqrt{2\kappa_j}+\sqrt{2\underline{\lambda}_j\ln(p_j-s_j)}\Big(\sqrt{\tfrac{\kappa_j}{\underline{\lambda}_j\ln(p_j-s_j)}}-1\Big).$$ 
 By assumption $\lim_{n\to\infty}\sqrt{n\bar{\lambda}}\beta_{\min,j}^*-\sqrt{2\underline{\lambda}_j\ln(p_j-s_j)} < \infty$ and there exists $k\in\R^+$ such that
  \begin{equation}\label{eq:proofnecbetaminbislinreg}
      \lim_{n\to \infty}\sqrt{n\bar{\lambda}}\beta_{\min,j}^*-\sqrt{2\kappa_j}+\sqrt{2\underline{\lambda}_j\ln(p_j-s_j)}\Big(\sqrt{\tfrac{\kappa_j}{\underline{\lambda}_j\ln(p_j-s_j)}}-1\Big)\leq k
  \end{equation}

Consider the case $\lim_{n\to \infty}\sqrt{\tfrac{\kappa_j}{\underline{\lambda}_j\ln(p_j-s_j)}}-1<0$. Since by assumption $\lim_{n\to+\infty}(p_j-s_j)^{\underline{\lambda}_j}=+\infty$, $p_j-s_j$ diverges and then by Part (i), we have $\lim_{n\to \infty}\sup_{\bbeta^*\in\bB}P(\hat{S}^{ei} = S)=0$. Now consider the case $\lim_{n\to \infty}\sqrt{\tfrac{\kappa_j}{\underline{\lambda}_j\ln(p_j-s_j)}}-1 \geq0$. Then, $\kappa_j\to\infty$ since by assumption $\lim_{n\to \infty}\underline{\lambda}_j\ln(p_j-s_j) >0$. Moreover, condition~\eqref{eq:proofnecbetaminbislinreg} implies that $\lim_{n\to \infty}\sqrt{n\bar{\lambda}}\beta_{\min,j}^*-\sqrt{2\kappa_j} \leq k$. By Part (ii), there exists $\bbeta^*\in\bB$ such that $\lim_{n\to \infty}P(\hat{S}^{ei} = S)< 1$.

\subsubsection{Proof of Corollary~\ref{cor:selconstimpossiblereg}}
Since $\hat{S}^{sd}$ is $\hat{S}^{ei}$ with $b=1$, the assumptions of Proposition~\ref{prop:nectypeItypeII} (iii) for $\hat{S}^{sd}$ are met and there exists $\bbeta^*\in\mathfrak{B}$ such that $\lim_{n \to \infty} P(\hat{S}^{sd} = S) < 1$. Since Assumptions~\hyperlink{lab:A1}{A1} and \hyperlink{lab:A2}{A2} hold, by Theorem~\ref{theo:suffcondlinearmodel} we have that $\inf_{\bbeta^*\in\mathfrak{B}}\lim_{n \to \infty} P(\hat{S}^{ei} = S)= 1$.

\subsubsection{Proof of Theorem~\ref{thm:convtoT}}
The proof strategy is the same as that of Theorem~\ref{theo:suffcondlinearmodel}, with suitable adjustments. 
The first step is to use Lemma~\ref{lem:boundexpnc} to bound $\E(NC(M))$ for every $M\not\in \T(\bkappa)$. The main difference is that
in the proof of Theorem~\ref{theo:suffcondlinearmodel} 
we took $T=S$ in Lemma~\ref{lem:boundexpnc}, whereas now we take  a model $T=T^M\in \T(\bkappa)$
that depends on $M$. 
Intuitively, $T^M$ contains large truly non-zero parameters that are missed by $M$, and hence $T^M$ should be chosen over $M$ asymptotically. More precisely,
we choose $T^M\in \T(\bkappa)$ such that 
$T^M\setminus M\subseteq S^L(\bkappa)$ and the elements in $M\setminus T^M$ are either inactive or in $S^S(\kappa)$. The latter condition 
and Assumption~\hyperlink{lab:A1}{A1} ensure that the assumptions of Lemma~\ref{lem:boundexpnc} are met for any $\bbeta^*\in \bB$. We then get a bound $\E(NC(M)) \leq e^{-\psi A_{T^M}}$ for every large enough $n$ and any $\psi \in (0,1)$, where $A_{T^M}=\gamma\Delta_{MT^M}+\tfrac{1-\gamma}{6}\mu_{Q_{T^M}M}$ (\emph{cf}~\eqref{eq:DeltaM} and~\eqref{eq:muQM}), $\gamma\in(1/2,1)$ is defined in~\eqref{eq:defSes} and $Q_{T^M}= M\cup T^M$. The second step is to obtain a lower bound for $A_{T^M}$, which gives an upper bound for $\E(NC(M))$ that holds for any $\bbeta^*\in \bB$ (distinct to that obtained in the proof of Theorem~\ref{theo:suffcondlinearmodel}).
The final step is to 
get an upper-bound $\sum_{M \in \M\setminus \T(\bkappa)} \E\left(NC(M)\right)$ that vanishes (as $n$ grows) under the assumptions of Theorem~\ref{thm:convtoT} for any $\bbeta^*\in \bB$. Then Lemma~\ref{lem:L0toL1convergencegen} immediately implies that $\sup_{\bbeta^*\in \bB}P(\hat{S}^{ei}\not \in \T(\bkappa))$ also vanishes.

For the first step of the proof, recall that the set $\T(\bkappa)$ contains all the models that are a superset of $S^L(\bkappa)$ and a subset of $S^L(\bkappa)\cup S^{I}(\bkappa)$. For any $M \not\in\T(\bkappa)$, take the unique $T^M\in\T(\bkappa)$ such that $T^M = \big[M\cap S^{I}(\bkappa)\big]\cup S^L(\bkappa) $, which implies $M\cap S^{I}(\bkappa)=T^M\cap S^{I}(\bkappa)$. 
That is, $T^M$ contains all the large signals plus the intermediate signals in $M$.
To show that $\E(NC(M)) \leq e^{-\psi A_{T^M}}$ we show that $A_{T^M}$ satisfies the conditions of Lemma~\ref{lem:boundexpnc}, taking $T=T^M$. That is, we wish to show that three conditions hold: $A_{T^M}>0$, $|M\setminus T^M|=o(A_{T^M})$, and $\mu_{Q_{T^M} T^M}=o(A_{T^M})$.

Observe that $\Delta_{MT^M}$, defined in~\eqref{eq:DeltaM}, can be rewritten as $\Delta_{MT^M}=\sum_{j=1}^b(|M_j\setminus T^M_j|-|T^M_j\setminus M_j|)\kappa_j$ and then:
$$ A_{T^M}=\sum_{j=1}^b|M_j\setminus T^M_j|\gamma\kappa_j+\tfrac{1-\gamma}{6}\mu_{Q_{T^M}M}-\sum_{j=1}^b|T^M_j\setminus M_j|\gamma\kappa_j $$
Since $\gamma <1$, $(1-\gamma)/6 >0$, by  Lemma~\ref{lem:boundrhogen}, for every $n\in \N$ we have, for any $\bbeta^*\in \bB$,
\begin{equation}\label{eq:nonoverfittingcond2}
   A_{T^M}\;\geq\; \sum_{j=1}^b|M_j \setminus T^M_j|\gamma\kappa_j + \sum_{j=1}^b |T^M_j\setminus M_j |\Big(\tfrac{1-\gamma}{6} n \rho(\bX) \min_{i\in T^M_j\setminus M_j}{\beta^*_i}^{2} - \gamma\kappa_j\Big).
\end{equation} 
We have that $T^M\subseteq \big(S^I(\kappa) \cup S^L(\bkappa)\big)$ since $T^M\in\T(\bkappa)$ and that $T^M\cap S^{I}(\bkappa) = M\cap S^{I}(\bkappa)$. It follows that $T^M\setminus M \subseteq S^L(\bkappa)$. By definition of $S^L(\bkappa)$, the rightmost component in~\eqref{eq:nonoverfittingcond2} is nonnegative and, if $|T^M\setminus M|\neq0$, it is positive, for every $n$ large enough. By Assumption~\hyperlink{lab:A1}{A1}, component $\gamma\sum_{j=1}^b|M_j \setminus T^M_j|\kappa_j$ is nonnegative, and, if $|M\setminus T^M|\neq0$, positive. Now, $M\neq T^M$ implies that necessarily $|M\setminus T^M|\neq0$ or $|T^M\setminus M|\neq0$ and then, for every $n$ large enough, $A_{T^M}>0$, establishing the first condition required by Lemma~\ref{lem:boundexpnc}. Regarding its second condition, if $|M\setminus T^M|=0$, we immediately have $|M\setminus T^M|=o(A_{T^M})$. If $|M\setminus T^M|\neq 0$, since the rightmost component in~\eqref{eq:nonoverfittingcond2} is nonnegative, we have
$$\frac{|M\setminus T^M|}{A_{T^M}}\leq \bigg[ \gamma\sum_{j=1}^b\frac{|M_j \setminus T^M_j|}{|M\setminus T^M|}\kappa_j \bigg]^{-1} \leq \bigg[ \gamma\min_{j=1,\ldots,b}\kappa_j \bigg]^{-1} $$
where the last inequality follows from $\sum_{j=1}^b\frac{|M_j \setminus T^M_j|}{|M\setminus T^M|}=1$. By Assumption~\hyperlink{lab:A1}{A1}, $\min_{j=1,\ldots,b}\kappa_j\to \infty$ as $n\to\infty$ and hence $|M\setminus T^M|=o(A_{T^M})$ when $|M\setminus T^M|\neq 0$ too.

Finally, consider the third condition in Lemma~\ref{lem:boundexpnc}.
In the proof of Theorem~\ref{theo:suffcondlinearmodel}, $\mu_{Q_SS}=o(A_S)$ is immediate because $ Q_S \setminus S= M\setminus S\subseteq S^C$, i.e. since all parameters in $M \setminus S$ are truly zero we have that $\mu_{Q_SS}=0$.
Here $M\setminus T^M$ is not necessarily a subset of $S^C$, hence $\mu_{Q_T^M T^M} \geq 0$. Note that, since $T^M\in\T(\bkappa)$, we have that $S^L(\bkappa)\subseteq T^M$. Moreover, we have $M\cap S^{I}(\bkappa)=T^M\cap S^{I}(\bkappa)$. It follows that $M\setminus T^M\subseteq (S^{I}(\bkappa) \cup S^L(\bkappa))^C$, that is the elements of $M\setminus T^M$ are either inactive or belong to $S^S(\kappa)$. If $M\cap S^{S}(\bkappa)=\emptyset$ then $M\setminus T^M\subseteq S^C$ and we immediately get $\mu_{Q_{T^M} T^M}=o(A_{T^M})$ because $\bbeta^*_{M \setminus T^M}=\mu_{Q_{T^M} T^M}=0$ for any $\bbeta^*\in \bB$. Assume now that $M\cap S^{S}(\bkappa)\neq\emptyset$. Using that the rightmost component in~\eqref{eq:nonoverfittingcond2} is nonnegative and Lemma~\ref{lem:upperboundnoncentralgen}, we have
$$\frac{\mu_{Q_{T^M} T^M}}{ A_{T^M}} \leq \frac{\mu_{Q_{T^M} T^M}}{ \gamma\sum_{j=1}^b|M_j \setminus T^M_j|\kappa_j} \leq \frac{n\,\bar{\lambda}\,\sum_{j=1}^b|(S_j\cap M_j )\setminus T^M_j|\,\max_{i\in (S_j\cap M_j )\setminus T^M_j}{\beta_{i}^*}^2}{ \gamma\sum_{j=1}^b|M_j \setminus T^M_j|\kappa_j}.$$
Observe that for all $j=1,\ldots,b$, $M_j \setminus T^M_j \supseteq (S_j\cap M_j )\setminus T^M_j$ and then $|M_j \setminus T^M_j| \geq |(S_j\cap M_j )\setminus T^M_j|$. We then get
$$\frac{\mu_{Q_{T^M} T^M}}{ A_{T^M}} \leq \frac{n\,\bar{\lambda}\,\sum_{j=1}^b|(S_j\cap M_j )\setminus T^M_j|\,\max_{i\in(S_j\cap M_j )\setminus T^M_j}{\beta_{i}^*}^2}{ \gamma\sum_{j=1}^b|(S_j\cap M_j ) \setminus T^M_j|\kappa_j}.$$
Moreover, since  $M\setminus T^M\subseteq (S^{I}(\bkappa) \cup S^L(\bkappa))^C$ as discussed earlier, we have, for all $j=1,\ldots,b$, $(S_j\cap M_j )\setminus T^M_j \subseteq S_j^{S}(\kappa_j)$. It follows $\max_{i\in(S_j\cap M_j )\setminus T^M_j}{\beta_{i}^*}^2 \leq \max_{i\in S_j^S}{\beta_{i}^*}^2$ and
\begin{equation}\label{eq:onbfsa}
  \frac{\mu_{Q_{T^M} T^M}}{ A_{T^M}} \leq \frac{n\,\bar{\lambda}\,\sum_{j=1}^b|(S_j\cap M_j )\setminus T^M_j|\,\max_{i\in S_j^{S}(\kappa_j)}{\beta_{i}^*}^2}{ \gamma\sum_{j=1}^b|(S_j\cap M_j ) \setminus T^M_j|\kappa_j}.  
\end{equation}
Let $\bar{r}:=\sum_{j=1}^b n\bar{\lambda}\max_{i\in S_j^{S}(\kappa_j)}{\beta_i^*}^2/(\gamma\kappa_j)$. We show next that $\bar{r}$ is an upper bound on $\mu_{Q_{T^M} T^M}/ A_{T^M}$. By restricting the sum in $\bar{r}$ to the $j$ such that $|(S_j\cap M_j )\setminus T^M_j|\neq 0$ and multiplying the numerator and denominator of the summand by $|(S_j\cap M_j )\setminus T^M_j|$, we get 
\begin{eqnarray}\label{eq:ovrwS}
    \bar{r}
    &\geq &\sum_{j=1,|(S_j\cap M_j )\setminus T^M_j|\neq 0}^b\frac{|(S_j\cap M_j )\setminus T^M_j|n\bar{\lambda}\max_{i\in S_j^{S}(\kappa_j)}{\beta_i^*}^2}{|(S_j\cap M_j )\setminus T^M_j|\gamma\kappa_j}
\end{eqnarray}
Note that for any collections $(\alpha_j,\delta_j)\in  \R\times \R\setminus\{0\}$, $j=1,\ldots,b$, we have
\begin{equation}\label{eq:truc}
    \sum_{j=1}^b\frac{\alpha_j}{\delta_j}=
\sum_{j=1}^b\frac{\alpha_j \frac{1}{\delta_j} (\delta_j + \sum_{l \neq j} \delta_l)}{\sum_{j=1}^b \delta_j}=
    \frac{\sum_{j=1}^b\alpha_j(1+\sum_{l\neq j} \frac{\delta_l}{\delta_j})}{\sum_{j=1}^b\delta_j}
\end{equation}
Using~\eqref{eq:truc} in the right-hand side of~\eqref{eq:ovrwS}, we get
\begin{align*}
     \bar{r}
    &\geq \frac{\sum_{j=1,|(S_j\cap M_j )\setminus T^M_j|\neq 0 }^b|(S_j\cap M_j )\setminus T^M_j|n\bar{\lambda}\max_{i\in S_j^{S}(\kappa_j)}{\beta_i^*}^2\Big(1+\sum_{l\neq j}\frac{|(S_l\cap M_l)\setminus T^M_l|\gamma\kappa_l}{|(S_j\cap M_j )\setminus T^M_j|\gamma\kappa_j}\Big)}{\sum_{j=1,|(S_j\cap M_j )\setminus T^M_j|\neq 0}^b|(S_j\cap M_j )\setminus T^M_j|\gamma\kappa_j} \\
    &\geq \frac{\sum_{j=1}^b|(S_j\cap M_j )\setminus T^M_j|n\bar{\lambda}\max_{i\in S_j^{S}(\kappa_j)}{\beta_i^*}^2}{\sum_{j=1}^b \big|(S\cap M_j)\setminus T^M_j\big|\gamma\kappa_j}
\end{align*}
where last inequality follows from $\Big(1+\sum_{l\neq j}\frac{|(S_l\cap M_l)\setminus T^M_l|\gamma\kappa_l}{|(S_j\cap M_j )\setminus T^M_j|\gamma\kappa_j}\Big)\geq 1$ for all $j$ and from the identity $\sum_{j=1,|(S_j\cap M_j )\setminus T^M_j|\neq 0}^b|(S_j\cap M_j )\setminus T^M_j|\gamma\kappa_j=\sum_{j=1}^b \big|(S\cap M_j)\setminus T^M_j\big|\gamma\kappa_j$. Then,
by~\eqref{eq:onbfsa},
$$\frac{\mu_{Q_{T^M} T^M}}{ A_{T^M}} \;\leq\; \bar{r} \;= \;\sum_{j=1}^b\frac{n\bar{\lambda}\max_{i\in S_j^{S}(\kappa_j)}{\beta_i^*}^2}{\gamma\kappa_j}.$$
By definition of $S_j^{S}(\kappa_j)$, $n\bar{\lambda}\max_{i\in S_j^{S}(\kappa_j)}{\beta_i^*}^2=o(\kappa_j)$ for all $j$ and $\mu_{Q_{T^M} T^M}=o(A_{T^M})$ for any $\bbeta^*\in \bB$ when $M\cap S^{S}(\bkappa)\neq\emptyset$ too.
We can now apply Lemma~\ref{lem:boundexpnc} and get that for every $M\in \mathcal M\setminus \T(\bkappa)$, for any $\bbeta^*\in \bB$  and $\psi\in(0,1)$, and for every $n$ large enough, $\E(NC(M))\;\leq\;e^{-\psi A_{T^M}}$.

The second step of the proof is to lower-bound $A_{T^M}=\gamma\Delta_{MT^M}+\tfrac{1-\gamma}{6}\mu_{Q_{T^M}M}$.
Let
$$A_{T^M}^*\;:=\;\gamma\sum_{j=1}^b|M_j \setminus T^M_j|\kappa_j + \sum_{j=1}^b |T^M_j\setminus M_j |\left(\tfrac{1-\gamma}{6} n \rho(\bX) \min_{i\in S_j^{L}(\kappa_j)}{\beta^*_i}^{2} - \gamma\kappa_j\right) $$
Recall that, for every $j$, $T^M_j\setminus M_j\subseteq S_j^{L}(\kappa_j)$. We have then $\min_{i\in T^M_j\setminus M_j}{\beta^*_i}^{2}\geq \min_{i\in S_j^{L}(\kappa_j)}{\beta^*_i}^{2}$, and by~\eqref{eq:nonoverfittingcond2}, $A_{T^M} \geq A_{T^M}^*$. It follows that for any $\psi\in(0,1)$ and for any $\bbeta^*\in\bB$ and every $n$ large enough,
\begin{equation}\label{eq:boundexp2}
    \E(NC(M))\;\leq\;e^{-\psi A_{T^M}^*}.
\end{equation}

To conclude the second part of the proof we lower-bound $\psi A^*_{T^M}=\sum_{j=1}^b|M_j \setminus T^M_j|\psi\gamma\kappa_j + \sum_{j=1}^b |T^M_j \setminus M_j|\psi\left(\tfrac{1-\gamma}{6} n \rho(\bX) \min_{i\in S_j^{L}(\kappa_j)}{\beta^*_i}^{2} - \gamma\kappa_j\right)$. To do this, we obtain a lower bound for $\psi \gamma \kappa_j$ and for $\psi\bigg(\tfrac{1-\gamma}{6} n \rho(\bX) \min_{i\in S_j^{L}(\kappa_j)}{\beta^*_i}^{2} - \gamma\kappa_j\bigg)$.

The definition of $S_j^{L}(\kappa_j)$ implies there exists some $g'_j \to \infty$ such that
\begin{equation}\label{eq:bngrws}
  \frac{(1-\gamma) n\rho(\bX)}{6}\min_{i\in S_j^{L}(\kappa_j)}{\beta^*_i}^{2}\,- \,\kappa_j \;=\; \ln(s_j) + g'_{j}.  
\end{equation}
Let $\delta \in (0,1)$ and denote $\bar{m}_j=\max\big\{\frac{2\ln(p_j-s_j)}{f_j},\frac{2\ln(s_j)}{g'_{j}}\big\}$, where $f_j$ is given in Assumption~\hyperlink{lab:A1}{A1}. 
Take $\psi=\max_{j=1,\ldots,b}\frac{\xi + \delta + \bar{m}_j}{1+\bar{m}_j}$ for some $\xi \in(0, 1- \delta)$ then $\psi \in (0,1)$, and we have, for every $j=1,\ldots,b$,
\begin{align}
    &\psi > \frac{\delta + \frac{2\ln(p_j-s_j)}{f_j}}{1+\frac{2\ln(p_j-s_j)}{f_j}} = \frac{\delta f_j/2 + \ln(p_j-s_j)}{f_j/2 +\ln(p_j-s_j)} \label{eq:psif2}\\
    &\psi > \frac{\delta + \frac{2\ln(s_j)}{g'_j}}{1+\frac{2\ln(s_j)}{g'_j}} = \frac{\delta g'_j/2 + \ln(s_j)}{g'_j/2 +\ln(s_j)} \geq \frac{\delta g'_j/2 + \ln(s_j)}{g'_j +\ln(s_j)}.\label{eq:psig2}
\end{align}
 Recall that Assumptions~\hyperlink{lab:A1}{A1} defines $f_j= \kappa_j - \ln(p_j-s_j)$ and $\gamma= \frac{1}{2}(1 + \max_j \frac{\ln(p_j-s_j)}{\kappa_j})$ respectively. Hence,
$$
\gamma\kappa_j\;\geq\;\frac12 \Big(1+\frac{\ln(p_j-s_j)}{\kappa_j}\Big)\kappa_j\;=\;\ln(p_j-s_j)+\frac12(\kappa_j-\ln(p_j-s_j)) \;=\; \ln(p_j-s_j)+ \frac12f_j.
$$
Hence, by~\eqref{eq:psif2}, we have
\begin{equation}\label{eq:ooofjfj3}
\psi\gamma\kappa_j \;\geq\; \psi\big(\ln(p_j-s_j)+ \frac12f_j\big) \;\geq\;  \ln(p_j-s_j)+ \delta\frac12f_j.
\end{equation}
Further,
\begin{equation}\label{eq:ooofjfj4}
\psi\bigg(\tfrac{1-\gamma}{6} n \rho(\bX) \min_{i\in S_j^{L}(\kappa_j)}{\beta^*_i}^{2} - \gamma\kappa_j\bigg) \;\geq\; \psi\left(\ln(s_j)+g_j'\right) \;\geq\; \ln(s_j)+\delta \frac{1}{2}g'_j.
\end{equation}
where the first inequality follows from~\eqref{eq:bngrws} and the second inequality from~\eqref{eq:psig2}. In~\eqref{eq:boundexp2}, $\psi A^*_{T^M}=\sum_{j=1}^b|M_j \setminus T^M_j|\psi\gamma\kappa_j + \sum_{j=1}^b |T^M_j \setminus M_j|\psi\left(\tfrac{1-\gamma}{6} n \rho(\bX) \min_{i\in S_j^{L}(\kappa_j)}{\beta^*_i}^{2} - \gamma\kappa_j\right)$. Then by~\eqref{eq:ooofjfj3} and~\eqref{eq:ooofjfj4}, we get, for any $\bbeta^*\in\bB$, 
\begin{equation}\label{eq:fyv2}
\E(NC(M))\;\leq\;\exp\left\{- \sum_{j=1}^b |M_j\setminus T^M_j|(\ln(p_j-s_j)+\delta \tfrac{f_j}{2})-\sum_{j=1}^b |T^M_j\setminus M_j|(\ln(s_j)+\delta \tfrac{g'_j}{2})\right\}.
\end{equation}

For the final step of the proof, denote $\mathcal{S}= \sum_{M \in \M\setminus \T(\bkappa)} \E\left(NC(M)\right)$ for convenience. Since~\eqref{eq:fyv2} holds for any $\bbeta^*\in\bB$, we have
\begin{align*}
    \sup_{\bbeta^*\in\bB }\mathcal{S} &\leq \sum_{M \in \M\setminus \T(\bkappa)}e^{- \sum_{j=1}^b |M_j\setminus T^M_j|\big(\ln(p_j-s_j)+\delta \tfrac{f_j}{2}\big)-\sum_{j=1}^b |T^M_j\setminus M_j|\big(\ln(s_j)+\delta \tfrac{g'_j}{2}\big)}.
\end{align*}
We split the sum in the right-hand side above into sums over models $M$
such that $T^M=T$ for some common $T\in \T(\bkappa)$. Denote for any $T\in \T$, $\M(T):=\{M \in \M\setminus \T(\bkappa) \,| T^M=T\}$, then
\begin{align}\label{eq:gepihw}
    \sup_{\bbeta^*\in\bB }\mathcal{S} &\leq \sum_{T \in \T(\bkappa)}\sum_{M \in \M(T)}e^{- \sum_{j=1}^b |M_j\setminus T_j|\big(\ln(p_j-s_j)+\delta \tfrac{f_j}{2}\big)-\sum_{j=1}^b |T_j\setminus M_j|\big(\ln(s_j)+\delta \tfrac{g'_j}{2}\big)}.
\end{align}
The right hand-side of~\eqref{eq:gepihw} is composed of a double sum over $T \in \T(\bkappa)$ and over $M \in \M(T)$. Consider the sum over $M \in \M(T)$, add $T$ to it, and denote it
\begin{equation}\label{eq:defST}
  \mathcal{S}(T)=\sum_{M \in \M(T) \cup T}e^{- \sum_{j=1}^b |M_j\setminus T_j|\big(\ln(p_j-s_j)+\delta \tfrac{f_j}{2}\big)-\sum_{j=1}^b |T_j\setminus M_j|\big(\ln(s_j)+\delta \tfrac{g'_j}{2}\big)}.
\end{equation}
In the summand in the right-hand side of~\eqref{eq:defST}, the case $M=T$ correspond to $|M_j\setminus T^M_j|=|T^M_j\setminus M_j|=0$ for all $j$ and the summand is then 1.
By~\eqref{eq:gepihw}, we then get that
\begin{align}\label{eq:gepihw2}
    \sup_{\bbeta^*\in\bB}\mathcal{S} &\leq \sum_{T \in \T(\bkappa)}\big(\mathcal{S}(T)-1\big).
\end{align}
For each $T=T^M$, we further split $\mathcal{S}(T)$ into sums over subsets of models $M$ that have $u_j$ more parameters than $T^M$ in block $j$, and are missing $w_j$ parameters from $T^M$ in block $j$.
Specifically, consider models $M$ such that, for all $j$, $|M_j\setminus T^M_j|=u_j$ and $|T^M_j\setminus M_j|=w_j$ with $u_j \in \{0,\ldots,p_j-s_j,\ldots,p_j-s_j+|S_j^{S}(\kappa_j)|\}$ and $w_j \in \{0,\ldots,|S_j^{L}(\kappa_j)|\}$. Denote by
$$S^{\bu}_{\bw}(T)=\sum_{M \in \M(T) \cup T:\forall \,j |M_j\setminus S_j|=u_j, |S_j\setminus M_j|=w_j } e^{- \sum_{j=1}^b u_j\big(\ln(p_j-s_j)+\delta \tfrac{f_j}{2}\big)-\sum_{j=1}^b w_j\big(\ln(s_j)+\delta \tfrac{g'_j}{2}\big)}.$$
We get
\begin{equation}\label{eq:uygfew2}
    \mathcal{S}(T)=\sum_{w_1=0}^{|S^L_1(\kappa)|}\cdots\sum_{w_b=0}^{|S^L_b(\kappa)|}\sum_{u_1=0}^{p_1-s_1+|S_1^S|}\cdots\sum_{u_b=0}^{p_b-s_b+|S_b^S|}S^{\bu}_{\bw}(T).
\end{equation}
The number of models missing, for all $j$, $w_j$ out of the $|S^L_j(\kappa)|$ large active parameters and having $u_j$ inactive or small active parameters from $B_j$ is $\prod_{j=1}^b {\binom{p_j-s_j+|S^S_j(\kappa)|}{u_j}}{\binom{|S^L_j(\kappa)|}{w_j}}$. We thus have
\begin{align*}
    S^{\bu}_{\bw}(T)&= \bigg(\prod_{j=1}^b {\binom{p_j-s_j+|S^S_j(\kappa)|}{u_j}}{\binom{|S^L_j(\kappa)|}{w_j}}\bigg)e^{- \sum_{j=1}^b u_j\big(\ln(p_j-s_j)+\delta \tfrac{f_j}{2}\big)-\sum_{j=1}^b w_j\big(\ln(s_j)+\delta \tfrac{g'_j}{2}\big)} \\
    &= \prod_{j=1}^b {\binom{p_j-s_j+|S^S_j(\kappa)|}{u_j}}e^{- u_j\big(\ln(p_j-s_j)+\delta \tfrac{f_j}{2}\big)}{\binom{|S^L_j(\kappa)|}{w_j}}e^{-w_j\big(\ln(s_j)+\delta \tfrac{g'_j}{2}\big)}.
\end{align*}
Inputting the expression above in~\eqref{eq:uygfew2} and factorizing over terms in $u_j$ and $w_j$ gives
\begin{align*}
    \mathcal{S}(T)\leq\prod_{j=1}^b & \Bigg(\sum_{u_j=0}^{p_j-s_j+|S^S_j(\kappa)|}{\binom{p_j-s_j+|S^S_j(\kappa)|}{u_j}}e^{- u_j(\ln(p_j-s_j)+\delta \tfrac{f_j}{2})}\Bigg)\\&.\Bigg(
    \sum_{w_j=0}^{|S^L_j(\kappa)|}{\binom{|S^L_j(\kappa)|}{w_j}}
    e^{-w_j(\ln(s_j)+\delta \tfrac{g'_j}{2})}\Bigg).
\end{align*}
By the bound in~\eqref{eq:upperboudbinom} and taking the terms in $u_j=0$ and $w_j=0$ out of the sums above, 
we have
\begin{align}\label{eq:ogr}
    \mathcal{S}(T)\leq\prod_{j=1}^b & \left(1+\sum_{u_j=1}^{p_j-s_j+|S_j^{S}(\kappa_j)|}e^{-u_j\left(\delta \tfrac{f_j}{2}-\ln\big(1+\frac{|S_j^{S}(\kappa_j)|}{p_j-s_j}\big)-1\right)}\right)\\&.\left(1+
    \sum_{w_j=1}^{|S_j^{L}(\kappa_j)|}    e^{-w_j\left(\delta \tfrac{g'_j}{2}+\ln\big(\frac{s_j}{|S_j^{L}(\kappa_j)|}\big)-1\right)}\right).\nonumber
\end{align}
Denote
$$
d_j\;=\;e^{1+\ln\big(1+\frac{|S_j^{S}(\kappa_j)|}{p_j-s_j}\big)-\delta \tfrac{f_j}{2}},\qquad h_j\;=\;e^{1-\ln\big(\frac{s_j}{|S_j^{L}(\kappa_j)|}\big)-\delta \tfrac{g'_j}{2}}.
$$
By assumption $|S_j^{S}(\kappa_j)|=O(p_j-s_j)$, and by definition of $|S_j^{L}(\kappa_j)|$ we have $s_j\geq|S_j^{L}(\kappa_j)|$. By Assumptions~\hyperlink{lab:A1}{A1} and the definition of $|S_j^{L}(\kappa_j)$, we also have $f_j\to\infty$ and $g'_j\to\infty$, and then $\lim_{n \to \infty}= d_j= \lim_{n \to \infty} h_j=0$. Using the properties of geometric series, 
 for every $j$ we have
\begin{eqnarray*}
    &1+\sum_{u_j=1}^{p_j-s_j+|S_j^{S}(\kappa_j)|}e^{-u_j\left(\delta \tfrac{f_j}{2}-\ln\big(1+\frac{|S_j^{S}(\kappa_j)|}{p_j-s_j}\big)-1\right)}
    =\frac{1-d_j^{p_j-s_j+|S_j^{S}(\kappa_j)|+1}}{1-d_j} \\
    &1+\sum_{w_j=1}^{|S_j^{L}(\kappa_j)|}    e^{-w_j\left(\delta \tfrac{g'_j}{2}+\ln\big(\frac{s_j}{|S_j^{L}(\kappa_j)|}\big)-1\right)} =\frac{1-h_j^{|S_j^{L}(\kappa_j)|+1}}{1-h_j},
\end{eqnarray*}
where both expressions converge to $1$ as $n$ grows. By~\eqref{eq:gepihw2} and~\eqref{eq:ogr}:
$$\sup_{\bbeta^*\in\bB}\mathcal{S}\leq\sum_{T\in\T}\Bigg(\prod_{j=1}^b\Big(\frac{1-d_j^{p_j-s_j+|S_j^{S}(\kappa_j)|+1}}{1-d_j}\Big)\Big(\frac{1-h_j^{|S_j^{L}(\kappa_j)|+1}}{1-h_j}\Big)-1\Bigg).$$
Each of the summand vanishes as $n\to\infty$. Moreover, by assumption $|S^I(\kappa)|=O(1)$ and then $|\T(\bkappa)|=2^{|S^I(\kappa)|}=O(1)$. We thus have $\lim_{n\to\infty}\sup_{\bbeta^*\in\bB}\mathcal{S}=\lim_{n\to\infty}\sup_{\bbeta^*\in\bB}\sum_{M \in \M\setminus \T(\bkappa)} \E\left(NC(M)\right)=0$.

Further, by Lemma~\ref{lem:L0toL1convergencegen}, 
$\sup_{\bbeta^*\in\bB}P(\hat{S}^{ei} \not\in \T(\bkappa) ) \,\,\leq \,\, (|\T(\bkappa)|+1) \sup_{\bbeta^*\in\bB}\mathcal{S} = (2^{|S^{I}(\bkappa)|}+1) \sup_{\bbeta^*\in\bB}\mathcal{S}$. Since $\sup_{\bbeta^*\in\bB}\mathcal{S}$ vanishes and $|S^{I}(\bkappa)|=O(1)$, $\lim_{n\to\infty}\inf_{\bbeta^*\in\bB}P(\hat{S}^{ei} \in \T(\bkappa) )=1$ as claimed.

\subsection{Proofs of Section~\ref{sec:informed_l0}}\label{suppsec:proofsempbayes}

\subsubsection{Proof of Proposition~\ref{prop:shatconsistence}}
Denote
$$A_j:=\frac{1}{p_j}\sum_{i \in B_j}\sum_{M \in \T(\bkappa)|i\in M}NC(M) \quad\text{and}\quad  C_j:=\frac{1}{p_j}\sum_{i \in B_j}\sum_{M \in \M\setminus \T(\bkappa)|i\in M}NC(M).$$
For every $j=1,\ldots,b$, we have the decomposition
\begin{gather}\label{eq:decompshat}
    \frac{\hat{s}_j}{p_j}
    =\frac{\sum_{i \in B_j} \sum_{M \in \M | i \in M} NC(M)}{p_j}
   =A_j\,+\, C_j.
\end{gather}
To show the lower bound on $\hat{s}_j/p_j$, we decompose $A_j$
\begin{eqnarray}
  A_j
   &=& \sum_{M \in \T(\bkappa)}NC(M)\sum_{i \in B_j}\frac{I(i\in M_j)}{p_j}\nonumber\\
   &=& \sum_{M \in \T(\bkappa)}NC(M)\sum_{i \in S_j^{L}(\kappa_j)}\frac{I(i\in M_j)}{p_j}+\sum_{M \in \T(\bkappa)}NC(M)\sum_{i \in B_j\setminus S_j^{L}(\kappa_j)}\frac{I(i\in M_j)}{p_j}\nonumber\\
   &=& \frac{|S_j^{L}(\kappa_j)|}{p_j}\sum_{M \in \T(\bkappa)}NC(M)+\sum_{M \in \T(\bkappa)}NC(M)\sum_{i \in B_j\setminus S_j^{L}(\kappa_j)}\frac{I(i\in M_j)}{p_j}\label{eq:wigr}.
\end{eqnarray}
where the last equality follows from $I(i\in M_j)=1$ for all $i\in S^L(\bkappa)$ when $M\in\T(\bkappa)$. The rightmost term above and $C_j$ are nonnegative, then by the linearity of the expectation
$$\E\Big(\frac{\widehat{s_j}}{p_j}\Big)\geq\frac{|S_j^{L}(\kappa_j)|}{p_j}\sum_{M \in \T(\bkappa)}\E(NC(M)).$$
By Theorem~\ref{thm:convtoT}, $\lim_{n\to\infty}\sum_{M \in \T(\bkappa)}\E(NC(M))=1$. It follows that $\lim_{n\to\infty}\E\Big(\frac{\widehat{s_j}}{p_j}\Big)\geq \frac{|S_j^{L}(\kappa_j)|}{p_j}$ for every $j=1,\ldots,b$.

We now prove the upper bound. Recall that $\T(\bkappa)$ by definition includes models that have no small signals, i.e. all parameters are in $S^L(\bkappa) \cup S^I(\kappa)$. That is,
for all $M\in\T(\bkappa)$, we have that $I(i\in M_j)=0$ for all $i\in B_j\setminus(S_j^{L}(\kappa_j)\cup S_j^{I}(\kappa_j))$. Hence, $A_j$ in~\eqref{eq:wigr} satisfies
\begin{align*}
    A_j
   &= \frac{|S_j^{L}(\kappa_j)|}{p_j}\sum_{M \in \T(\bkappa)}NC(M)+\sum_{M \in \T(\bkappa)}NC(M)\sum_{i \in S_j^{I}(\kappa_j)}\frac{I(i\in M_j)}{p_j}\\
   &\leq\frac{|S_j^{L}(\kappa_j)|}{p_j}\sum_{M \in \T(\bkappa)}NC(M)+ \frac{|S_j^{I}(\kappa_j)|}{p_j}\sum_{M \in \T(\bkappa)}NC(M)
\end{align*}
where the inequality follows from $\sum_{i \in S_j^{I}(\kappa_j)} I(i\in M_j) \leq |S_j^{I}(\kappa_j)|$  for all $M$. By~\eqref{eq:decompshat}, we then have
\begin{equation}\label{eq:uuuuu}
    \frac{\hat{s}_j}{p_j}\leq \frac{|S_j^{L}(\kappa_j)|+|S_j^{I}(\kappa_j)|}{p_j}\sum_{M \in \T(\bkappa)}NC(M) + C_j.
\end{equation}
Moreover, for every $j=1,\ldots,b$, $C_j$ satisfies
\begin{equation}\label{eq:uuuuu2}
    C_j = \sum_{M \in \M\setminus \T(\bkappa)}NC(M)\sum_{i \in B_j}\frac{I(i\in M_j)}{p_j}\leq \sum_{M \in \M\setminus \T(\bkappa)}NC(M)
\end{equation}
where the inequality follows from $\sum_{i \in B_j}\frac{I(i\in M_j)}{p_j}\leq 1$ for all $M$.
Taking expectations in~\eqref{eq:uuuuu} and~\eqref{eq:uuuuu2} gives
\begin{equation*}
    \E\bigg(\frac{\widehat{s_j}}{p_j}\bigg) 
    \leq \frac{|S_j^{L}(\kappa_j)|+|S_j^{I}(\kappa_j)|}{p_j}\sum_{M \in \T(\bkappa)}\E(NC(M)) + \sum_{M \in \M\setminus \T(\bkappa)}\E(NC(M)).
\end{equation*}
 By Theorem~\ref{thm:convtoT}, we have on one hand $\lim_{n\to \infty}\sum_{M \in \T(\bkappa)}\E(NC(M))= 1$ and, on the other hand, $\lim_{n\to \infty}\sum_{M \in \M\setminus\T(\bkappa)}\E(NC(M))= 0$. It follows that $ \lim_{n\to \infty}\E\big(\frac{\widehat{s_j}}{p_j}\big) \leq  \frac{|S_j^{L}(\kappa_j)|+|S_j^{I}(\kappa_j)|}{p_j}= \frac{s_j-|S_j^{S}(\kappa_j)|}{p_j}$ for every $j=1,\ldots,b$, which proves the upper bound.

 \subsubsection{Proof of Lemma~\ref{lem:empiricalbayesinterpret}}

Denote the marginal likelihood of $\by$ given $\btheta$ by $H(\btheta)=p(\by \mid \btheta)$. For every $j=1,\ldots,b$, the partial derivative of its logarithm with respect to $\theta_j$ is
    \begin{equation}\label{eq:logderivative}
        \frac{\partial \ln H(\btheta)}{\partial \theta_j}=\frac{\partial H(\btheta)}{\partial \theta_j}\,H(\btheta)^{-1}.
    \end{equation}
    Observe that:
    \begin{equation}\label{eq:partialmarglikelihood}
    H(\btheta)=\sum_{M \in \M}p(\by \mid M,\btheta)p(M \mid \btheta)
    \quad \text{and}\quad 
    \frac{\partial H(\btheta)}{\partial \theta_j}=\sum_{M\in \M}p(\by \mid M,\btheta)\frac{\partial p(M \mid \btheta)}{\partial \theta_j}
    \end{equation}
    Recall that each model is defined as $M=(m_1,\ldots,m_p)$ where $m_i = I(\beta_i \neq 0)$ indicates whether variable $j$ is included under $M$, and that our choice of model prior in~\eqref{eq:block_modelprior} is 
    \begin{align}
        p( M \mid \btheta)=
        \prod_{j=1}^b \big(\theta_j\big)^{\sum_{i \in B_j} m_i}
        \big(1 - \theta_j\big)^{p_j - \sum_{i \in B_j} m_i}.
        \nonumber
    \end{align}
    Hence, simple algebra shows that for every $M\in \mathcal M$
    \begin{equation}\label{eq:partialprobmodel}
    \frac{\partial p(M \mid \btheta)}{\partial \theta_j}= p(M \mid \btheta) \bigg(\frac{\sum_{i\in B_j}m_i}{\theta_j}-\frac{p_j-\sum_{i\in B_j}m_i}{1-\theta_j} \bigg).       
    \end{equation}
    
    Replacing~\eqref{eq:partialprobmodel} into~\eqref{eq:partialmarglikelihood}, and using that for any function $f$ $$\sum_{M\in \M}\sum_{i\in B_j}m_if(M)=\sum_{i\in B_j}\sum_{M\in \M:m_i=1}f(M),$$ we get
    \begin{align}\label{eq:partialmarglikelihood2}
        \frac{\partial H(\btheta)}{\partial \theta_j}&\;=\;\frac{1}{\theta_j}\sum_{i \in B_j}\sum_{M\in \M: m_i=1}p(\by \mid M,\btheta)p(M \mid \btheta) \\&\quad- \frac{1}{1-\theta_j}\bigg[p_j H(\btheta)-\sum_{i \in B_j}\sum_{M\in \M: m_i=1}p(\by \mid M,\btheta)p(M \mid \btheta). \nonumber
        \bigg]
    \end{align}
    Note that
    \begin{align*}
        \frac{\sum_{M\in \M: m_i=1}p(\by \mid M,\btheta)p(M \mid \btheta)}{H(\btheta)} = \frac{\sum_{M\in \M: m_i=1}p(\by , M\mid \btheta)}{p(\by \mid \btheta)}= P(m_i = 1\mid \by,\btheta)
    \end{align*}
    By~\eqref{eq:logderivative}, we then get the following expression of the partial derivative, for every $j=1,\ldots,b$,
    $$\frac{\partial \ln H(\btheta)}{\partial \theta_j}=\frac{1}{\theta_j}\sum_{i \in B_j}P(m_i = 1\mid \by,\btheta) - \frac{1}{1-\theta_j}\bigg[p_j-\sum_{i \in B_j}P(m_i = 1\mid \by,\btheta)
        \bigg]$$
    Setting the partial derivatives to 0 and solving for $\theta_j$ gives the desired result.

\subsubsection{Proof of Theorem~\ref{theo:informedl0consist}} The proof strategy is to show Assumptions~\hyperlink{lab:A1}{A1} and \hyperlink{lab:A2}{A2} hold to apply Theorem~\ref{theo:suffcondlinearmodel}. We first derive a convenient decomposition of the penalties $\kappa^A_j$. The second step is to show that, under Assumption~\hyperlink{lab:A3}{A3}, $\kappa^{A}_j$ satisfies Assumption~\hyperlink{lab:A1}{A1} with probability going to 1 as $n$ grows. The third step is to show that Assumption~\hyperlink{lab:A4}{A4} implies Assumption~\hyperlink{lab:A2}{A2} for the $\kappa^{A}_j$ with probability going to 1. The consistency of $\hat S^{A,ei}$ then follows from Theorem~\ref{theo:suffcondlinearmodel}.

Denote for any $M\in\M$, $NC^{\circ}(M)$, the normalized criterion value for model $M$ under Step 1 penalty $\kappa^{\circ}$. For this choice of penalty and every $j=1,\ldots,b$, we have
\begin{align*}
    \frac{\widehat{s_j}}{p_j}
    &\;=\; \frac{1}{p_j}\sum_{i\in B_j}\,\,\sum_{M\in \M}NC^{\circ}(M)I(i\in M)\\
   &\;=\; \frac{1}{p_j}\sum_{M \in \M}NC^{\circ}(M)\sum_{i \in B_j}I(i\in M) \\
   &\;=\;\frac{s_j}{p_j}NC^{\circ}(S) + \sum_{M \in \M | M \neq S}\frac{|M_j|}{p_j}NC^{\circ}(M).
\end{align*}
Using that $NC^{\circ}(S)=1-\sum_{M \in \M | M \neq S}NC^{\circ}(M)$, we get
\begin{equation*}
   \frac{\hat{s}_j}{p_j}
   =\frac{s_j}{p_j} + \sum_{M \in \M | M \neq S}\frac{|M_j|-s_j}{p_j}NC^{\circ}(M).
\end{equation*}
Consider the decomposition of the sum in the right-hand side above between the sum over models $M$ that contain more parameters than $S$ in block $j$ and the sum over those that contain fewer parameters than $S$ in block $b$. Denote
$$
    O_j^{\circ} := \sum_{M \in \M | |M_j| > s_j}\frac{|M_j|-s_j}{p_j}NC^{\circ}(M)\quad\text{and}\quad
    U_j^{\circ} := \sum_{M \in \M | |M_j| < s_j}\frac{s_j-|M_j|}{p_j}NC^{\circ}(M).
$$
We have
\begin{equation}\label{eq:rewriteshat}
  \frac{\hat{s}_j}{p_j}  = \frac{s_j}{p_j}\,\,+\,\, O_j^{\circ}-\,\,U_j^{\circ}.
\end{equation}
Observe that we have the following decomposition of Step 2 penalties
$$\kappa^{A}_j
    = \ln(p_j-s_j) + \ln(\sqrt{n}) + \ln\Big(\frac{p_j-\hat{s}_j}{p_j-s_j}\Big). \nonumber$$
By~\eqref{eq:rewriteshat}, it follows that
\begin{eqnarray}\label{eq:decompstep2pen}
    \kappa^{A}_j&=& \ln(p_j-s_j)+\ln(\sqrt{n})+\ln\Big(1-\frac{p_j(O^{\circ}_j-U^{\circ}_j)}{p_j-s_j}\Big),
\end{eqnarray}
completing the first step of the proof.

We continue with the second step of the proof: showing that the $\kappa^{A}_j$'s satisfy Assumption~\hyperlink{lab:A1}{A1} with probability going to 1. Recall that Assumption~\hyperlink{lab:A1}{A1} states that there exists $f_j\to \infty$ (as $n\to \infty$) such that for every sufficiently large $n$,
\begin{equation*}
    \kappa_j \;=\;\ln(p_j-s_j) + f_j.
\end{equation*}
Since $U_j^{\circ}$ is nonnegative, a lower bound on $\kappa^{A}_j$ is
\begin{equation}\label{eq:decompstep2pen2}
    \kappa^{A}_j
    \geq  \ln(p_j-s_j)+\ln(\sqrt{n})+\ln\Big(1-\frac{p_j O^{\circ}_j}{p_j-s_j}\Big).
\end{equation}

Plugging in the definition of $O^{\circ}_j$, we have that
$$\frac{p_j O^{\circ}_j}{p_j-s_j} =  \sum_{M \in \M | |M_j| > s_j}\frac{|M_j|-s_j}{p_j-s_j}NC^{\circ}(M)\leq \sum_{M \in \M | |M_j| > s_j}NC^{\circ}(M)$$
where the inequality follows from $(|M_j|-s_j)/(p_j-s_j)\leq 1$ for all $M$.
Note that if $M$ is such that $|M_j|>s_j$, then $M\not\in\T(\kappa^{\circ})$ (this follows immediately from the definition of $\T(\bkappa)$ in~\eqref{eq:Tkappa}) and therefore $\sum_{M \in \M | |M_j| > s_j}NC^{\circ}(M)\leq \sum_{M \in \M \setminus \T(\kappa^{\circ})}NC^{\circ}(M)$. Moreover, $\kappa^{\circ}$ satisfies Assumption~\hyperlink{lab:A1}{A1}, Assumption~\hyperlink{lab:A3}{A3} is assumed to hold and the assumptions of Theorem~\ref{thm:convtoT} are met for $\kappa^{\circ}$.
Then, by Theorem~\ref{thm:convtoT}, $\lim_{n\to\infty} \sum_{M \in \M \setminus \T(\kappa^{\circ})}\E(NC^{\circ}(M))\lim_{n\to\infty}\sum_{M \in \M | |M_j| > s_j}\E(NC^{\circ}(M)) = 0$. It follows that $\frac{p_j O^{\circ}_j}{p_j-s_j}$ vanishes in probability and so does $\ln\Big(1-\frac{p_j O^{\circ}_j}{p_j-s_j}\Big)$.

With probability going to 1, we then have that
\begin{equation}\label{eq:finallowerboundkappaj2}\kappa^{A}_j 
    \geq  \ln(p_j-s_j)+\ln(\sqrt{n})
\end{equation}
and hence that the $\kappa^{A}_j$'s satisfy Assumption~\hyperlink{lab:A1}{A1}.

For the third part of the proof,
we now show that Assumption~\hyperlink{lab:A4}{A4} implies Assumption~\hyperlink{lab:A2}{A2} for the $\kappa^{A}_j$. Assumption~\hyperlink{lab:A2}{A2} for the $\kappa_j^{A}$ states that for each block $j$ there exists $g_j\to \infty$ such that for large enough $n$,
    \begin{equation*}
        \sqrt{\frac{(1-\gamma) n\rho(\bX)}{6}}{{\beta_{\min,j}^*}} \,- \,\sqrt{\kappa^{A}_j} \;=\; \sqrt{\ln(s_j)} + g_j .
    \end{equation*}
where $\gamma$ takes value
\begin{equation}\label{eq:gammaka}
    \gamma=\frac12\Big(1+\max_j\frac{\ln(p_j-s_j)}{\kappa^{A}_j}\Big).
\end{equation}
To show that Assumption~\hyperlink{lab:A4}{A4} implies Assumption~\hyperlink{lab:A2}{A2} for the $\kappa_j^{A}$ with probability going to 1, it suffices to show that the following two inequalities
\begin{align}
    \sqrt{\frac{ (1-\gamma)n\rho(\bX)}{6}}{\beta_{\min,j}^*} \,&\geq \,  \sqrt{\frac{ (1-\xi)n\rho(\bX)}{6}}{\beta_{\min,j}^*}, \label{eq:oaebr3}\;\;\text{and} \\
    - \,\sqrt{\kappa^{A}_j} \,&\geq \,-\sqrt{\ln\bigg(p-|S^{L}(\kappa^{\circ})|\bigg)+\frac12\ln(n)}\label{eq:oaebr4}
\end{align}
hold with probability going to 1 for $\gamma$ as in~\eqref{eq:gammaka} and $\xi=\frac12\big(1+\max_j \frac{\ln(p_j-s_j)}{\ln(p_j-s_j)+0.5\ln(n)}\big)$ (defined in Assumption~\hyperlink{lab:A4}{A4}). We first show~\eqref{eq:oaebr3} holds with probability going to 1 and then that~\eqref{eq:oaebr4} does too.

By~\eqref{eq:finallowerboundkappaj2} we have that with probability going to 1, for any $j$
\begin{equation*}
\frac{\ln(p_j-s_j)}{\kappa^{A}_j} 
    \leq  \frac{\ln(p_j-s_j)}{\ln(p_j-s_j)+\ln(n)/2}.
\end{equation*}
It follows that
\begin{equation*}
    \gamma=\frac12\Big(1+\max_{j=1, \ldots,b}\frac{\ln(p_j-s_j)}{\kappa^ A_j}\Big) \leq \frac12\big(1+\max_j \frac{\ln(p_j-s_j)}{\ln(p_j-s_j)+\ln(n)/2}\big)=\xi
\end{equation*}
and~\eqref{eq:oaebr3} holds with probability going to 1.

We now upper bound $\kappa^{A}_j$ to show~\eqref{eq:oaebr4} holds with probability going to 1. By~\eqref{eq:decompstep2pen}, we can write
    $$\kappa_j^{A}=\kappa_j^{EB}+\ln(\widehat s_j)= \ln(p_j-s_j)+\ln\big(\sqrt{n}\big)+\ln\Big(1-\frac{p_j(O^{\circ}_j-U^{\circ}_j)}{p_j-s_j}\Big).$$
Since $O_j^{\circ} \geq 0$, we obtain that 
\begin{equation*}
    \kappa^{A}_j \leq \ln\big(p_j-s_j\big)+\ln(\sqrt{n})+\ln\bigg(1+\frac{p_j}{p_j-s_j}U^{\circ}_j\bigg).
\end{equation*}
We split the sum in $U^{\circ}_j$ between models in $\T(\kappa^{\circ})$ and those not in $\T(\kappa^{\circ})$. 
$$U_j^{\circ} = \sum_{M \in \T(\kappa^{\circ}) | |M_j| < s_j}\frac{s_j-|M_j|}{p_j}NC^{\circ}(M)+\sum_{M \in \M\setminus\T(\kappa^{\circ}) | |M_j| < s_j}\frac{s_j-|M_j|}{p_j}NC^{\circ}(M).
$$
If $M\in \T(\kappa^{\circ})$, then by definition $|M_j|\geq |S^{L}_j(\kappa^{\circ})|$ and thus $s_j-|M_j|\leq s_j-|S^{L}_j(\kappa^{\circ})|$. A bound on $s_j-|M_j|$ for $M\not\in \T(\kappa^{\circ})$ is simply $s_j-|M_j|\leq s_j$. It follows that
\begin{eqnarray*}
    U^{\circ}_j \leq \frac{s_j-|S^{L}_j(\kappa^{\circ})|}{p_j}\sum_{M \in \T(\kappa^{\circ}) | |M_j| < s_j}NC^{\circ}(M)+  \frac{s_j}{p_j}\sum_{M \in \M \setminus \T(\kappa^{\circ})| |M_j| < s_j}NC^{\circ}(M)
\end{eqnarray*}
By Theorem~\ref{thm:convtoT}, $\sum_{M \in  \T(\kappa^{\circ})| |M_j| < s_j}NC^{\circ}(M)$ and $\sum_{M \in \M \setminus \T(\kappa^{\circ})| |M_j| < s_j}NC^{\circ}(M)$ converge in probability to 1 and 0 respectively. We then get that, with probability going to 1,
\begin{align}
    \frac{p_j}{p_j-s_j}U^{\circ}_j &\leq \frac{s_j-|S^{L}_j(\kappa^{\circ})|}{p_j-s_j} \label{eq:gewou}
\end{align}
By~\eqref{eq:gewou}, with probability going to 1,
\begin{align}\label{eq:finalupperboundkappaj2}
    \kappa_j^A &\leq  \ln\big(p_j-s_j\big)+\ln(\sqrt{n})+\ln\bigg(1+\frac{s_j-|S^{L}_j(\kappa^{\circ})|}{p_j-s_j}\bigg)\nonumber\\
    &= \ln\big(p_j-|S^{L}_j(\kappa^{\circ})|\big)+\frac{1}{2}\ln(n).
\end{align}
which shows~\eqref{eq:oaebr4} holds with probability going to 1 and that
Assumption~\hyperlink{lab:A4}{A4} implies Assumption~\hyperlink{lab:A2}{A2} holds for the $\kappa^{A}_j$ with probability going to 1. 

Since Assumptions~\hyperlink{lab:A1}{A1} and \hyperlink{lab:A2}{A2} hold with probability going to 1, by Theorem~\ref{theo:suffcondlinearmodel}, $\inf_{\bbeta^*\in\bB}\lim_{n\to\infty}P(\hat{S}^{A,ei}= S)= 1$, as claimed.

\subsection{Proof of Corollary~\ref{cor:TLconsist}}

The result follows directly from Theorem~\ref{theo:informedl0consist}, given that $\kappa^{\circ}$ is assumed to satisfy Assumption~\hyperlink{lab:A3}{A3} and that
Assumption~\hyperlink{lab:A4}{A4} holds under Assumption~\hyperlink{lab:A5}{A5} .

\end{document}